\documentclass{amsart}
\usepackage{amssymb,stmaryrd,faktor,tikz-cd,float}
\usepackage[shortlabels]{enumitem}
\providecommand*\texorpdfstring[2]{#1}

\newtheorem*{thma}{Theorem A}
\newtheorem*{thmb}{Theorem B}
\newtheorem*{thmc}{Theorem C}
\newtheorem*{thmd}{Theorem D}
\newtheorem{thm}{Theorem}[section]
\newtheorem{lemma}[thm]{Lemma}
\newtheorem{cor}[thm]{Corollary}
\newtheorem{claim}{Claim}[thm]
\newtheorem{fact}[thm]{Fact}
\newtheorem{prop}[thm]{Proposition}
\newtheorem{conj}[thm]{Conjecture}

\theoremstyle{definition}
\newtheorem{defn}[thm]{Definition}
\newtheorem{conv}[thm]{Convention}

\theoremstyle{remark}
\newtheorem{remark}[thm]{Remark}

\DeclareMathOperator{\pr}{Pr}
\DeclareMathOperator{\ssup}{ssup}
\DeclareMathOperator{\crit}{crit}
\DeclareMathOperator{\refl}{Refl}
\DeclareMathOperator{\pl}{P\ell}
\DeclareMathOperator{\tcf}{tcf}

\DeclareMathOperator{\ubd}{{\sf unbounded}}
\DeclareMathOperator{\onto}{{\sf onto}}
\DeclareMathOperator{\non}{non}
\DeclareMathOperator{\reg}{Reg}
\DeclareMathOperator{\cf}{cf}
\DeclareMathOperator{\dom}{dom}
\DeclareMathOperator{\im}{Im}
\DeclareMathOperator{\otp}{otp}
\DeclareMathOperator{\acc}{acc}
\DeclareMathOperator{\nacc}{nacc}
\DeclareMathOperator{\Tr}{Tr}
\DeclareMathOperator{\tr}{tr}
\DeclareMathOperator{\p}{P}
\DeclareMathOperator{\stat}{Stat}
\DeclareMathOperator{\U}{U}
\DeclareMathOperator{\cg}{CG}

\renewcommand\restriction{\mathbin\upharpoonright}
\renewcommand\mid{\mathrel{|}\allowbreak}
\renewcommand\L{\mathrm{L}}
\newcommand\V{\mathrm{V}}
\newcommand\stick{{{\ensuremath \mspace{2mu}\mid\mspace{-12mu} {\raise0.6em\hbox{$\bullet$}}}}}
\newcommand\diagonal{\bigtriangleup}
\newcommand\s{\subseteq}
\newcommand\sq{\sqsubseteq}
\newcommand\br{\blacktriangleright}
\newcommand\symdiff{\mathbin\triangle}
\newcommand\tree{\mathbf{T}}
\newcommand\ns{\textup{NS}}
\newcommand\sa{\textup{SA}}
\newcommand\amen{\textup{A}}
\newcommand\bd{\textup{bd}}

\newcommand*\axiomfont[1]{\textsf{\textup{#1}}}
\newcommand\zfc{\axiomfont{ZFC}}

\newcommand\bpfa{\axiomfont{BPFA}}

\subjclass[2010]{Primary 03E02; Secondary 03E35, 03E55}

\setlist[enumerate,1]{label={(\roman*)}}
\newenvironment{why}[1][Proof]{\proof[#1]\mbox{}}{\endproof}

\title[Was Ulam right? I]{Was Ulam right? I:\\ Basic theory and subnormal ideals}
\date{Preprint as of December 9, 2021. For the latest version, visit \textsf{http://p.assafrinot.com/47}.}

\author{Tanmay Inamdar}
\address{Department of Mathematics, Bar-Ilan University, Ramat-Gan 5290002, Israel.}
\author{Assaf Rinot}
\address{Department of Mathematics, Bar-Ilan University, Ramat-Gan 5290002, Israel.}
\urladdr{http://www.assafrinot.com}

\begin{document}
\begin{abstract} We introduce various colouring principles which generalise the so-called \emph{onto mapping principle} of Sierpi\'nski to larger cardinals and general ideals. We prove that these principles capture the notion of an Ulam matrix and allow to characterise large cardinals, most notably weakly compact and ineffable cardinals. We also develop the basic theory of these colouring principles, connecting them to the classical negative square bracket partition relations, proving pumping-up theorems, and deciding various instances of theirs. 
We also demonstrate that our principles provide a uniform way of obtaining non-saturation results for ideals satisfying a property we call \emph{subnormality} in contexts where Ulam matrices might not be available.
\end{abstract}
\maketitle
\tableofcontents
	
\section{Introduction}
Throughout this paper, $\kappa$ denotes an infinite cardinal
and $\theta$ denotes a cardinal with $2\le\theta\le\kappa$. Some additional conventions are listed in Subsection~\ref{nandc} below.

The set-theoretic study of strong colourings traditionally centres around the study of all pairs of cardinals $(\kappa,\theta)$ for which the partition relation $\kappa\nrightarrow[\kappa]^2_\theta$ holds:
\begin{defn}[{\cite[\S18]{MR202613}}] \label{hungarian} $\kappa\nrightarrow[\kappa]^2_\theta$ asserts the existence of a colouring $c:[\kappa]^2\rightarrow\theta$ such that, for every $B\in[\kappa]^\kappa$, $c``[B]^2=\theta$.
\end{defn}

The most prominent open questions in this vein are:
\begin{enumerate}[(1)]
\item Does $\lambda^+\nrightarrow[\lambda^+]^2_\lambda$ hold for every singular cardinal $\lambda$?
\item How large must an inaccessible cardinal $\kappa$ be for $\kappa\nrightarrow[\kappa]^2_\kappa$ to not hold?
\item  Does  $\kappa\nrightarrow[\kappa]^2_\omega$ hold for every regular $\kappa>\aleph_0$ that is not weakly compact?\footnote{Throughout this article we adopt the convention that $\aleph_0$ is not weakly compact.}
\end{enumerate}
	
Since these questions have been outstanding for very long now, it is natural to identify closely-related questions that may be more approachable yet have a similar web of applications.
For this, we formulate here a family of colouring principles along these lines and study their validity. Our first example is the following.
\begin{defn} \label{defnubd} For an ideal $J$ over $\kappa$, $\ubd(J,\theta)$ 
asserts
the existence of a colouring $c:[\kappa]^2\rightarrow\theta$ that is upper-regressive 
(i.e., $c(\alpha,\beta)<\beta$ for all $\alpha<\beta<\kappa$)
such that, for every $B\in J^+$, there is an $\alpha<\kappa$ such that $\otp(c[\{\alpha\}\circledast B])=\theta$.
\end{defn}

Let $J^{\bd}[\kappa]$ stand for the ideal of bounded subsets of $\kappa$,
let 
$\ns_\kappa$ stand for the ideal of nonstationary subsets of $\kappa$,
and let $\ns_\kappa\restriction S:=\ns_\kappa\cap\mathcal P(S)$.
We prove:
\begin{thma}\label{theorema}\begin{enumerate}[(1)]
\item $\ubd(J^{\bd}[\lambda^+],\lambda)$ holds for every singular cardinal $\lambda$;
\item If $\kappa$ is a regular uncountable cardinal for which $\ubd(J^{\bd}[\kappa],\kappa)$ fails, then $\kappa$ is greatly Mahlo;
\item $\ubd(J^{\bd}[\kappa],\omega)$ holds iff $\kappa$ is not weakly compact;
\item If $\kappa=\cf(\kappa)>\omega$, then $\ubd(\ns_\kappa\restriction S,\omega)$ holds iff $S\s \kappa$ is not ineffable.
\end{enumerate}
\end{thma}
	
As these findings are close in spirit to the classical open questions,
the possibility arises that they could serve as hints towards their eventual solution.
More broadly, there is a general bootstrapping phenomenon in which strong colouring theorems for a cardinal $\kappa_0$
give rise to strong club-guessing theorems for a cardinal $\kappa_1>\kappa_0$ that then 	give rise to strong colouring theorems for a cardinal $\kappa_2>\kappa_1$.
To demonstrate, let us briefly review a relatively recent construction of a colouring that is strong in the traditional sense.

In \cite[\S3]{paper18}, a colouring witnessing a strong form of $\kappa\nrightarrow[\kappa]^2_\kappa$ denoted $\pr_1(\kappa,\kappa,\kappa,\chi)$ 
was constructed from $\square(\kappa)$ and the oscillation oracle $\pl_6(\chi^+,\omega,\chi)$.
As made clear by the proof of \cite[Theorem~2.3]{paper15},
for an infinite $\theta$, the existence of a club-guessing sequence that may be partitioned into $\theta$ many active pieces
gives rise to $\pl_6(\chi^+,\theta,\chi)$. 	 Whether every club-guessing sequence may be partitioned is an open problem. 
Prior to our research project, the best results appearing in the literature were due to Shelah \cite[\S3]{she572} where a variety of methods was used.
While working on the paper \cite{paper46}, we realised that the following strengthening of the $\ubd(\ldots)$ principle is sufficient	for partitioning club-guessing sequences in a uniform way. 
\begin{defn}\label{defonto} For an ideal $J$ over $\kappa$, $\onto(J,\theta)$ 	asserts
the existence of a colouring $c:[\kappa]^2\rightarrow\theta$ 
such that, for every $B\in J^+$, there is an $\alpha<\kappa$ with $c[\{\alpha\}\circledast B]=\theta$.
\end{defn}

Let us now step back to the simpler problem of partitioning \emph{sets} which are positive with respect to an ideal.
The classical elementary method uses Ulam matrices \cite{ulam1930masstheorie}. 
These apply to successor cardinals, and Hajnal \cite{MR260597} has generalised them to cardinals which admit a stationary set not reflecting at regulars. 
Another method is the one used to give an elementary proof (see \cite[\S8]{MR1940513}) of the famous theorem of Solovay \cite{rvmc} that every stationary subset of a regular cardinal $\kappa$ can be decomposed into $\kappa$-many pairwise disjoint stationary subsets. 
The original inspiration for our definition of $\ubd(\ldots)$ and $\onto(\ldots)$ was the \emph{onto mapping principle} of Sierpi\'nski \cite{sierpinski1934hypothese},
but upon further investigation we realised that these two principles and their variants have sufficient generality to capture --- and hence to compare --- these and other methods for proving results about the non-weak-saturation properties of ideals. We emphasise though that our focus, which comes from our desired application, is on providing a \emph{uniform} method of obtaining these results. 
For that reason, constructions such as those appearing in \cite{MR0369081,paper29} relying on a recursive process are not applicable even though they are in some cases stronger results
when seen purely from the point of view of weak saturation. We elaborate more on this at end of Subsection~\ref{nonweaksaturation}.

As suggested by Clauses (3) and (4) of Theorem A, 
the principles under discussion also provide a characterisation of large cardinal properties. 
It is easy to see that
\begin{itemize}
\item 	$\kappa$ is ineffable iff $\onto(\ns_\kappa,2)$ fails, and that
\item  $\kappa$ is almost ineffable iff $\onto(J^{\bd}[\kappa],2)$ fails.
\end{itemize}
We prove here:
\begin{itemize}
\item $\kappa>\aleph_0$ is weakly compact iff $\onto(J^{\bd}[\kappa],3)$ fails;
\item $\kappa\ge2^{\aleph_0}$ is weakly compact iff $\onto(J^{\bd}[\kappa],\aleph_0)$ fails.
\end{itemize}

In G{\"o}del's constructible universe $\L$, many of our principles coincide,
but in general it is possible to distinguish between them.
First, $\onto(J^{\bd}[\aleph_1],\aleph_0)$ holds, and a result of Larson \cite{MR2298476} shows that $\onto(\ns_{\aleph_1},\aleph_1)$ may consistently fail.
Second, using large cardinals, additional patterns of failure may be obtained:
\begin{thmb}
Assuming the consistency of large cardinals,
each of the following are consistent with $\kappa$ being a regular uncountable limit cardinal:
\begin{enumerate}[(1)]
\item $\kappa=2^\theta$, but $\ubd(\ns_\kappa,\theta)$ fails;
\item $\aleph_0<\theta<\kappa$,	$\onto(J^{\bd}[\kappa],\theta)$ holds, but 	$\ubd(\ns_\kappa,\theta^+)$ fails;
\item $\sup\{\theta<\kappa\mid \onto(J^{\bd}[\kappa],\theta)\text{ holds}\}=\kappa$,
but $\ubd(\ns_\kappa,\kappa)$ fails;
\item $\onto(\ns_\kappa,\kappa)$ holds, but 	$\ubd(J^{\bd}[\kappa],\omega)$ fails.
\end{enumerate}
\end{thmb}

We shall prove pump-up theorems for deriving strong forms of $\onto(J,\theta)$ from $\ubd(J,\theta')$, typically with $\theta<\theta'$.
We also prove theorems for getting $\onto(J,\theta)$ by other means. Here is a list of corollaries.
	
\begin{thmc}\begin{enumerate}[(1)]
\item If $\kappa$ is a successor cardinal, then $\onto(J^\bd[\kappa],\theta)$ holds for every regular cardinal $\theta<\kappa$,
and $\stick(\kappa)$ implies $\onto(J^\bd[\kappa],\kappa)$;	\item For every stationary $S\s\kappa$, there exists a stationary $S'\s S$ such that $\onto(\ns_\kappa\restriction S',\theta)$ holds for every regular cardinal $\theta<\kappa$;
\item If $\diamondsuit^*(\kappa)$ holds, then so does $\onto(\ns_\kappa,\kappa)$;
\item If $\diamondsuit(T)$ holds for some stationary $T\s\kappa$ that does not reflect at regulars,
then $\onto(J^\bd[\kappa],\kappa)$ holds.
\end{enumerate}
\end{thmc}

The proof of Clause~(4) goes through the colouring principle 
$\ubd^*(\ldots)$ having the property that, for normal ideals $I$ and $J$, $\ubd^*(J,I^+)$ allows to derive $\onto(I,\theta)$ from $\onto(J,\theta)$.
What's more, we shall prove that 
for every stationary $T\s\kappa$, there exists a club $C\s\kappa$ such that for $T':=T\cap C$:

$\ubd^*(J^\bd[\kappa],\{ T'\})$ holds
$\iff$
$\kappa$ carries a triangular Ulam matrix with support $T'$
$\iff$ $T$ does not reflect at regulars.

It thus follows from Clause~(2) of Theorem~A that the principle $\ubd(J^{\bd}[\kappa],\allowbreak\kappa)$ is weaker than the existence of a triangular Ulam matrix at $\kappa$
yet $\ubd(\ns_\kappa,\allowbreak\theta)$ 
implies that any positive set of any normal ideal $J$ over $\kappa$ may be partitioned into $\theta$-many pairwise disjoint $J$-positive sets.
Here is a concise summary of what is known about $\onto(\ns_\kappa,\theta)$.

\begin{thmd} If $\onto(\ns_\kappa,\theta)$ fails for a pair of infinite regular cardinals $\theta<\kappa$, then:
\begin{enumerate}[(1)]
\item $\kappa$ is greatly Mahlo;
\item $\diamondsuit^*(\reg(\kappa))$ fails;
\item $\refl(\kappa,\kappa,\reg(\kappa))$ holds. In particular:
\item $\square(\kappa,{<}\mu)$ fails for all $\mu<\kappa$.
\item $\kappa\nrightarrow[\kappa]^2_\theta$ fails. In particular, there are no $\kappa$-Souslin trees.
\end{enumerate}
\end{thmd}

\subsection{Organization of this paper}
In Section~\ref{sectioncolour} we introduce the colouring principles, some of which we have already seen and other variants, that we will be studying in this paper and we give the basic application that motivated us, that of partitioning positive sets of an ideal into many disjoint positive sets.

In Section~\ref{sectionstronglyamenable} we introduce the ideal $\sa_\kappa$ which is strongly tied to the weak-compactness of $\kappa$. 
In $\L$, the ideal $\sa_\kappa$ is nontrivial if and only if $\kappa$ is weakly compact.
In general, if $\kappa$ is weakly compact, then $\sa_\kappa$ is nontrivial,
and if $\sa_\kappa$ is nontrivial then $\kappa$ is weakly compact in $\L$.
We also show that $\ubd(J^\bd[\kappa],\kappa)$ holds exactly when $\sa_\kappa$ is trivial,
and provide some facts connecting forcing with these ideals.
The proof of Clause~(2) of Theorem~A will be found there.

In Section~\ref{sectionamenable} we introduce the ideal $\amen_\kappa$ which is strongly tied to ineffability of $\kappa$.
If $\kappa$ is ineffable, then $\amen_\kappa$ is nontrivial. In $\L$, the converse holds as well.
We show that $\ubd(\ns_\kappa,\kappa)$ holds exactly when $\amen_\kappa$ is trivial,
and then we provide some consistency strength lower bounds for when this does not happen. Again, we end with some facts connecting forcing with these ideals.

In Section~\ref{sectionulam} we prove that one of the strongest principles we formulate, the $\ubd^*$ principle, captures exactly the notion of an Ulam matrix as formulated by Ulam and its generalisation due to Hajnal. This also allows us to give explicit evidence that the $\ubd$ principles we have formulated are a more applicable method of partitioning positive sets of various ideals.

Section~\ref{sectionpumping} is the most technical of our sections. The theme here is to obtain implications between the various instances of our colouring principles as well as the classical negative square bracket principles. In particular we prove pumping-up theorems for the principles we have introduced as well as establish various monotonicity results between them. To do this we use a variety of guessing principles some of which are available in $\zfc$ and some of which are not. 
The proof of Theorem~D will be found there.

In Section~\ref{sectionzfc} we prove $\zfc$ results about our principles bolstered by the results of the previous sections, most strongly those of Section~\ref{sectionpumping} which allow us to draw stronger conclusions. In particular we consider the case of singular cardinals and their successors, as well as cardinals below the continuum. 
The proofs of Clause~(1) of Theorem~A, and Clause~(2) of Theorem~C will be found there.
This section is concluded with an affirmative answer to a weak form of Fodor's question.

In Section~\ref{sectionontomax} we consider various aspects of the $\onto$ principles with the maximum number of colours. In particular we obtain them in certain instances from guessing principles and we establish pump-up theorems between them and some of the classical principles.
The proofs of Clauses (1),(3) and (4) of Theorem~C will be found there.

In Section~\ref{sectionfailures} we collect together various failures and consistent failures of our colouring principles, 
some of them obtained by inspecting Kunen's model \cite[\S3]{MR495118} demonstrating the resurrection phenomena of large cardinals.
We also prove that the instance of $\ubd^{++}$, let alone $\onto^{++}$, with the maximum number of colours always fails. The proof of Theorem B will be found there.

In Section~\ref{sectionweaklycompact} we prove Clause~(3) of Theorem~A and some related results which allow us to characterise weakly compact cardinals using our principles.

In Section~\ref{sectionineffable} we prove Clause~(4) of Theorem~A and some related results which allow us to characterise ineffable cardinals using our principles.

In the Appendix, we provide a diagram which summarises our results as they relate to the characterisation of weakly compact and ineffable cardinals.

\subsection{So was Ulam right?}\label{ulamremark} As previously mentioned, while working on the paper \cite{paper46} about club-guessing and after having proved particular applications of many of the theorems in this paper, 
we realised that the onto mapping principle of Sierpi\'nski suggests a common framework for capturing almost all the partitioning results we were able to prove except for the ones proved using Ulam matrices. 
Later with the introduction of $\ubd^*$ we were able to capture Ulam matrices as well. 

So which is the correct way to partition positive sets of a $\kappa$-complete ideal over $\kappa$?
For a large class of ideals $J$ which we call \emph{subnormal} (see Definition~\ref{subnormalideals} below), 
the new principle $\ubd(J,\theta)$ turns out to be sufficient. 
For the general case, Ulam matrices are still superior,
though there is a \emph{narrow} variant of $\ubd(J,\theta)$ that does the job, as well.
Narrow colourings will be the subject matter of a sequel to this paper \cite{paper53}.

\subsection{Notation and conventions}\label{nandc} 
Let $\reg(\kappa)$ denote the collection of all infinite regular cardinals below $\kappa$.
Let $E^\kappa_\theta:=\{\alpha < \kappa \mid \cf(\alpha) = \theta\}$,
and define $E^\kappa_{\le \theta}$, $E^\kappa_{<\theta}$, $E^\kappa_{\ge \theta}$, $E^\kappa_{>\theta}$,  $E^\kappa_{\neq\theta}$ analogously.
For a set of ordinals $A$, we write $\ssup(A) := \sup\{\alpha + 1 \mid \alpha \in A\}$, $\acc^+(A) := \{\alpha < \ssup(A) \mid \sup(A \cap \alpha) = \alpha > 0\}$,
$\acc(A) := A \cap \acc^+(A)$, and $\nacc(A) := A \setminus \acc(A)$.
For a stationary $S\s \kappa$, we write
$\Tr(S):= \{\alpha \in E^\kappa_{>\omega}\mid  S\cap \alpha\text{ is stationary in }\alpha\}$.
For the iterated trace operations $\Tr^\alpha$, see Definition~\ref{iteratedtrace} below.
A cardinal $\kappa$ is \emph{greatly Mahlo} iff $\Tr^\alpha(\reg(\kappa))$ is stationary for every $\alpha<\kappa^+$.
The principle $\refl(\kappa,S,T)$ asserts that for every sequence $\langle S_i\mid i<\kappa\rangle$ of stationary subsets of $S$, 
there exists $\delta\in T$ such that $\delta\in\bigcap_{i<\delta}\Tr(S_i)$.
For $A,B$ sets of ordinals, we denote $A\circledast B:=\{(\alpha,\beta)\in A\times B\mid \alpha<\beta\}$
and we identify $[B]^2$ with $B\circledast B$.
In particular, we interpret the domain of a colouring $c:[\kappa]^2\rightarrow\theta$ as a collection of ordered pairs.
In scenarios in which we are given an unordered pair $p=\{\alpha,\beta\}$,
we shall write $c(\{\alpha,\beta\})$ for $c(\min(p),\max(p))$. We also agree to interpret $c(\{\alpha,\beta\})$ as $0$, whenever $\alpha=\beta$.
For $\theta\neq 2$, $[\kappa]^\theta$ stands for the collection of all subsets of $\kappa$ of size $\theta$.
Similar to Definition~\ref{hungarian}, $\kappa\nrightarrow[\mu; \nu]^2_\theta$ asserts the existence of a colouring $c:[\kappa]^2\rightarrow\theta$ such that, for all $A\in[\kappa]^\mu$ and $B\in[\kappa]^\nu$, $c[A \circledast B]=\theta$.

\section{Colouring principles and applications}
\label{sectioncolour}
In this section we introduce the basic objects that are the subject of this paper. To start, in Definitions~\ref{def21} and \ref{def23} we state in full generality the colouring principles we will study. Proposition~\ref{remark25} which follows is a prototype of the results of Section~\ref{sectionpumping}. In Subsection~\ref{subsectionsubnormal} we define the titular class of ideals we shall focus on, the \emph{subnormal ideals} and gather some elementary facts about them. In Subsection~\ref{subsectionlargecardinals} we remind the reader of some basic facts about weakly compact and ineffable cardinals with which, as succinctly captured by the diagram in the Appendix, our colouring principles are naturally connected. We also prove some new results here which can be read independently of the rest of the paper but which, together with the results of Sections~\ref{sectionweaklycompact}, provide new characterisations of weakly compact cardinals. In subsection~\ref{nonweaksaturation} our focus is on demonstrating non-weak-saturation results using our colouring principles. In particular, we justify our focus on subnormal ideals by showing how they allow us to upgrade colouring principles of the form we have already seen in 
Definitions \ref{defnubd} and \ref{defonto} to stronger principles which explicitly imply non-weak-saturation results. We finish with two more applications. The first, in Subsection~\ref{subsectionpartition}, is to partitioning club-guessing. The second, in Subsection~\ref{subsectioninconsistency}, to give another proof of Kunen's Inconsistency Theorem.

\begin{defn}\label{def21} For 
a family $\mathcal A\s\mathcal P(\kappa)$,
an ideal $J$ over $\kappa$,
and a set $S\s\kappa$:
\begin{itemize}
\item $\onto^{++}(\mathcal A,J,\theta)$ asserts
the existence of a colouring $c:[\kappa]^2\rightarrow\theta$ with the property that,
for every $A\in\mathcal A$ and every sequence $\langle B_\tau\mid \tau<\theta\rangle$ of elements of $J^+$,
there is an $\eta\in A$ such that $\{ \beta\in B_\tau\setminus(\eta+1)\mid c(\eta,\beta)=\tau\}\in J^+$ for every $\tau<\theta$;
\item $\onto^+(\mathcal A,J,\theta)$ asserts 
the existence of a colouring $c:[\kappa]^2\rightarrow\theta$ with the property that,
for all $A\in\mathcal A$ and $B\in J^+$,
there is an $\eta\in A$ such that, for every $\tau<\theta$, $\{\beta\in B\setminus(\eta+1)\mid c(\eta,\beta)=\tau\}\in J^+$;
\item $\onto(\mathcal A,J,\theta)$ asserts 
the existence of a colouring $c:[\kappa]^2\rightarrow\theta$ with the property that,
for all $A\in\mathcal A$ and $B\in J^+$,
there is an $\eta\in A$ such that $$c[\{\eta\}\circledast B]=\theta;$$
\item $\onto^-(S,\theta)$ asserts 
the existence of a colouring $c:[\kappa]^2\rightarrow\theta$ with the property that,
for every club $D\s\kappa$ 
and for every regressive map $f:S\cap D\rightarrow\kappa$,
there are $\eta_0,\eta_1<\kappa$ such that 
$$c[\{\eta_0\}\circledast\{\beta\in S\cap D\mid f(\beta)=\eta_1\}]=\theta.$$
\end{itemize}

If we omit $\mathcal A$ (as in Definition~\ref{defonto}), then we mean that $\mathcal A:=\{\kappa\}$.
\end{defn}
\begin{remark} \label{relationsremark}\begin{enumerate}
\item It is clear that the principles $\onto^{++}(\mathcal A,J,\theta)$, $\onto^+(\mathcal A,J,\theta)$, and $\onto(\mathcal A,J,\theta)$ 
are decreasing in strength in the sense that a colouring witnessing a principle occurring earlier also witnesses a principle later in this order. 
As well, if $J$ is a normal ideal extending $\ns_\kappa\restriction S$ for a stationary subset $S\subseteq \kappa$,
then any colouring witnessing $\onto(J,\theta)$ is a witness to $\onto^-(S,\theta)$.
\item The principle $\onto(\{\aleph_0\},J^{\bd}[\aleph_1],\aleph_1)$ is better known as \emph{Sierpi\'nski's onto mapping principle} (see Fact~\ref{sierpinski} below) 
which gives rise to the notion of a \emph{narrow} colouring, that is, a colouring witnessing a principle $\mathsf p(\mathcal A,J,\theta)$ in which there is an $A\in\mathcal A$ with $|A|<|\bigcup J|$.
The instance $\onto([\kappa]^\kappa,J^{\bd}[\kappa],\theta)$ is better known as the partition relation
$\kappa\nrightarrow[\faktor{{\scriptstyle{{\kappa}\circledast\kappa}}}{    {}^{1\circledast\kappa}}]^2_{\theta}$ (see Proposition~\ref{prop46} below).
For additional connections to other well-studied concepts, see Propositions \ref{prop45} and \ref{prop47} below.		
\item\label{upgraderemark}
As was just demonstrated, the principle $\onto(\ldots)$ is in line with traditional strong colouring principles,
where one wishes to realise many colours on every positive set of an ideal. 
On the other hand, the principle $\onto^+(\ldots)$ is more catered towards decomposing positive sets of an ideal into many disjoint positive sets. 
Nevertheless, as we shall see in Section~\ref{nonweaksaturation}, 
if the ideal satisfies some extra properties, most importantly for this paper that of \emph{subnormality} (see Definition~\ref{subnormalideals}), 
then one can upgrade a colouring witnessing $\onto(\ldots)$ to one witnessing $\onto^+(\ldots)$. 
In some cases in fact the same colouring witnesses the stronger principle.
\end{enumerate}
\end{remark}

We now consider a weakening of the above principles by relaxing the requirement of realising the full range to one of realising the maximal ordertype. 
As we will see in Theorem~\ref{babycg}, this is often enough for certain applications. 
Recall that for $\theta\leq \kappa$ a colouring $c:[\kappa]^2 \rightarrow \theta$ is \emph{upper-regressive} if $c(\alpha,\beta)<\beta$ for all $\alpha<\beta<\kappa$.
\begin{defn}\label{def23} For families $\mathcal A,\mathcal T\s\mathcal P(\kappa)$ and an ideal $J$ over $\kappa$:
\begin{itemize}
\item $\ubd^{++}(\mathcal A,J,\theta)$ asserts 
the existence of an upper-regressive colouring $c:[\kappa]^2\rightarrow\theta$ with the property that,
for every $A\in\mathcal A$ and every sequence $\langle B_\tau\mid \tau<\theta\rangle$ of elements of $J^+$,
there is an $\eta\in A$ and an injection $h:\theta\rightarrow\theta$ such that, 
for every $\tau<\theta$, $\{ \beta\in B_\tau\setminus(\eta+1)\mid c(\eta,\beta)=h(\tau)\}\in J^+$;
\item $\ubd^+(\mathcal A,J,\theta)$ asserts 
the existence of an upper-regressive colouring $c:[\kappa]^2\rightarrow\theta$ with the property that,
for all $A\in\mathcal A$ and $B\in J^+$,
there is an $\eta\in A$ such that
$$\otp(\{\tau<\theta\mid \{\beta\in B\setminus(\eta+1)\mid c(\eta,\beta)=\tau\}\in J^+\})=\theta;$$
\item $\ubd^*(\mathcal A,J,\mathcal T)$ asserts 
the existence of an upper-regressive colouring $c:[\kappa]^2\rightarrow\kappa$ with the property that,
for all $A\in\mathcal A$ and $B\in J^+$,
there is $T\in\mathcal T$, such that, for every $\tau\in T$,
there are $\eta\in A\cap\tau$ and $\beta\in B\setminus(\tau+1)$ such that $c(\eta,\beta)=\tau$;
\item $\ubd(\mathcal A,J,\theta)$ asserts 
the existence of an upper-regressive colouring $c:[\kappa]^2\rightarrow\theta$ with the property that,
for all $A\in\mathcal A$ and $B\in J^+$,
there is an $\eta\in A$ such that 
$$\otp(c[\{\eta\}\circledast B])=\theta.$$
\end{itemize}

If we omit $\mathcal A$ (as in Definition~\ref{defnubd}), then we mean that $\mathcal A:=\{\kappa\}$.
\end{defn}

\begin{remark}\label{remark25a}
\begin{enumerate} 
\item For $\theta$ a finite cardinal, 
$\ubd(\mathcal A, J, \theta)$ is equivalent to $\onto(\mathcal A, J, \theta)$,
and $\ubd^{+}(\mathcal A, J, \theta)$ is equivalent to $\onto^{+}(\mathcal A, J, \theta)$.
\item The principle $\ubd^*(\mathcal A',J,\mathcal T)$ implies $\ubd^*(\mathcal A,J',\mathcal T')$
whenever $\mathcal A\s\mathcal A'$ and $J \s J'$ and $\mathcal T\s\mathcal T'$.
\item
For $\kappa$ regular uncountable, $\ubd^*(\mathcal A,J,(\ns_\kappa)^+)$ implies $\ubd(\mathcal A,\allowbreak J,\kappa)$.
A reasonable conjecture would assert that $\ubd^*(\mathcal A,J,(\ns_\kappa)^*)$ implies $\onto(\mathcal A,J,\kappa)$,
but this conjecture is refuted by Facts~\ref{ulamoriginal} and \ref{larson}.
\end{enumerate}
\end{remark}

\begin{prop}\label{remark25} For $\mathcal A\s\mathcal P(\kappa)$, an ideal $J$ over $\kappa$, and a (possibly finite) cardinal $\theta<\kappa$:
\begin{enumerate}[(1)]
\item If $\onto(\mathcal A,J,\kappa)$ holds, then it may be witnessed by an upper-regressive map;
\item If $\onto(\mathcal A,J,\kappa)$ holds, then so does $\ubd^*(\mathcal A,J,\{\kappa\setminus\epsilon\mid \epsilon<\kappa\})$;
\item If $J\supseteq J^{\bd}[\kappa]$ and $2^\theta\le\kappa$, 
then $\ubd^{++}(J,\theta)$ implies $\onto^{++}(J,\theta)$,
and likewise $\ubd^+(J,\theta)$ implies $\onto^+(J,\theta)$.
\end{enumerate}
\end{prop}
\begin{proof} (1) Suppose that $c:[\kappa]^2\rightarrow\kappa$ is a colouring witnessing $\onto(\mathcal A,J,\kappa)$.
Fix a surjection $f:\kappa\rightarrow\kappa$ such that the preimage of any singleton has size $\kappa$.
Define a colouring $d:[\kappa]^2\rightarrow\kappa$ by letting, for all $\eta<\beta<\kappa$,
 $d(\eta,\beta):=f(c(\eta,\beta))$ provided that $f(c(\eta,\beta))<\beta$; otherwise, let $d(\eta,\beta):=0$.
Clearly, $d$ is upper-regressive.

Now, given $A\in\mathcal A$ and $B\in J^+$, fix $\eta\in A$ such that $c[\{\eta\}\circledast B]=\kappa$.
 To see that $d[\{\eta\}\circledast B]=\kappa$,
let $\tau<\kappa$ be arbitrary. 
As the set $\{ \beta\in B\mid f(c(\eta,\beta))=\tau\}$ has size $\kappa$, it contains an element $\beta^*$ above $\max\{\eta,\tau\}$, so that $d(\eta,\beta^*)=\tau$.
  
(2) Any upper-regressive witness to $\onto(\mathcal A,J,\kappa)$ witnesses $\ubd^*(\mathcal A,J,\allowbreak\{{\kappa\setminus\epsilon}\mid \epsilon<\kappa\})$.

(3) Assuming $2^\theta\le\kappa$, fix a bijection $\pi:\kappa\leftrightarrow\kappa\times{}^\theta\theta$.
Now, given a colouring $c:[\kappa]^2 \rightarrow \theta$,
derive a colouring $d:[\kappa]^2\rightarrow \theta$ via $d(\eta, \beta) := h'(c\{\eta', \beta\})$ where  $\pi(\eta)= ( \eta',h')$.
It is easy to see that if $J\supseteq J^{\bd}[\kappa]$, 
then $d$ witnesses $\onto^{++}(J,\theta)$ whenever $c$ witnesses $\ubd^{++}(J, \theta)$,
and that $d$ witnesses $\onto^+(J,\theta)$ whenever $c$ witnesses $\ubd^+(J, \theta)$.
\end{proof}

\subsection{Subnormal ideals}\label{subsectionsubnormal} 
For a set of ordinals $S$, $J^{\bd}[S]$ stands for the ideal of bounded subsets of $S$. 
For $\nu$ an infinite cardinal, let
$$\mathcal J^\kappa_\nu:=\{ J\mid J\text{ is a }\nu\text{-complete ideal over }\kappa\text{ extending }J^{\bd}[\kappa]\}.$$
In particular, $\mathcal J^\kappa_\omega$ is the set of all ideals $J$ over $\kappa$ extending $J^\bd[\kappa]$. For $\kappa$ of uncountable cofinality, 
$\ns_\kappa$ stands for the ideal of nonstationary subsets of $\kappa$, and $\ns_\kappa\restriction S$ stands for $\ns_\kappa\cap\mathcal P(S)$.

For any ideal $J$ over $S$, we denote its dual filter by $J^*:=\{ S\setminus X\mid X\in J\}$,
and its collection of positive sets by $J^+:=\mathcal P(S)\setminus J$.
An ideal $J$ is \emph{trivial} iff $J^+=\emptyset$.
Note that  all the principles of Definitions \ref{def21} and \ref{def23} hold vacuously for trivial $J$'s,
and that any nontrivial ideal over a singular cardinal $\kappa$ extending $J^{\bd}[\kappa]$ 
is not $\cf(\kappa)^+$-complete, let alone $\kappa$-complete. 

\begin{conv}
In case $S$ is a cofinal subset of a cardinal $\kappa$, it is customary to identify an ideal $J$ over $S$ with the ideal $\hat J=\{ X\cup Y\mid X\in J,~ Y\in\mathcal P(\kappa\setminus S)\}$ over $\kappa$,
since they have the same collection of positive sets.
In particular, when we talk about ideals over $\kappa$ extending $J^{\bd}[\kappa]$
this also covers the cases of ideals over a cofinal subset $S$ of $\kappa$ extending $J^{\bd}[S]$.
\end{conv}

\begin{defn}[folklore] An ideal $J$ over $\kappa$ is said to be \emph{normal} if for every sequence $\langle E_\eta \mid \eta < \kappa\rangle$ of sets from $J^*$, 
its diagonal intersection $\diagonal_{\eta < \kappa} E_\eta:= \{\beta < \kappa \mid \beta \in \bigcap_{\eta < \beta} E_\eta\}$ is in $J^*$.
\end{defn}

Note that if $\kappa$ is singular, then there is no nontrivial normal ideal over $\kappa$ extending $J^{\bd}[\kappa]$.
The following is well-known.
\begin{fact}\label{normalfacts} Suppose that $\kappa$ is regular uncountable, and that $J$ is a normal ideal over some stationary set $S\s\kappa$.
If $J$ extends $J^{\bd}[S]$, then:
\begin{enumerate}[(1)]
\item $J$ is $\kappa$-complete;
\item $J$ extends $\ns_\kappa\restriction S$.
\end{enumerate}
\end{fact}

We now introduce a variation of normality that we call \emph{subnormality}.
Every normal ideal is subnormal, but so is $J^{\bd}[\kappa]$ for every infinite cardinal $\kappa$.

\begin{defn}\label{subnormalideals}An ideal $J$ over $\kappa$ is said to be \emph{subnormal} if for every sequence $\langle E_\eta\mid \eta<\kappa\rangle$ of sets from $J^*$,
the following two hold:
\begin{enumerate}
\item  for every $B\in J^+$, there exists $B'\s B$ in $J^+$ such that, for every $(\eta,\beta)\in [B']^2$,
$\beta\in E_\eta$;
\item  for all $A,B\in J^+$, there exist $A'\s A$ and $B'\s B$ with $A',B'\in J^+$ such that, for every $(\eta,\beta)\in A'\circledast B'$,
$\beta\in E_\eta$.
\end{enumerate}
\end{defn} 

\begin{prop} Suppose that $J$ is an ideal over $\kappa$,
and $f:\kappa\rightarrow\kappa$ is some function.
Denote $f_*(J):=\{ X\s\kappa\mid f^{-1}[X]\in J\}$.
\begin{enumerate}[(1)]
\item If $J$ is subnormal and $f$ is order-preserving on a set in $J^*$,	
then $f_*(J)$ is subnormal;
\item If $f$ is order-preserving, then for every principle $\textsf p\in\{\onto,\onto^+,\onto^{++},\allowbreak\ubd,\ubd^+,\ubd^{++}\}$,
$\textsf p(J,\theta)$ implies $\textsf p(f_*(J),\theta)$.
\end{enumerate}
\end{prop}
\begin{proof} 
For notational simplicity, denote $I:=f_*(J)$.

(1) Suppose that $C \in J^*$ is a set such that $f \restriction C$ is order-preserving. So, for every pair $\alpha<\beta$ of ordinals from $f[C]$, 
if $\bar\alpha,\bar\beta\in C$ are such that $f(\bar\alpha)=\alpha$ and $f(\bar\beta)=\beta$, then
it is the case that $\bar\alpha<\bar\beta$.

Suppose that $J$ is subnormal. To verify that $I$ is subnormal, let $\langle E_\eta\mid \eta<\kappa\rangle$ be a sequence of sets from $I^*$.
Note that for every $\eta<\kappa$ the set $\bar E_\eta:=f^{-1}[E_{f(\eta)}]$ is in $J^*$.
\begin{enumerate}
\item Given $B\in I^+$, the set $\bar B:=C\cap f^{-1}[B]$ is in $J^+$,
and so by subnormality of $J$, we may pick $\bar B'\s\bar B$ in $J^+$ such that $\bar\beta\in\bar E_{\bar\eta}$ for every $(\bar\eta,\bar\beta)\in [\bar B']^2$.
Evidently, $B':=f[\bar B']$ is a subset of $f[C]\cap B$ lying in $I^+$.

Let $(\eta,\beta)\in [B']^2$ be arbitrary.
Fix $\bar\eta,\bar\beta\in\bar B'$ such that $f(\bar\eta)=\eta$ and $f(\bar\beta)=\beta$.
Then $(\bar\eta,\bar\beta)\in[\bar B']^2$, so that $\bar\beta\in \bar E_{\bar\eta}=f^{-1}[E_{f(\bar\eta)}]=f^{-1}[E_\eta]$
and $\beta=f(\bar\beta)\in E_\eta$, as sought.

\item Given $A,B\in I^+$, pick $\bar A'\s C\cap f^{-1}[A]$ and $\bar B'\s C\cap f^{-1}[B]$ both in $J^+$
such that $\bar\beta\in \bar E_{\bar\eta}$ for every $(\bar\eta,\bar\beta)\in\bar A'\circledast\bar B'$.
Evidently $A':=f[\bar A']$ is a subset of $f[C]\cap A$,
$B':=f[\bar B']$ is a subset of $f[C]\cap B$, and $A',B'\in I^+$.

Let $(\eta,\beta)\in A'\circledast B'$ be arbitrary.
Fix $\bar\eta\in\bar A'$ and $\bar\beta\in\bar B'$ such that $f(\bar\eta)=\eta$ and $f(\bar\beta)=\beta$.
Then $(\bar\eta,\bar\beta)\in\bar A'\circledast \bar B'$, so that $\bar \beta \in \bar E_{\bar\eta}$. 
It follows that $\beta=f(\bar\beta)\in f[\bar E_{\bar\eta}]=f[f^{-1}[E_{f(\bar\eta)}]]=E_{f(\bar\eta)}=E_\eta$,
as sought.
\end{enumerate}

(2) Suppose that $f$ is order-preserving,
and that $c:[\kappa]^2\rightarrow\theta$ is a colouring witnessing that $\textsf p(J,\theta)$ holds for a principle $\textsf p$ as above.
For every $\delta\in \im(f)$, let $\bar \delta$ denote the unique ordinal in $\kappa$ to satisfy $f(\bar \delta)=\delta$.
Then, pick a colouring $d:[\kappa]^2\rightarrow\theta$ such that 
$d(\alpha,\beta)=c(\bar\alpha,\bar\beta)$ for every $(\alpha,\beta)\in[\im(f)]^2$.
Note that since $f$ is order-preserving, if $c$ is upper-regressive, then we may also require $d$ to be upper-regressive.

Finally, for any $B\in I^+$, $\bar B:=f^{-1}[B]$ is in $J^+$,
and for every subset $\bar X\s\bar B$ lying in $J^+$,
$X:=f[\bar A]$ is a subset of $B$ lying in $I^+$, and for every $\eta<\kappa$,
$c[\{\eta\}\circledast\bar X]= d[\{f(\eta)\}\circledast X]$.
It thus follows that $d$ witnesses $\mathsf p(I,\theta)$.
\end{proof}

\begin{lemma} \label{subnormalcomplete}
Suppose that 		$J\in\mathcal J^\kappa_\kappa$ is subnormal,
and that $\langle E_\eta \mid \eta< \kappa\rangle$ is a sequence of sets from $J^*$.
\begin{enumerate}
\item For every $B \in J^+$,
there exists $B' \s B$ in $J^+$ such that, for every $(\alpha, \beta) \in [B']^2$, $\beta \in \bigcap _{\eta\le\alpha}E_{\eta}$;
\item For all $A,B \in J^+$,
there exist $A'\s A$ and $B'\s B$ with $A',B'\in J^+$ such that, for every $(\alpha,\beta)\in A'\circledast B'$,
$\beta \in \bigcap _{\eta\le\alpha}E_{\eta}$.
\end{enumerate}
\end{lemma}
\begin{proof}
For each $\alpha< \kappa$, let $E'_\alpha:= \bigcap_{\eta \leq \alpha} E_{\eta}$ which by the $\kappa$-completeness of $J$ is in $J^*$. 
Now, invoke the subnormality of $J$ on the sequence $\langle E'_\alpha \mid \alpha< \kappa\rangle$.
\end{proof}

\begin{lemma}\label{dseparated} Suppose that $J\in\mathcal J^\kappa_\omega$ is subnormal,
$B\in J^+$ and $D$ is some cofinal subset of $\kappa$. Then there exists $B'\s B$ in $J^+$ which is \emph{$D$-separated},
that is, for every $(\eta,\beta)\in[B']^2$, there is $\delta\in D$ with $\eta<\delta<\beta$.
\end{lemma}
\begin{proof}  For every $\eta<\kappa$, set $\delta_\eta:=\min(D\setminus(\eta+1))$. 
As $J$ extends $J^{\bd}[\kappa]$, for every $\eta<\kappa$, 
$E_\eta:=\kappa\setminus(\delta_\eta+1)$ is in $J^*$.
As $J$ is subnormal, fix $B'\s B$ in $J^+$ such that, for every $(\eta,\beta)\in [B']^2$,
$\beta\in E_\eta$. Then for every pair $\eta<\beta$ of points from $B'$, $\eta<\delta_\eta<\beta$.
\end{proof}

\subsection{Large cardinals,  colourings and $C$-sequences}\label{subsectionlargecardinals} 
In this subsection, we consider a couple of large cardinal notions.

\begin{defn}[\cite{MR0167422,MR0151397,MR166107}]\label{defnwc}
A cardinal $\kappa$ is \emph{weakly compact} if it is strongly inaccessible
and there are no $\kappa$-Aronszajn trees. Equivalently, if it is uncountable and $\kappa \nrightarrow[\kappa]^2_2$ fails.
\end{defn}
\begin{remark}
Note that we have adopted the convention that $\aleph_0$ is \emph{not} weakly compact even though $\aleph_0 \rightarrow[\aleph_0]^2_2$ does hold.
\end{remark}		
Shortly after Jensen found a construction of a Souslin tree in $\L$ \cite{MR3618579}, 
he extracted from it a combinatorial principle named \emph{diamond}. The standard formulation of $\diamondsuit$, as appears in \cite{jensen},
is very close to Jensen's original principle, but nevertheless is due to Kunen.
It didn't take long until Silver squeezed more out of the proof and formulated $\diamondsuit^*$,
and then the three of them came up with the stronger principle that the proof gave, known as $\diamondsuit^+$. 
These strong combinatorial principles along with related large cardinals notions were then studied in the Jensen-Kunen manuscript \cite{jensen1969some}.
	
\begin{defn}[\cite{jensen1969some}]
For $\kappa$ a regular uncountable cardinal, a subset $ S \s \kappa$ is an \emph{ineffable} (resp.~\emph{almost ineffable}) subset of $\kappa$ if for every sequence $\langle A_\beta\mid \beta \in  S\rangle$, 
there is an $A \s \kappa$ such that the following set is stationary (resp.~cofinal) in $\kappa$: 
$$\{\beta \in  S \mid A_\beta\cap\beta = A\cap \beta\}.$$

We say that an uncountable cardinal $\kappa$ is \emph{ineffable} if it is regular and it is an ineffable subset of itself. 
If it is clear from the context, then we will simply call a subset $S\s \kappa$ an ineffable set without making a reference to $\kappa$.
\end{defn}

\begin{fact}[Kunen, \cite{jensen1969some}]
For $\kappa$ an uncountable regular cardinal, a subset $S \s \kappa$ is ineffable iff 
for every colouring $c:[\kappa]^2\rightarrow2$, there exists a stationary  $B\s S$ such that $c$ is constant over $[B]^2$ (i.e., $\kappa \nrightarrow[\stat(S)]^2_2$ fails).
\end{fact}

\begin{fact}[implicit in \cite{jensen1969some}]\label{reducetoone} Suppose $c:[\kappa]^2\rightarrow\theta$ is a colouring with $0<\theta<\kappa$.
\begin{enumerate}[(1)] 
\item For every $S$ ineffable in $\kappa$, there exists a stationary $B\s S$,
such that, for every $\eta<\kappa$, $|c[\{\eta\}\circledast B]|=1$;
\item For every $S$ almost ineffable in $\kappa$, there exists a cofinal $B\s S$,
such that, for every $\eta<\kappa$, $|c[\{\eta\}\circledast B]|=1$.
\end{enumerate}
\end{fact}

We complement this with a related result about weakly compact cardinals.
\begin{prop}\label{reducetotwo} Suppose that $c:[\kappa]^2\rightarrow\theta$ is a colouring with $0<\theta<\kappa$.

If $\kappa=\aleph_0$ or if $\kappa$ is weakly compact, 
then there is a cofinal subset $B \s \kappa$ such that, for every $\eta< \kappa$, $|c[\{\eta\}\circledast B]|\le 2$.
\end{prop}
\begin{proof} 
Suppose that $\kappa \rightarrow[\kappa]^2_2$ holds. It follows that $\kappa$ is regular
and there exists a function $f:\kappa\rightarrow\theta$ with the property that for every $\alpha<\kappa$ there is a $\beta\ge\alpha$ such that $f\restriction\alpha\s c(\cdot,\beta)$. The least such $\beta$ will be denoted by $\beta(\alpha)$.
Recursively construct an increasing sequence of ordinals $\langle\alpha_\xi \mid \xi< \kappa\rangle$ such that, for all $\xi < \zeta < \kappa$, 
$\alpha_\xi < \beta(\alpha_\xi) < \alpha_{\zeta} < \beta(\alpha_{\zeta})<\kappa$.
In particular, $B: = \{\beta(\alpha_\xi) \mid \xi< \kappa\}$ is cofinal in $\kappa$.
Towards a contradiction, suppose that there is some $\eta< \kappa$ such that $c[\{\eta\}\circledast B]\geq 3$.
Pick a triple $\beta_0<\beta_1<\beta_2$ in $B\setminus(\eta+1)$ such that $|\{ c(\eta,\beta_i)\mid i<3\}|=3$.
For each $i<3$, pick $\alpha^i\in \{\alpha_\xi\mid \xi<\kappa\}$ such that $\beta_i=\beta(\alpha^i)$.
Evidently,
$$\eta<\beta_0<\alpha^1<\beta_1<\alpha^2<\beta_2.$$
For $i<2$, since $f\restriction \alpha^{i+1} \s c(\cdot, \beta_{i+1})$ and $\eta<\alpha^{i+1}$, $c(\eta,\beta_{i+1})=f(\eta)$.
So, $c(\eta,\beta_1)=c(\eta,\beta_2)$, contradicting the choice of $\beta_1$ and $\beta_2$.
\end{proof}
\begin{remark} 	In Corollary~\ref{weaklycompactthreepieces},
Proposition~\ref{effabletwopieces} and Remark~\ref{almosteffabletwopieces}, 
we establish the converse results. 
\end{remark}

Recall that a \emph{$C$-sequence} over a set of ordinals $S$ is a sequence $\vec C=\langle C_\beta\mid\beta\in S\rangle$ 
such that, for every $\beta\in S$, $C_\beta$ is a closed subset of $\beta$ with $\sup(C_\beta)=\sup(\beta)$.
In particular, $\alpha\in C_{\alpha+1}$ for all $\alpha<\kappa$.
Each of the large cardinal properties under discussion admit a characterisation in the language of $C$-sequences, as follows.

\begin{fact}[{\cite[Theorem~1.8]{TodActa}}]\label{factwc} A strongly inaccessible cardinal $\kappa$ is weakly compact iff for every $C$-sequence $\vec C=\langle C_\beta\mid\beta<\kappa\rangle$,
there exists a club $D$ in $\kappa$ such that, for every $\alpha<\kappa$, for some $\beta\in[\alpha,\kappa)$, $D\cap\alpha=C_\beta\cap\alpha$.
\end{fact}

\begin{fact}[{\cite[\S5]{benneria2021compactness}}]\label{fact213} A strongly inaccessible cardinal $\kappa$ is ineffable iff for every $C$-sequence $\vec C=\langle C_\beta\mid\beta<\kappa\rangle$,
there exists a club $D$ in $\kappa$ such that  $\{ \beta<\kappa\mid D\cap\beta= C_\beta\}$ is stationary in $\kappa$.
\end{fact}
\begin{remark} \label{remark213}
A similar proof shows that a strongly inaccessible cardinal $\kappa$ is almost ineffable iff for every $C$-sequence $\vec C=\langle C_\beta\mid\beta<\kappa\rangle$,
there exists a club $D$ in $\kappa$ such that  $\{ \beta<\kappa\mid D\cap\beta= C_\beta\}$ is cofinal in $\kappa$.
\end{remark}
We now obtain a result in the same spirit for weakly compact cardinals.
\begin{prop}\label{thm215} If $\kappa$ is weakly compact, then for every $C$-sequence $\vec C=\langle C_\beta\mid\beta<\kappa\rangle$,
there exists a club $D$ in $\kappa$ such that $\{ \beta<\kappa\mid D\cap\beta\s C_\beta\}$ is cofinal in $\kappa$.
\end{prop}
\begin{proof} Suppose that $\kappa$ is weakly compact. In particular, $\kappa$ is strongly inaccessible.
Now, given a $C$-sequence $\vec C=\langle C_\beta\mid\beta<\kappa\rangle$,
using Fact~\ref{factwc}, let us fix a club $D\s\kappa$ such that, 
for every $\alpha<\kappa$, for some $\beta\in[\alpha,\kappa)$, $D\cap\alpha=C_\beta\cap\alpha$.
For each $\epsilon<\kappa$, let $\beta_\epsilon$ denote the least ordinal $\beta$ such that $\epsilon<\beta<\kappa$ and 
$D\cap(\epsilon+1)=C_{\beta}\cap(\epsilon+1)$.
Now, consider the club $E:=\{\delta\in D\mid \forall\epsilon<\delta~(\beta_\epsilon<\delta)\}$.
\begin{claim} Let $\epsilon\in E$. Then $E\cap\beta_\epsilon\s C_{\beta_\epsilon}$.
\end{claim}
\begin{why} As $E\cap(\epsilon+1)\s D\cap(\epsilon+1)\s C_{\beta_\epsilon}$,
it suffices to prove that $E\cap[\epsilon+1,\beta_\epsilon)=\emptyset$.
Let $\delta$ be any element of $E$ above $\epsilon$,
then, by the definition of $E$, $\delta>\beta_\epsilon$, as sought.
\end{why}
So $\{\beta<\kappa\mid E\cap\beta\s C_\beta\}$ covers $\{ \beta_\epsilon\mid \epsilon\in E\}$,
which is a cofinal subset of $\kappa$.
\end{proof}

Models constructed by Kunen \cite{MR495118} show that the preceding property does not characterise weakly compact cardinals even when the cardinal is strongly inaccessible.
\begin{cor}\label{cor216}
\begin{enumerate}[(1)]
\item Assuming the consistency of a weakly compact cardinal,
it is consistent  that $\kappa$ is a strongly inaccessible cardinal which is not weakly compact but for every $C$-sequence $\vec C=\langle C_\beta\mid\beta<\kappa\rangle$,
there exists a club $D$ in $\kappa$ such that $\{ \beta<\kappa\mid D\cap\beta\s C_\beta\}$ is cofinal in $\kappa$;
\item Assuming the consistency of an ineffable cardinal,
it is consistent that $\kappa$ is a strongly inaccessible cardinal which is not weakly compact but for every $C$-sequence $\vec C=\langle C_\beta\mid\beta<\kappa\rangle$,
there exists a club $D$ in $\kappa$ such that $\{ \beta<\kappa\mid D\cap\beta\s C_\beta\}$ is stationary in $\kappa$.
\end{enumerate}
\end{cor}
\begin{proof}
Recall that in Kunen's model $V^{\mathbb R}$ from \cite[\S3]{MR495118}, $\kappa$ is a strongly inaccessible cardinal that is not weakly compact, but there exists a $\kappa$-cc forcing $\mathbb T$ such that:
\begin{itemize}
\item if $V\models``\kappa\text{ is weakly compact}"$,
then $V^{\mathbb R,\mathbb T}\models ``\kappa\text{ is weakly compact}"$;
\item if $V\models``\kappa\text{ is ineffable}"$;
then $V^{\mathbb R,\mathbb T}\models ``\kappa\text{ is ineffable}"$.
\end{itemize}

As $\mathbb T$ is $\kappa$-cc, for every club $D\s\kappa$ in $V^{\mathbb R,\mathbb T}$,
there exists a club $D'\s \kappa$ in $V^{\mathbb R}$ such that $D'\s D$. The conclusion now follows from Fact~\ref{fact213} and Proposition~\ref{thm215}.
\end{proof}

\subsection{Non-weak-saturation of ideals}\label{nonweaksaturation} In this short subsection, $\kappa$ stands for a regular uncountable cardinal,
and all ideals are understood to be nontrivial.
An ideal $J$ is \emph{weakly $\theta$-saturated} iff for every sequence $\langle B_i\mid i<\theta\rangle$ of $J^+$-sets 	there exists $(i,j)\in[\theta]^2$ with $B_i\cap B_j\neq\emptyset$. 
In this subsection we show that our colouring principles $\ubd(\ldots)$ and $\onto(\ldots)$
may often be boosted to $\ubd^+(\ldots)$ and $\onto^+(\ldots)$,
thus, inducing uniform decompositions witnessing that an ideal $J$ is not weakly $\theta$-saturated.
In some cases in fact any colouring witnessing the weaker principle witnesses as well the stronger principle. 
See Remark~\ref{noupgraderemark} below for a contrasting situation. 

In order to not burden the reader's memory, we only require familiarity with the definition of the $\onto(\ldots)$ and $\ubd(\ldots)$ principles in this subsection. 
In particular, instead of using the $\onto^+(\ldots)$ and $\ubd^+(\ldots)$ notation we use their expanded definitions.

\begin{prop}\label{weaksaturation} Suppose that $S\s\kappa$ is stationary.
\begin{enumerate}[(1)]
\item 
Suppose that $\ubd(\ns_\kappa\restriction S, \theta)$ holds.
Then there exists a colouring $c:[\kappa]^2 \rightarrow \theta$
such that, for every normal ideal $J$ over $S$ extending $J^{\bd}[S]$, for every $B \in J^+$, there is an $\eta< \kappa$ such that for $\theta$-many $\tau < \theta$,
$$\{\beta \in B\setminus(\eta+1) \mid c(\eta, \beta) = \tau\} \in J^+.$$
\item 
Suppose $c$ is a colouring witnessing that $\onto(\ns_\kappa\restriction S, \theta)$ holds.
Then, for every normal ideal $J$ over $S$ extending $J^{\bd}[S]$, for every $B \in J^+$, there is an $\eta< \kappa$ such that for every $\tau < \theta$,
$$\{\beta \in B\setminus(\eta+1) \mid c(\eta, \beta) = \tau\} \in J^+.$$
\end{enumerate}
\end{prop}
\begin{proof} (1)
The case $\theta=\kappa$ is covered by Lemma~\ref{lemma44}, so suppose that $\theta<\kappa$.
Fix a colouring $c:[\kappa]^2 \rightarrow \theta$ witnessing $\ubd(\ns_\kappa\restriction S, \theta)$.
Towards a contradiction, suppose that $J$ is a normal ideal over $S$ extending $J^{\bd}[S]$ and $B\in J^+$ together form a counterexample. 
Denote $B_\eta^\tau:=\{\beta \in B\setminus(\eta+1) \mid c(\eta, \beta) = \tau\} $.
By the choice of $B$, for every $\eta < \kappa$, 
the set
$T_\eta:=\{\tau<\theta\mid B_\eta^\tau\in J^+\}$ has size $<\theta$.
By Fact~\ref{normalfacts}, $J$ is $\theta^+$-complete, so $E_\eta:=S\setminus\bigcup_{\tau\in \theta\setminus T_\eta}B_\eta^\tau$ is in $J^*$.
As $J$ is normal, also $E:=\diagonal_{\eta < \kappa}E_\eta$ is in $J^*$.
As $B' := B\cap E$ is in $J^+$, Fact~\ref{normalfacts} implies that $B'$ is stationary,
so, since $c$ witnesses $\ubd(\ns_\kappa\restriction S, \theta)$,
there is an $\eta < \kappa$ such that $c[\{\eta\}\circledast  B']$ has ordertype $\theta$. In particular, we may pick $\tau\in c[\{\eta\}\circledast  B']\setminus T_\eta$.
Pick $\beta \in B'$ above $\eta$ such that $c(\eta, \beta)=\tau$. As $\beta \in B'$ and $\beta > \eta$, we have that $\beta \in E_\eta$,
and as $\tau\in\theta\setminus T_\eta$, $E_\eta\cap B_\eta^\tau=\emptyset$.
This is a contradiction.

(2) Left to the reader.
\end{proof}

We remind the reader before they read the following result that by our convention $\theta$ is always a cardinal, and if $\theta$ is in fact an infinite cardinal then $\theta+\theta=\theta$.
\begin{prop}\label{prop219} Suppose that $J$ is a subnormal ideal over $\kappa$ extending $J^{\bd}[\kappa]$.
For a colouring $c:[\kappa]^2\rightarrow\theta$, $B \in J^+$, and $\eta<\kappa$, denote:
$$\mathcal T^c_\eta(B):=\{ \tau<\theta\mid \{\beta \in B\setminus(\eta+1) \mid c(\eta, \beta) = \tau\} \in J^+\}.$$

\begin{enumerate}[(1)]
\item Suppose that $c$ witnesses $\ubd(J^+,J, \theta)$.
If $\theta$ is infinite, suppose also that $\theta<\kappa$ and that $J$ is $\theta^+$-complete.

Then, for all $A,B \in J^+$, there is an $\eta\in A$ such that $|\mathcal T^c_\eta(B)|=\theta$.
\item Suppose that $c$ witnesses $\onto(J^+,J, \theta)$.

Then, for all $A,B \in J^+$, there is an $\eta\in A$ such that $\mathcal T^c_\eta(B)=\theta$.
\item Suppose that $J$ is $\kappa$-complete, $\theta<\kappa$, 
$c$ witnesses $\ubd(J,\allowbreak\theta+\theta)$,
and $\pi:\theta+\theta\rightarrow\theta$ is some $2$-to-$1$ surjection.

Then, for every $B \in J^+$, there is an $\eta< \kappa$ such that $|\mathcal T^{\pi\circ c}_\eta(B)|=\theta$.
\item Suppose that $J$ is $\kappa$-complete, 
$c$ witnesses $\onto(J, \theta+\theta)$,
and $\pi:\theta+\theta\rightarrow\theta$ is some $2$-to-$1$ surjection.

Then, for every $B \in J^+$, there is an $\eta< \kappa$ such that $\mathcal T^{\pi\circ c}_\eta(B)=\theta$.
\end{enumerate}
\end{prop}
\begin{proof} 
Due to constraints of space, we settle for proving Clause~(3) as the proof contains all the ideas required to prove all the other clauses. 
Towards a contradiction, suppose that $B\in J^+$ forms a counterexample. 
Denote $B_\eta^\tau:=\{\beta \in B\setminus(\eta+1) \mid \pi(c(\eta, \beta)) = \tau\} $.
By the choice of $B$, for every $\eta < \kappa$, 
the set
$T_\eta:=\{\tau<\theta\mid B_\eta^\tau\in J^+\}$ has size $<\theta$.
As $J$ is $\kappa$-complete, $E_\eta:=\kappa\setminus\bigcup_{\tau\in \theta\setminus T_\eta}B_\eta^\tau$ is in $J^*$ for every $\eta < \kappa$.
By Lemma~\ref{subnormalcomplete}, we may find $B' \s B$ in $J^+$ such that, for every $(\alpha, \beta) \in [B']^2$, $\beta \in \bigcap _{\eta\le\alpha}E_{\eta}$.

By the choice of $c$,
there is an $\eta<\kappa$ such that the order-type of $c[\{\eta\}\circledast  B']$ is the cardinal $\theta+\theta$. 
Let $\alpha:=\min(B'\setminus(\eta+1))$.
Note that the sets $\{\eta\}\circledast  B'$ and $\{\eta\}\times (B'\setminus(\alpha+1))$ differ on at most one element.
As $\pi$ is a $2$-to-$1$ map from $\theta+\theta$ to $\theta$, $\pi[c[\{\eta\}\times (B'\setminus(\alpha+1))]]$ has order-type $\theta$.

Pick $\tau\in \pi[c[\{\eta\}\times (B'\setminus(\alpha+1))]]\setminus T_\eta$.
Pick $\beta \in B'$ above $\alpha$ such that $\pi(c(\eta, \beta))=\tau$. As $\eta<\alpha<\beta$ and  $(\alpha,\beta)\in[B']^2$, 
it is the case that $\beta \in E_\eta$,
and as $\tau\in\theta\setminus T_\eta$, $E_\eta\cap B_\eta^\tau=\emptyset$.
This is a contradiction.
\end{proof}

\begin{cor}\label{cor228} Suppose that $J$ is a subnormal ideal over $\kappa$ extending $J^{\bd}[\kappa]$.

In any of the following cases, every element of $J^+$ can be split into $\theta$-many disjoint elements of $J^+$:
\begin{enumerate}[(1)]
\item $\theta$ is finite and  $\ubd(J^+,J, \theta)$ holds;
\item $\theta<\kappa$ is infinite, $J$ is $\theta^+$-complete and  $\ubd(J^+,J, \theta)$ holds;
\item $\theta<\kappa$, $J$ is $\kappa$-complete and  $\ubd(J, \theta)$ holds;
\item $J$ is $\kappa$-complete and  $\onto(J, \theta)$ holds.
\end{enumerate}
\end{cor}
\begin{proof}  
(1) and (2) are covered by Proposition~\ref{prop219}(1).

(3) This follows from Proposition~\ref{prop219}(3).

(4) This follows from Proposition~\ref{prop219}(4).
\end{proof}

We end this subsection by commenting on the \emph{uniform} manner in which our principles provide non-weak-saturation results, 
and on the advantage of this aspect.
We focus for concreteness on the principle $\onto^+(J, \theta)$, so let $c:[\kappa]^2\rightarrow \theta$ be a colouring witnessing it. 
For each $\eta< \kappa$ and $\tau<\theta$, define $U_{\eta, \tau}:= \{\beta <\kappa \mid c(\eta, \beta) = \tau\}$. 
The resulting matrix $\langle U_{\eta, \tau} \mid \eta< \kappa, \tau< \theta\rangle$ has the property that for every $B \in J^+$, 
for some $\eta< \kappa$, the row $\vec {U_\eta}=\langle U_{\eta, \tau} \mid \tau < \theta\rangle$ itself serves to shatter $B$ into $\theta$-many disjoint elements in $J^+$. 
In particular, the colouring $c$ guarantees that there is a $\kappa$-list of candidates such that any element of $J^+$ is partitioned into $\theta$-many disjoint pieces, 
all in $J^+$, by at least one of the candidates. 
This is convenient in local-to-global considerations such as in \cite[\S3]{paper38},
where one wants to stabilise on a large set the `code' for a successful local partition.
To compare, when invoking Ulam matrices (see Section~\ref{sectionulam}), 
it is not necessarily the case that every element of the row $\vec {U_\eta}$ has a positive intersection with $B$;
instead, $\theta$ many elements of $\vec {U_\eta}$ have positive intersection with $B$. 
In such a case, the code for a successful partition would have to be larger in that it also records which elements of the row to keep (see also the differences between the two clauses of the forthcoming Theorem~\ref{babycg}).

The possibility of stabilisation arising from the small pool of candidates is an advantage that our colouring principles enjoy over unrestricted processes for obtaining non-weak-saturation results as in \cite{MR0369081,paper29}.
For instance, the approach of \cite{paper29} involves so-called \emph{postprocessing functions}, but the number of such functions is $2^{2^{<\kappa}}$
and there is no uniform representation of those.

However, the uniformity comes at a price, as will be made clear by comparing the positive result in \cite[\S5.2]{MR0369081} with the negative result in Proposition~\ref{prop912} below.

\subsection{Partitioning club-guessing}\label{subsectionpartition}
In this section we mention some applications of our colouring principles to partitioning the club guessing which will appear in the forthcoming \cite{paper46}. 
We state our results in the simplest forms in order to ease the burden of notation on the reader. 
In \cite{paper46}, we consider stronger forms of guessing such as guessing more points, guessing on a fixed set, as well as guessing with respect to ideals different from the bounded ideals.

\begin{defn} Suppose $S\s \kappa$, and $\vec h=\langle h_\delta:C_\delta\rightarrow\kappa \mid \delta \in S\rangle$ is a sequence for which 
$\vec C:=\langle C_\delta\mid\delta\in S\rangle$ is a $C$-sequence.
\begin{itemize}
\item $\vec C$ is \emph{$\xi$-bounded} iff $\otp(C_\delta)\le\xi$ for all $\delta\in S$;
\item $\vec C$ \emph{guesses clubs} iff for every club $D\s \kappa$ there is a $\delta \in S$ such that $\sup(\nacc(C_\delta)\cap D)=\delta$;
\item $\vec h$ \emph{$\theta$-guesses clubs} iff for every club $D\s \kappa$ there is a $\delta \in S$ such that,
for every $\tau<\theta$, $\sup(\{\beta\in\nacc(C_\delta)\cap D\mid h_\delta(\beta)=\tau\})=\delta$.
\end{itemize}
\end{defn}

\begin{fact}[Shelah]\label{clubguessingfact} For regular cardinals $\theta<\theta^+<\kappa$, 
there exists a $\theta$-bounded $C$-sequence $\vec C=\langle C_\delta\mid\delta \in E^\kappa_\theta\rangle$ that guesses clubs
in the following strong sense. For every club $D\s\kappa$, there is $\delta\in E^\kappa_\theta$ such that $C_\delta\s D$.
\end{fact}

\begin{thm}[\cite{paper46}]\label{babycg} Suppose that $\xi< \kappa$ are infinite regular cardinals, and $S\s E^\kappa_\xi$ is stationary
set that carries a $C$-sequence which guesses clubs.
\begin{enumerate}
\item For every cardinal $\theta<\xi$ such that $\ubd(J ^\bd[\xi], \theta)$ holds,
there exists a coloured $C$-sequence $\langle h_\delta:C_\delta\rightarrow\theta\mid \delta \in S\rangle$ 
that $\theta$-guesses clubs;
\item For every cardinal $\theta\le\xi$ such that $\onto(J ^\bd[\xi], \theta)$ holds,
there exists a coloured $C$-sequence $\langle h_\delta:C_\delta\rightarrow\theta\mid \delta \in S\rangle$ 
that $\theta$-guesses clubs.
\end{enumerate}
\end{thm}

Note that when $\theta<\xi$, $\ubd(J ^\bd[\xi], \theta)$ has the exact same powerful effect as its stronger sibling $\onto(J ^\bd[\xi], \theta)$.
In particular, in the case that $\xi=\theta^+$ for $\theta$ singular, while the chain of implications
$$\xi\nrightarrow[\xi]^2_\theta\implies \onto(J ^\bd[\xi], \theta)\implies\ubd(J ^\bd[\xi], \theta)$$
shows that $\ubd(\ldots)$ is a double weakening of the classical partition relation,
it turns out that the result of Theorem A(1) suffices for the desired application.

We end this subsection by stating two corollaries that will appear in \cite{paper46}.
\begin{cor}[\cite{paper46}] Suppose that $\xi < \kappa$ are regular uncountable cardinals.

In any of the following cases, 
there exists a coloured $C$-sequence $\langle h_\delta:C_\delta\rightarrow\theta\mid \delta \in S\rangle$ 
that $\theta$-guesses clubs:	
\begin{enumerate}
\item $\theta=\omega$, and $\xi$ is not ineffable;
\item $\theta<\xi$, and $\square(\xi,{<}\mu)$ holds for some $\mu<\xi$, e.g., if $\L\not\models\xi\text{ is weakly compact}$;
\item $\theta<\xi$, and $\xi$ is not greatly Mahlo;
\item $\theta=\xi$ and $\diamondsuit^*(\xi)$ holds;
\item $\theta=\xi$ and $\diamondsuit(T)$ holds for a stationary $T\s\xi$ that does not reflect at regulars;
\item $\theta=\xi$ is a successor cardinal, and $\stick(\xi)$ holds.
\end{enumerate}
\end{cor}

\begin{cor}[\cite{paper46}]\label{thm223}
Assuming $\non(\mathcal M) = \aleph_1$, 
there exists a coloured $C$-sequence $\langle h_\delta:C_\delta\rightarrow\aleph_1\mid \delta \in E^{\aleph_2}_{\aleph_1}\rangle$ 
satisfying that, for every club $D\s \aleph_2$, there is a $\delta \in E^{\aleph_2}_{\aleph_1}$, such that,
for every $\tau<\aleph_1$, the following set is stationary in $\delta$:
$$\{ \sup(C_\delta\cap\beta)\mid \beta\in\nacc(C_\delta)\cap D\ \&\ h_\delta(\beta)=\tau\}.$$
\end{cor}

\subsection{Kunen's inconsistency theorem}\label{subsectioninconsistency}
We end this section by giving yet another proof of Kunen's celebrated theorem concerning nontrivial elementary embeddings of the universe, using
the colouring principles we have introduced.
Admittedly the proof is highly inefficient, but given the occasion of this paper, we find it appropriate to include this tribute to Kunen.
\begin{thm}[\cite{MR311478}] Suppose that $j:V\rightarrow M$ is a nontrivial elementary embedding. Then $V\neq M$.
\end{thm}
\begin{proof} Let $\theta:=\crit(j)$ and let $\lambda$ denote the first fixed-point of $j$ above $\theta$, 
namely, $\lambda=\sup_{n<\omega}j^n(\theta)$.
Set $\kappa:=\lambda^+$ and $S:=E^\kappa_\omega$. Note that $j(\kappa)=\kappa$ and $j(S)=S$.

Fix a $C$-sequence $\vec C=\langle C_\beta\mid\beta\in S\rangle$
such that each $C_\beta$ is a cofinal subset of $\beta$ of order-type $\omega$.
Evidently, $\vec C$ witnesses that $S\in A_\kappa$ in the sense of Definition~\ref{amenideal}.
Then, by Lemma~\ref{lemma44}, $\ubd(J,\kappa)$ holds for $J:=\ns_{\kappa}\restriction S$.
It then follows from Theorem~\ref{thm42}(2) that $\onto(J^*,J,\theta)$ holds.
In simple words, this gives a colouring $c:[\kappa]^2\rightarrow\theta$
such that for every club $D\s\kappa$ and every stationary $B\s S$,
there exists an $\eta\in D\cap S$ such that $c[\{\eta\}\circledast B]=\theta$.
It follows that, in $M$, $j(c)$ is a colouring from $[\kappa]^2$ to $j(\theta)$ 
such that for every club $D\s\kappa$ and every stationary $B\s S$,
there exists an $\eta\in D\cap S$ such that $c[\{\eta\}\circledast B]=j(\theta)$.

Towards a contradiction, suppose that $V=M$. In particular, $D:=\acc^+(j``\kappa)$ is in $M$.
As $D$ is a club in $\kappa$, $B:=D\cap S$ is stationary in $\kappa$,
and since $\theta<j(\theta)$, we may then find $\eta\in D\cap S=B$ and $\beta\in B$ above $\eta$ such that $j(c)(\eta,\beta)=\theta$.
As $j$ is continuous at ordinals of cofinality $<\theta$,
$D\cap E^\kappa_{<\theta}=j`` E^\kappa_{<\theta}$.
In particular, $(\eta,\beta)\in[D\cap E^\kappa_\omega]^2=[j``E^\kappa_\omega]^2\s[j``\kappa]^2$,
so we may pick $(\bar\eta,\bar\beta)\in[\kappa]^2$ such that $j(\bar\eta)=\eta$ and $j(\bar\beta)=\beta$.
Then $\theta=j(c)(\eta,\beta)=j(c(\bar\eta,\bar\beta))\in j``\theta=\theta$. This is a contradiction.
\end{proof}

\section{Strongly amenable $C$-sequences}
\label{sectionstronglyamenable}
In this section, $\kappa$ denotes a regular uncountable cardinal. 
Our entire focus in this section is the principle $\ubd(J^\bd[S], \kappa)$ for a cofinal $S \s \kappa$.
We start by introducing an ideal on $\kappa$ which we show in Lemma~\ref{lemma43} provides an exact characterisation of this principle. 
We then establish some properties of the ideal and finish by examining the behaviour of this ideal under taking some standard forcing extensions.

\begin{defn} Let $S\s\kappa$. A $C$-sequence $\vec C=\langle C_\beta\mid\beta\in S\rangle$ 
is \emph{strongly amenable in $\kappa$} iff for every club $D$ in $\kappa$,
the set $\{ \beta\in S \mid D\cap\beta\subseteq C_\beta\}$ is bounded in $\kappa$.
\end{defn}

\begin{remark}\label{trivialamenable}	As made clear by the proof of Proposition~\ref{thm215}, 
a $C$-sequence over all of $\kappa$ is strongly amenable in $\kappa$ iff it is \emph{nontrivial} in the sense of \cite[Definition~6.3.1]{TodWalks}.
More generally, a $C$-sequence $\vec C$ over a stationary subset $S$ of $\kappa$
is strongly amenable in $\kappa$ iff $\chi(\vec C)=1$ in the sense of \cite[Definition~4.2]{paper35}.
\end{remark}

\begin{defn} $\sa_\kappa:=\{S\s\kappa\mid S\text{ carries a }C\text{-sequence strongly amenable in }\kappa\}$.
\end{defn}

\begin{lemma}\label{lemma43} For a cofinal subset $S\s\kappa$, the following are equivalent:
\begin{enumerate}[(1)]
\item $S\in\sa_\kappa$;
\item there exists an upper-regressive colouring $c:[\kappa]^2\rightarrow\kappa$ such that, for every cofinal $B\s S$,	
there exists an $\eta<\kappa$ such that, 	for every $\eta'\in[\eta,\kappa)$, 
$$\sup\{ c(\eta',\beta)\mid \beta\in B\setminus(\eta'+1)\}=\kappa;$$
\item $\ubd(J^+, J, \kappa)$ holds for every ideal $J$ over $S$ extending $J^{\bd}[S]$;
\item $\ubd(J^\bd[S],\kappa)$ holds.
\end{enumerate}
\end{lemma}
\begin{proof} $(1)\implies(2)$:	Suppose that $\vec C=\langle C_\beta\mid\beta\in S\rangle$ is a $C$-sequence, strongly amenable in $\kappa$.
Pick an upper-regressive colouring $c:[\kappa]^2\rightarrow\kappa$ such that, for all $\beta\in S$ and $\eta<\beta$, 
$c(\eta,\beta)=\min(C_\beta\setminus\eta)$.

To see that $c$ is as sought, let $B\s S$ be cofinal.
Towards a contradiction, suppose that, for every $\eta<\kappa$, there exist $\eta'\in[\eta,\kappa)$ and $\varsigma_\eta<\kappa$ such that,
$$\sup\{c(\eta',\beta)\mid \beta\in B\setminus(\eta'+1)\}=\varsigma_\eta.$$
Consider the club $D:=\{ \alpha<\kappa\mid \forall\eta<\alpha\,(\max\{\eta',\varsigma_\eta\}<\alpha)\}$.
Let $\epsilon:=\sup\{\beta\in S\mid D\cap\beta\s C_\beta\}$,
and then fix $\beta\in B$ above $\epsilon$.
As $D\cap\beta\nsubseteq C_{\beta}$, let us pick $\alpha\in D\cap\beta\setminus C_{\beta}$.
\begin{claim} Let $\eta<\alpha$. Then $c(\eta',\beta)<\alpha$.
\end{claim}
\begin{why} As $\alpha \in D$, we have that $\eta'<\alpha<\beta$ and that	$c(\eta',\beta)\le\varsigma_\eta<\alpha$.
\end{why}

Thus, for every $\eta<\alpha$, $\eta\le\eta'\le\min(C_\beta\setminus\eta')=c(\eta',\beta)<\alpha$.
This means that $\{\min(C_{\beta}\setminus\eta')\mid \eta<\alpha\}$ is unbounded in $\alpha$, while $\alpha\notin C_{\beta}$,
contradicting the fact that $C_{\beta}$ is closed.

$(2)\implies(3)\implies(4)$: By the definitions, recalling that $\kappa$ is regular.

$(4)\implies(1)$:	
Let $c$ witness $\ubd(J^{\bd}[S],\kappa)$.
By Clauses (1) and (2) of the upcoming Lemma~\ref{stronglyamenableideal},
it suffices to prove that $S':=S\cap \reg(\kappa)\setminus\{\aleph_0\}$ carries a $C$-sequence
strongly amenable in $\kappa$.
For any $\beta\in S'$, as $c$ is upper-regressive, $$C_\beta:=\{\delta<\beta\mid \forall\eta<\delta[c(\eta,\beta)<\delta]\}$$ is a club in $\beta$.

Towards a contradiction, suppose that  $\vec C=\langle C_\beta \mid \beta \in S' \rangle$ is not strongly amenable in $\kappa$.
Fix a club $D\s\kappa$ for which
the set $B:=\{ \beta\in S'\mid D\cap\beta\s C_\beta\}$ is cofinal in $\kappa$.
Pick $\eta<\kappa$ such that $\sup(c[\{\eta\}\circledast B])=\kappa$
and then find $\delta\in D$ above $\eta$.
Then, for every $\beta\in B$ above $\delta$, we get from $\eta<\delta$ and $\delta\in D\cap\beta\s C_\beta$ that $c(\eta,\beta)<\delta$. 
Combined with the regularity of $\kappa$ this gives that $\sup(c[\{\eta\}\circledast B])<\kappa$, contradicting the choice of $\eta$.	
\end{proof}

\begin{cor}\label{cor34} If $\ubd(J^\bd[\kappa],\kappa)$ holds, then so does $\ubd^*([\kappa]^\kappa,J^{\bd}[\kappa],\allowbreak[\kappa]^{\kappa})$.
\end{cor}
\begin{proof} By the implication $(4)\implies(2)$ of Lemma~\ref{lemma43}.
\end{proof}

\begin{defn}[Iterated trace operations]\label{iteratedtrace} For $S \s \kappa$,
for each $\alpha <\kappa^+$, define $\Tr^\alpha(S)$ recursively as follows: 
\begin{itemize}
\item $\Tr^0(S) := S$;
\item $\Tr^{\alpha+1}(S):= \Tr(\Tr^\alpha(S))$;
\item for $\alpha\in\acc(\kappa)$, $\Tr^\alpha(S): = \bigcap_{\alpha' < \alpha}\Tr^\alpha(S)$;
\item for $\kappa\leq \alpha < \kappa^+$ a limit ordinal, let $\pi: \kappa\leftrightarrow \alpha$ be a bijection and then $\Tr^\alpha(S):= \diagonal_{\alpha' < \kappa}\Tr^{\pi(\alpha')}(S)$.
\end{itemize}
\end{defn}
A folklore result is that these operations are well-defined up to equivalence modulo $\ns_\kappa$.

\begin{lemma}\label{stronglyamenableideal} 
\begin{enumerate}[(1)]
\item $[\kappa]^{<\kappa}\s\sa_\kappa$;
\item $\{\beta<\kappa\mid \cf(\beta)<\beta\}\in\sa_\kappa$;
\item $\sa_\kappa$ is a  $\kappa$-complete ideal;
\item $\ns_\kappa \s \sa_\kappa$;
\item $\{\kappa\setminus\Tr(S)\mid S\in(\ns_\kappa)^+\}\s\sa_\kappa$;
\item for every $\alpha< \kappa$, $\kappa \setminus \Tr^\alpha(\kappa) \in \sa_\kappa$;
\item for every $S\in(\sa_\kappa)^+$, $\Tr(S)\in(\sa_\kappa)^+$.
\end{enumerate}
\end{lemma}
\begin{proof} (1) This is obvious.

(2)  Denote $S:=\{\beta<\kappa\mid \cf(\beta)<\beta\}$, and let $\vec C=\langle C_\beta\mid\beta\in S\rangle$ be any $C$-sequence such that $\min(C_\beta) > \otp(C_\beta)=\cf(\beta)$ for all infinite $\beta\in S$.
Towards a contradiction, suppose that $D$ is some club in $\kappa$,
for which $B:=\{\beta\in S\mid D\cap\beta\s C_\beta\}$ is unbounded in $\kappa$.
Pick an infinite $\alpha \in D$, and then pick $\alpha'<\kappa$ such that $\otp(D \cap \alpha') > \alpha$. Finally, pick $\beta\in B\setminus\alpha'$. Then 
$$\min(C_\beta) > \otp(C_\beta) \ge \otp(D\cap \beta) \geq \otp(D\cap \alpha') > \alpha.$$
However, $\alpha \in D\cap \beta\s C_\beta$, which is a contradiction.	

(3) It is clear that $\sa_\kappa$ is downward-closed with respect to inclusion.
To see that $\sa_\kappa$ is $\kappa$-complete, suppose that $\{ S_i\mid i<\sigma\}$ is a given family of elements of $\sa_\kappa$,
with $\sigma<\kappa$. We would like to show that $S:=\bigcup_{i<\sigma}S_i$ is in $\sa_\kappa$.
Since $\sa_\kappa$ is downward-closed, and by possibly replacing
$S_i$ by $S_i\setminus\bigcup_{j<i}S_j$, we may assume that the elements of
$\langle S_i\mid i<\sigma\rangle$  are pairwise disjoint. For each $i<\sigma$, by $S_i\in\sa_\kappa$,
let $\langle C_\beta\mid\beta\in S_i\rangle$ be some $C$-sequence strongly amenable in $\kappa$.
We claim that $\vec C=\langle C_\beta\mid\beta\in S \rangle$ is strongly amenable in $\kappa$.
Indeed, for every club $D\s\kappa$, if $B:=\{\beta\in S\mid D\cap\beta\s C_\beta\}$ is unbounded in $\kappa$,
then for some $i<\sigma$, $B\cap S_i$ is unbounded in $\kappa$, contradicting the fact that $\langle C_\beta\mid\beta\in S_i\rangle$ is strongly amenable in $\kappa$.	

(4) Fix $S\in\ns_\kappa$, and we shall show that $S\in\sa_\kappa$. 
By Clause~(2), we may assume that $S\s\reg(\kappa)$.
As $S\in\ns_\kappa$, fix a club $E \s \kappa$ disjoint from $S$. 
In particular, $S\cap\acc(E)=\emptyset$, so for every $\beta \in S$ we may fix a club $C_\beta$  in $\beta$ disjoint from $E$. 
We claim that $\vec C=\langle C_\beta \mid \beta \in S\rangle$ is strongly amenable in $\kappa$. Towards a contradiction, suppose that this is not so, as witnessed by a club $D \s \kappa$. 
That is, suppose that the set $B:= \{\beta \in S \mid D\cap \beta \s C_\beta\}$ is unbounded in $\kappa$. 
Then we can find $\beta \in B$ above $\min(D\cap E)$. Then $D\cap E\cap \beta \s D\cap \beta \s C_\beta$ and all three of these sets are nonempty. 
This implies that $E\cap C_\beta$ is nonempty, which is a contradiction.

(5) Let $S\s\kappa$ be stationary. By Clause~(2), to see that $\kappa\setminus\Tr(S)$ is in $\sa_\kappa$, 
it suffices to show that $T:=\reg(\kappa)\setminus\Tr(S)$ is in $\sa_\kappa$.
Fix a $C$-sequence $\vec C=\langle C_\beta\mid \beta\in T\rangle$ such that each $C_\beta$ is a club in $\beta$ disjoint from $S$.
Now, if there exists a club $D$ such that $\{\beta\in T\mid D\cap\beta\s C_\beta\}$ is unbounded in $\kappa$,
then we may pick $\beta$ in this set above $\min(D\cap S)$. This is a contradiction.

(6) Using $\kappa$-completeness of the ideal and Clause~(5).

(7) This follows from Clause~(5).
\end{proof}

\begin{cor}\label{square_is_amenable} If $\square(\kappa,{<}\mu)$ holds with $\mu<\kappa$,
then $\ubd(J^\bd[\kappa],\kappa)$ holds.

In particular, if $\ubd(J^\bd[\kappa],\kappa)$ fails, then $\kappa$ is weakly compact in $\L$.
\end{cor}
\begin{proof} By \cite[Theorem~2.13]{MR3730566}, if $\square(\kappa,{<}\mu)$ holds with $\mu<\kappa$,
then there exists a family of ${<}\mu$ many stationary sets that do not reflect simultaneously.
\end{proof}

\begin{remark} As made clear by the proof of \cite[Lemma~1.23]{paper29}, in fact,
any transversal for a $\square(\kappa,{<}\mu)$-sequence, with $\mu<\kappa$, is strongly amenable in $\kappa$.
\end{remark}

\begin{cor} \label{strongamenmahlo}
If $\ubd(J^\bd[\kappa],\kappa)$ fails, then $\refl(\kappa,\kappa,\reg(\kappa))$ holds and hence $\kappa$ is greatly Mahlo.
\end{cor}
\begin{proof}
By Lemma~\ref{lemma43} and Remark~\ref{trivialamenable},
if $\ubd(J^\bd[\kappa],\kappa)$ fails, then $\kappa$ does not carry a nontrivial $C$-sequence.
By \cite[Lemma~2.12]{paper35}, if $\kappa$ does not carry a nontrivial $C$-sequence, then $\refl(\kappa,\kappa,\reg(\kappa))$ holds
and hence $\kappa$ is greatly Mahlo.
\end{proof}

\begin{cor} \label{amensubset} For every stationary $S\s\kappa$, there exists a stationary $S'\s S$ with $S'\in\sa_\kappa$.
\end{cor}
\begin{proof} Let $S\s\kappa$ be stationary. By Lemma~\ref{stronglyamenableideal}(2), it suffices to consider the case where $S$ consists of regular cardinals. 
By Lemma~\ref{stronglyamenableideal}(5), it then suffices to prove that $S':=S\setminus\Tr(S)$ is stationary, but this is standard.
\end{proof}

In Corollary~\ref{wcvl} below, we shall show that, assuming $\V=\L$,
for every regular uncountable cardinal $\kappa$, $\kappa\in\sa_\kappa$ iff $\kappa$ is not weakly compact.
However, this equivalence does not hold in general: Corollary~\ref{cor216}(1) gives 
the consistency of a strongly inaccessible cardinal $\kappa$ that does not carry a strongly amenable $C$-sequence, and yet $\kappa$ is not weakly compact. 
Furthermore, and in contrast to Corollary~\ref{strongamenmahlo},
we shall soon show that $\kappa\notin\sa_\kappa$ does not even imply that $\kappa$ is strongly inaccessible.

\begin{prop}\label{strongamenpullneg} Let $\mathbb P$ be a $\theta$-cc poset, and $S\s\kappa$.
\begin{enumerate}[(1)]
\item If $\theta\le\kappa$ and $V^{\mathbb P}\models S\notin \sa_\kappa$, then $S \notin \sa_\kappa$;
\item If $\theta<\kappa$ and 	$V^{\mathbb P}\models S\in \sa_\kappa$, then $S \in \sa_\kappa$. 
\end{enumerate}
\end{prop}
\begin{proof} (1)	The proof is exactly the same as that of Corollary~\ref{cor216}. Namely, by the $\kappa$-cc of $\mathbb P$, any club $D \s \kappa$ in $V^{\mathbb P}$ contains a club $D'$ which is in $V$. The result follows.

(2) Suppose $\theta<\kappa$ and $V^{\mathbb P}\models S\in \sa_\kappa$; we shall show that $S \in \sa_\kappa$. 
By Corollary~\ref{strongamenmahlo}, we can assume that $\kappa$ is weakly Mahlo. 
By Lemma~\ref{stronglyamenableideal}, we can moreover assume that $S$ consists of regular uncountable cardinals greater than $\theta$.
By our hypothesis, we may fix in $V^{\mathbb P}$ a $C$-sequence $\vec C=\langle C_\beta \mid \beta \in S\rangle$ which is strongly amenable in $\kappa$. 
As $\mathbb P$ has the $\theta$-cc,
and $S$ consists of regular uncountable cardinals greater than $\theta$,
there is, in $V$, a $C$-sequence $\vec c=\langle c_\beta \mid \beta \in S\rangle$ such that $c_\beta \s C_\beta$ for each $\beta\in S$.
Work in $V$. Towards a contradiction, suppose that $\vec c$ is not a strongly amenable in $\kappa$.
Then there is some club $D \s \kappa$ such that the set $\{\beta \in S \mid D \cap \beta \s c_\beta\}$ is unbounded in $\kappa$. 
Now notice that, in $V ^{\mathbb P}$,
$$\{\beta \in S \mid D \cap \beta \s c_\beta\}\s \{\beta \in S \mid D \cap \beta \s C_\beta\}$$
and the former set is unbounded in $\kappa$ whereas the latter set is bounded in $\kappa$ by virtue of $\vec C$ being strongly amenable in $\kappa$. This is a contradiction.
\end{proof}
\begin{remark}\label{forcingsquare} Assuming the consistency of a weakly compact (resp.~ineffable) cardinal, 
it is consistent that $\kappa \in \sa_\kappa$ holds and yet
$\kappa$ is weakly compact (resp.~ineffable) in $\L$.
This is obtained by simply forcing over $\L$ to add a $\square(\kappa)$ sequence (see \cite[pp.~686]{MR3129734}).
\end{remark}

\begin{cor} Assuming the consistency of a weakly compact cardinal, it is consistent that $\kappa=2^{\aleph_0}$, and $\kappa\notin\sa_\kappa$.
\end{cor}
\begin{proof} This follows from \cite[Corollary~4.9]{paper35} using Remark~\ref{trivialamenable}, 
but here is a short direct proof. By Proposition~\ref{thm215}, if $\kappa$ is weakly compact, then $\kappa \notin \sa_\kappa$.
Now start with $\kappa$ weakly compact and force to add $\kappa$ many Cohen reals. 
As the forcing poset is $ccc$, by Proposition~\ref{strongamenpullneg}(1) we must have that $\kappa \notin \sa_\kappa$ in the forcing extension as well. 
\end{proof}

\section{Amenable $C$-sequences}\label{sectionamenable}
In this section, $\kappa$ denotes a regular uncountable cardinal. 
Our entire focus in this section is the principle $\ubd(\ns_\kappa \restriction S, \kappa)$ for a stationary $S \s \kappa$. 
We start by introducing an ideal on $\kappa$ from \cite{paper29} which we show in Lemma~\ref{lemma44} provides an exact characterisation of this principle. 
We then establish some properties of the ideal and finish by examining the behaviour of this ideal under taking some standard forcing extensions 
and stating a conjecture which pertains to the only point of asymmetry between the results of this section and those of Section~\ref{sectionstronglyamenable}.

\begin{defn}[\cite{paper29}]\label{amenable} Let $S\s\kappa$. A $C$-sequence $\vec C=\langle C_\beta\mid\beta\in S\rangle$ is \emph{amenable in $\kappa$}
iff for every club $D\s\kappa$, the set $\{ \beta\in S \mid D\cap\beta\s C_\beta\}$ is nonstationary in $\kappa$.
\end{defn}
The notion of an amenable $C$-sequence is implicit in the elementary proof (see \cite[Theorem~8.10]{MR1940513}) of Solovay's theorem
that any stationary subset of $\kappa$ may be decomposed into $\kappa$-many stationary sets. 
It was made explicit in \cite{paper29}, 
in proving that, assuming $\square(\kappa)$,
every fat subset of $\kappa$ may be decomposed into $\kappa$-many fat sets.

\begin{fact}[\cite{paper29}]\label{amenablefact} For a $C$-sequence $\vec C=\langle C_\beta\mid\beta\in S\rangle$ over a stationary subset $S\s\acc(\kappa)$,
the following are equivalent:
\begin{enumerate}[(1)]
\item $\vec C$ is amenable;
\item for every club $D\s\kappa$, the set $\{ \beta\in S \mid \sup(D\cap\beta\setminus C_\beta)<\beta\}$ is nonstationary in $\kappa$;
\item for every cofinal $A\s\kappa$, the set $\{\beta\in S\mid A\cap\beta\s C_\beta\}$ is nonstationary.
\end{enumerate}
\end{fact}

\begin{defn}\label{amenideal} $\amen_\kappa:=\{S\s\kappa\mid S\text{ carries a }C\text{-sequence amenable in }\kappa\}$.
\end{defn}

\begin{lemma}\label{lemma44} For a stationary subset $S\s\kappa$, the following are equivalent:
\begin{enumerate}[(1)]
\item $S\in\amen_\kappa$;
\item there exists an upper-regressive colouring $c:[\kappa]^2\rightarrow\kappa$ such that,
for every normal ideal $J$ over $S$ extending $J^{\bd}[S]$, 
for every $B\in J^+$,
there exists $\eta<\kappa$, such that,
for every $\eta'\in[\eta,\kappa)$, 
$$\sup\{ \tau<\kappa\mid \{ \beta\in B\mid c(\eta',\beta)=\tau\}\in J^+\}=\kappa;$$
\item $\ubd([\kappa]^\kappa,J,\kappa)$ holds for every normal ideal $J$ over $S$, extending $J^{\bd}[S]$;
\item $\ubd(\ns_\kappa\restriction S,\kappa)$ holds.
\end{enumerate}
\end{lemma}
\begin{proof} $(1)\implies(2)$ Suppose that $\vec C=\langle C_\beta\mid\beta\in S\rangle$ is a $C$-sequence, amenable in $\kappa$.
Pick an upper-regressive colouring $c:[\kappa]^2\rightarrow\kappa$ such that, for all $\beta\in S\cap \acc(\kappa)$ and $\eta<\beta$, $c(\eta,\beta):=\min(C_\beta\setminus\eta)$.

To see that $c$ is as sought, let $J$ be a normal ideal over $S$ extending $J^{\bd}[S]$, and let $B\in J^+$.
Towards a contradiction, suppose that, for every $\eta<\kappa$, there exist $\eta'\in[\eta,\kappa)$ and $\varsigma_\eta<\kappa$ such that,
for every $\tau\in\kappa\setminus\varsigma_\eta$, the set
$$E_{\eta,\tau}:=\{\beta\in S\mid\beta\notin B\text{ or } c(\eta',\beta)\neq\tau\}$$
is in $J^*$.
For every $\tau\in\varsigma_\eta$, let $E_{\eta,\tau}:=S$.
Since $J$ is normal, $E:=\diagonal_{\eta<\kappa}\diagonal_{\tau<\kappa}E_{\eta,\tau}$ is in $J^*$.
Note that $E=\{\beta\in S\mid \forall\eta<\beta\forall\tau\in\beta\setminus\varsigma_\eta\,(\beta\in E_{\eta,\tau})\}$.

Consider the club $D:=\{\alpha<\kappa\mid\forall \eta<\alpha\,(\max\{\eta',\varsigma_\eta\}<\alpha)\}$.
As $J$ extends $J^{\bd}[S]$, Fact~\ref{normalfacts} implies that $B\cap E$ is stationary.
So, by the amenability of $\vec C$ , let us fix $\beta\in B\cap E\cap\acc(\kappa)$ with $D\cap\beta\nsubseteq C_{\beta}$.
Then pick $\alpha\in D\cap\beta\setminus C_{\beta}$.
\begin{claim} Let $\eta<\alpha$. Then $c(\eta',\beta)<\alpha$.
\end{claim}
\begin{why} As $c$ is upper-regressive, $\tau:=c(\eta',\beta)$ is less than $\beta$.
As $\alpha\in D$ and  $\eta<\alpha$, we have that $\eta \leq \eta'<\alpha<\beta$. If $\tau\in\beta\setminus\varsigma_\eta$, then 	$\beta\in B\cap E_{\eta,\tau}$, which is an absurdity.
Altogether, $\tau<\varsigma_\eta$. As $\eta<\alpha$ and $\alpha\in D$, furthermore, $c(\eta',\beta)=\tau<\varsigma_\eta<\alpha$.
\end{why}

Thus, for every $\eta<\alpha$, $\eta\le\eta'\le\min(C_\beta\setminus\eta')=c(\eta',\beta)<\alpha$.
This means that $\{\min(C_{\beta}\setminus\eta')\mid \eta<\alpha\}$ is unbounded in $\alpha$, while $\alpha\notin C_{\beta}$,
contradicting the fact that $C_{\beta}$ is closed.

$(2)\implies(3)\implies(4)$: This is immediate.	

$(4)\implies(1)$: By Lemma~\ref{stronglyamenableideal}  and Clause~(1) of the upcoming Lemma~\ref{amenableideal},
it suffices to prove that $S':=S\cap\reg(\kappa)\setminus\{\aleph_0\}$ carries an amenable $C$-sequence.
Let $c$ witness $\ubd(\ns_\kappa\restriction S,\kappa)$.
For any $\beta\in S'$, as $c$ is upper-regressive, $C_\beta:=\{\delta<\beta\mid \forall\eta<\delta[c(\eta,\beta)<\delta]\}$ is a club in $\beta$.
Towards a contradiction, suppose that $\vec C=\langle C_\beta \mid \beta \in S' \rangle$ is not amenable in $\kappa$.
Fix a club $D\s\kappa$ for which
the set $B:=\{ \beta\in S'\mid D\cap\beta\s C_\beta\}$ is stationary.
Pick $\eta<\kappa$ such that $\sup(c[\{\eta\}\circledast B])=\kappa$.
Fix $\delta\in D$ above $\eta$.
Then, for every $\beta\in B$ above $\delta$, we get from $\delta\in C_\beta\setminus(\eta+1)$ that $c(\eta,\beta)<\delta$. 
Since $\kappa$ is regular, we conclude that $\sup(c[\{\eta\}\circledast B])<\kappa$, contradicting the choice of $\eta$.
\end{proof}

\begin{cor} If $\ubd(\ns_\kappa,\kappa)$ holds, then so do $\ubd^*([\kappa]^\kappa,\ns_\kappa,\allowbreak[\kappa]^{\kappa})$
and $\ubd^+([\kappa]^\kappa,\ns_\kappa,\kappa)$.\qed
\end{cor}
\begin{remark}
The preceding cannot be improved any further since in Proposition~\ref{prop912} below, 
we shall show that $\ubd^{++}(\ns_\kappa,\kappa)$ must fail.
\end{remark}

\begin{defn}[\cite{paper46}]\label{cgstar} For stationary subsets $S,T$ of $\kappa$,
$\cg^*(S,T)$ asserts the existence of a $C$-sequence, $\vec C=\langle C_\delta\mid\delta \in S\rangle$ such that,
for every club $E\s\kappa$, there are club many $\delta \in S$ such that,
$\sup(\nacc(C_\delta)\cap E\cap T)=\delta$.
\end{defn}
\begin{lemma}\label{amenableideal} 
\begin{enumerate}[(1)]
\item $\sa_\kappa\s\amen_\kappa$;
\item $\amen_\kappa$ is a  $\kappa$-complete normal ideal;
\item for every $\alpha< \kappa^+$, $\kappa \setminus \Tr^\alpha(\kappa) \in \amen_\kappa$;
\item $\{S\in(\ns_\kappa)^+\mid \cg^*(S,\kappa)\}\s \amen_\kappa$;
\item for every $S\in(\amen_\kappa)^+$, $\Tr(S)\in(\amen_\kappa)^+$.
\end{enumerate}
\end{lemma}
\begin{proof} (1) This is obvious.

(2) It is clear that $\amen_\kappa$ is downward-closed with respect to inclusion.
To see that $\amen_\kappa$ is $\kappa$-complete, suppose that $\{ S_i\mid i<\sigma\}$ is a given family of elements of $\amen_\kappa$,
with $\sigma<\kappa$. We would like to show that $S:=\bigcup_{i<\sigma}S_i$ is in $\amen_\kappa$.
Since $\amen_\kappa$ is downward-closed, we may assume that the elements of
$\langle S_i\mid i<\sigma\rangle$  are pairwise disjoint. For each $i<\sigma$, since $S_i\in\amen_\kappa$,
let $\langle C_\beta\mid\beta\in S_i\rangle$ be some $C$-sequence amenable in $\kappa$.
We claim that $\vec C=\langle C_\beta\mid\beta\in S \rangle$ is amenable in $\kappa$.
Indeed, for every club $D\s\kappa$, if $T:=\{\beta\in S\mid D\cap\beta\s C_\beta\}$ is stationary,
then for some $i<\sigma$, $T\cap S_i$ is stationary, contradicting the fact that $\langle C_\beta\mid\beta\in S_i\rangle$ is amenable in $\kappa$.

Next, to see that $\amen_\kappa$ is normal, suppose that
$\langle S_i\mid i<\kappa\rangle$ is a $\kappa$-sequence consisting of elements of $\amen_\kappa$.
We would like to show that the following set is in $\amen_\kappa$:
$$S:=\{\beta<\kappa\mid \exists i<\beta(\beta\in S_i)\}.$$ 
By the preceding findings, it suffices to prove that
$S':=S\cap \reg(\kappa)\setminus\{\aleph_0\}$ is in $\amen_\kappa$.

For each $i<\kappa$, fix a $C$-sequence $\langle C^i_\beta\mid \beta\in S_i\rangle$ amenable in $\kappa$.
Let $\beta\in S'$ be arbitrary. Then $\beta$ is a regular uncountable cardinal,
and $\{ C_\beta^i\mid  i<\beta\ \&\ \beta\in S_i\}$ consists of at most $\beta$ many clubs in $\beta$, 
so that by using a diagonal intersection or any other mean, we may find a club $C_\beta$ in $\beta$
such that, for every $i<\beta$ with $\beta\in S_i$,	
$C_\beta\setminus C_\beta^i$ is bounded in $\beta$.
We claim that $\vec C=\langle C_\beta\mid\beta\in S \rangle$ is amenable in $\kappa$.
Towards a contradiction, suppose that $D$ is a club in $\kappa$ for which $T:=\{\beta\in S'\mid D\cap\beta\s C_\beta\}$ is stationary in $\kappa$. 
As $S'\s S$ and by Fodor's lemma, let us fix some $i<\kappa$ such that $T\cap S_i$ is stationary. 
By another application of Fodor's lemma, let us fix some $\epsilon<\kappa$
such that $T_{i,\epsilon}:=\{\beta\in T\cap S_i\mid \sup(C_\beta\setminus C_\beta^i)=\epsilon\}$ is stationary.
Now, consider the club $D':=D\setminus(\epsilon+1)$.
We claim that $\{\beta\in S_i\mid D'\cap\beta\s C_\beta^i\}$ is stationary in $\kappa$,
contradicting the fact that $\langle C_\beta^i\mid \beta\in S_i\rangle$ is amenable in $\kappa$.
To see this, let $\beta$ be an arbitrary element of the stationary set $T_{i,\epsilon}$.
Then $D\cap\beta\s C_\beta\s C_\beta^i\cup(\epsilon+1)$,
so that $D'\cap\beta\s C_\beta^i$.

(3) Since $\kappa\setminus\Tr(\kappa) \in \sa_\kappa \s \amen_\kappa$ and because $\amen_\kappa$ is normal and $\kappa$-complete.

(4) Suppose $\vec C$ witnesses $\cg^*(S,\kappa)$. Given a club $D\s\kappa$, let $E:=\acc(D)$.
By the choice of $\vec C$, $B:=\{\beta\in S\mid \sup(\nacc(C_\beta)\cap E)=\beta\}$ covers a club relative to $S$.
For every $\beta\in B$ and $\epsilon\in\nacc(C_\beta)\cap E$, as $\epsilon\in\acc(D)$, 
we may find some $\delta_\epsilon\in D$ with $\sup(C_\beta\cap\epsilon)<\delta_\epsilon<\epsilon$,
and then $\{ \delta_\epsilon\mid \epsilon\in\nacc(C_\beta)\cap E\}$ is a subset of $D\cap\beta\setminus C_\beta$ which is cofinal in $\beta$.
So $\{ \beta\in S \mid \sup(D\cap\beta\setminus C_\beta)<\beta\}$ is disjoint from $B$,
and we are done, recalling Fact~\ref{amenablefact}.

(5) Let $S\in(\amen_\kappa)^+$. 
By Clause~(1) and Lemma~\ref{stronglyamenableideal}(5), $\kappa\setminus\Tr(S)\in\sa_\kappa\s\amen_\kappa$.
So if $\Tr(S)\in \amen_\kappa$, then $\kappa\in\amen_\kappa$,
contradicting the fact that $S\s \kappa$ with $S\notin\amen_\kappa$.
\end{proof}

\begin{prop}\label{diamondstar} For $S\s\kappa$, if $S\cap\reg(\kappa)$ is nonstationary,
or if $\diamondsuit^*(S\cap\reg(\kappa))$ holds, then $S\in\amen_\kappa$.
\end{prop}
\begin{proof} Let $S\s\kappa$. Set $S':=S\cap\reg(\kappa)$. By Lemmas \ref{stronglyamenableideal}(2) and \ref{amenableideal}(2), $S\setminus S'$ is in $\amen_\kappa$,
thus it remains to see whether $S'$ belongs to $\amen_\kappa$.
By Lemmas \ref{stronglyamenableideal}(4) and \ref{amenableideal}(2), if $S'$ is nonstationary, then $S'\in\amen_\kappa$.
Next, suppose that $S'$ is stationary and moreover $\diamondsuit^*(S')$ holds.
Fix a witnessing sequence $\langle\mathcal A_\beta \mid \beta \in S'\rangle$.
This means that each $\mathcal A_\beta$ is a collection of no more than $\beta$ many subsets of $\beta$,
and, for every subset $A$ of $\kappa$, for club many $\beta\in S'$, $A\cap\beta\in\mathcal A_\beta$.
Now, for every $\beta\in S'$, by taking a diagonal intersection, we can fix a club $C_\beta$ in $\beta$ such that for every $D \in \mathcal A_\beta$ which is a club in $\beta$, $C_\beta\s ^*D$. 
Consequently, $\langle C_\beta\mid \beta\in S'\rangle$ is a $\cg^*(S',\kappa)$-sequence, and then Lemma~\ref{amenableideal}(4) implies that $S'\in\amen_\kappa$.
\end{proof}
	
\begin{cor}\label{amenableremark} If $\V=\L$, then $\amen_\kappa=\{ S \s \kappa\mid S\text{ is not  ineffable}\}$.
\end{cor}
\begin{proof} It is clear from the definitions that for any set $S \s \kappa$, if $S$ is in $\amen_\kappa$ then it is not ineffable. 
For the other direction, suppose that $S \s \kappa$ is not ineffable. 
By Lemmas \ref{stronglyamenableideal}(4) and \ref{amenableideal}(1), we may assume that $S$ is stationary.
By a theorem of Jensen (see \cite[Theorem 5.39]{MR3243739}), $\V=\L$ implies that $\diamondsuit^*$ holds over every stationary non-ineffable subset of $\kappa$.
So $\diamondsuit^*(S)$ holds, and then Proposition~\ref{diamondstar} implies that $S\in\amen_\kappa$.
\end{proof}

The characterisation of $A_\kappa$ under $\V=\L$ is not true in general,
as established by Corollary~\ref{cor216}(2).

\begin{prop}\label{amenpullneg}\label{amenpullpos}
Let $\mathbb P$ be a $\theta$-cc poset, and $S\s\kappa$.
\begin{enumerate}[(1)]
\item If $\theta\le\kappa$ and $V^{\mathbb P}\models S\notin \amen_\kappa$, then $S \notin \amen_\kappa$;
\item If $\theta<\kappa$ and $V^{\mathbb P}\models S\in \amen_\kappa$, then $S \in \amen_\kappa$. 
\end{enumerate}
\end{prop}
\begin{proof}
(1)	The proof is exactly the same as that of Proposition~\ref{strongamenpullneg}(1).

(2) 	The proof is almost exactly the same as that of Proposition~\ref{strongamenpullneg}(2) using only Lemma~\ref{amenableideal}(1) and  the extra fact that as $\mathbb P$ is in particular $\kappa$-cc, stationary subsets of $\kappa$ remain stationary after forcing with $\mathbb P$.
\end{proof}

\begin{cor}\label{amenweakinacc} Assuming the consistency of an ineffable cardinal, it is consistent that $\kappa=2^{\aleph_0}$ and $\kappa\notin\amen_\kappa$.
\end{cor}
\begin{proof} As we have already mentioned, if $\kappa$ is ineffable then it follows purely from the definitions that $\kappa \notin \amen_\kappa$. 
Now start with $\kappa$ ineffable and force to add $\kappa$ many Cohen reals. 
As the forcing poset is $ccc$, by Proposition~\ref{amenpullpos}(1) we must have that $\kappa \notin \amen_\kappa$ in the forcing extension as well. 
\end{proof}

We end this section with a conjecture in the spirit of Corollary~\ref{square_is_amenable}.
\begin{conj} \label{conj412}If $\kappa$ is a regular uncountable cardinal and $\kappa\notin\amen_\kappa$, then $\kappa$ is ineffable in $\L$.
\end{conj}

\section{\texorpdfstring{$\ubd^*$}{unbounded$^*$} and Ulam-type matrices}\label{sectionulam}

In this section, $\kappa$ denotes a regular uncountable cardinal.

\medskip

An \emph{Ulam matrix} is a useful tool in proving that various ideals are not weakly saturated.
It was introduced in \cite{ulam1930masstheorie} by Ulam, who proved that every successor cardinal $\kappa=\theta^+$ admits an Ulam matrix,
that is, a matrix $\langle U_{\eta,\tau}\mid \eta<\theta, \tau<\theta^+\rangle$
such that:
\begin{itemize}
\item For every $\eta<\theta$, $\langle U_{\eta, \tau}\mid \tau<\theta^+\rangle$ consists of pairwise disjoint subsets of $\theta^+$;
\item For every $\tau<\theta^+$, $|\theta^+\setminus \bigcup_{\eta<\theta}U_{\eta,\tau}|\le\theta$.
\end{itemize}
In our language, Ulam's theorem is as follows:	
\begin{fact}[Ulam, \cite{ulam1930masstheorie}]\label{ulamoriginal} For every infinite cardinal $\theta$, $\ubd^*(\{\theta\},J^{\bd}[\theta^+],\allowbreak\{\theta^+\setminus\theta\})$ holds.
\end{fact}

In \cite{MR260597}, Hajnal formulated a variation called \emph{triangular Ulam matrix}
which is also applicable to inaccessible cardinals, and has similar non-saturation consequences.
In this section, we shall show that this finer concept of Hajnal is also captured by the principle $\ubd^*(\ldots)$,
and that some standard combinatorial hypotheses give rise to useful instances of $\ubd^*(\ldots)$ 
weaker than those corresponding to Ulam matrices. We finish with some applications of the $\ubd^*(\ldots)$ principles.

\begin{defn} An \emph{Ulam-type matrix}  for a cardinal $\kappa$ is a triangular matrix $\langle U_{\eta, \tau}\mid \eta<\tau<\kappa\rangle$ 
satisfying that, for every $\eta<\kappa$, $\langle U_{\eta, \tau}\mid \eta<\tau<\kappa\rangle$ consists of pairwise disjoint subsets of $\kappa$.
\end{defn}

\begin{defn}[Hajnal, \cite{MR260597}]\label{ulam} A cardinal $\kappa$ is said to admit
a \emph{triangular Ulam matrix} iff there exists an Ulam-type matrix 
$\langle U_{\eta,\tau}\mid \eta<\tau<\kappa\rangle$  for which the set
$T:=\{\tau<\kappa\mid  |\kappa\setminus(\bigcup_{\eta<\tau}U_{\eta,\tau})|<\kappa\}$
(which we call \emph{the support}) is stationary in $\kappa$.
\end{defn}

The core technical result of \cite{MR260597} reads as follows.
\begin{fact}[Hajnal, \cite{MR260597}]\label{hajnalulam} Suppose that $T\s\kappa$ is stationary, but $\Tr(T)\cap\reg(\kappa)$ is nonstationary.
Then there is a club $C\s\kappa$ such that, for every $\beta\in C$, there is a function $f_\beta : T\cap C\cap \beta\rightarrow \beta$ which is regressive and injective.
\end{fact}

\begin{lemma}\label{hajnaltocolouring} For any stationary $T\s\kappa$, the following are equivalent:
\begin{enumerate}[(1)]
\item $\ubd^*(J^\bd[\kappa],\{ T\})$ holds;
\item $\ubd^*(J,\{T\})$ holds for some normal ideal $J$ over $\kappa$ extending $J^\bd[\kappa]$;
\item there is a stationary $S\s \kappa$ such that, for every $\beta\in S$, there is a function $f_\beta : T\cap \beta\rightarrow \beta$ which is regressive and injective;
\item there is a cofinal $A\s \kappa$ such that, for every $\beta\in A$, there is a function $f_\beta : T\cap \beta\rightarrow \beta$ which is regressive and injective;
\item $\kappa$ carries a triangular Ulam matrix with support $T$.
\end{enumerate}
\end{lemma}
\begin{proof} $(1)\implies(2)$: Just take $J:=\ns_\kappa$.

$(2)\implies(3)$ Suppose that $c$ witnesses $\ubd^*(J,\{ T\})$
for a normal ideal $J$ extending $J^\bd[\kappa]$.
In particular, for every $\tau\in T$, the set $N_\tau:=\{\beta<\kappa\mid \forall\eta<\tau\,(c(\eta,\beta)\neq\tau)\}$ is in $J$. 
As $J$ is normal, $S := \diagonal_{\tau < \kappa}(\kappa\setminus N_\tau)$ is in its dual filter. 
By Fact~\ref{normalfacts},  $J$ extends $\ns_\kappa$, and hence $S$ is stationary in $\kappa$.
Now, let $\beta\in S$.
Then, for every $\tau \in T \cap \beta$, $\beta \notin N_\tau$ and so there is an $\eta < \tau$ such that $c(\eta, \beta) = \tau$. 
So, for every $\beta \in S$, the function $f_\beta : T \cap \beta \rightarrow \beta$ given by $f_\beta(\tau) := \min\{\eta < \tau \mid  c(\eta, \beta) = \tau\}$ is well-defined and regressive. It is also clearly injective.

$(3) \implies (4)$: This is trivial.

$(4)\implies(5)$: Given a sequence $\langle f_\beta\mid \beta\in A\rangle$ as above, define a matrix $\langle U_{\eta,\tau}\mid \eta<\tau<\kappa\rangle$ via
$$U_{\eta,\tau}:=\{\beta<\kappa\mid f_{\min(A\setminus\beta)}(\tau)=\eta\}.$$
We leave the verification to the reader.

$(5) \implies (1)$: Suppose that $\langle U_{\eta,\tau} \mid \eta < \tau< \kappa  \rangle$ is a triangular Ulam matrix with support $T$.
Let $c:[\kappa]^2\rightarrow \kappa$ be any upper-regressive colouring satisfying that for $\eta< \beta< \kappa$, if there is a $\tau$ such that $\eta <\tau< \beta$ and $\beta \in U_{\eta, \tau}$ then $c(\eta, \beta) = \tau$. Note that as $\langle U_{\eta,\tau} \mid \eta < \tau< \kappa  \rangle$ is an Ulam-type matrix, this is well-defined.
Now, suppose that $B \in (J^\bd[\kappa])^+$ is given, and let $\tau\in T$ be arbitrary.
As $B \setminus (\bigcup_{\eta< \tau}U_{\eta, \tau})$ has size less than $\kappa$,
we may pick a $\beta \in (\bigcup_{\eta< \tau}U_{\eta, \tau})\setminus (\tau+1)$. Then if $\eta< \tau$ is such that $\beta \in U_{\eta, \tau}$ then $c(\eta, \beta) = \tau$.	
\end{proof}

\begin{cor}\label{hajnaltocolouring2} For a stationary $T\s\kappa$, the following are equivalent:
\begin{enumerate}
\item $\Tr(T)\cap\reg(\kappa)$ is nonstationary;
\item There exists a club $C\s\kappa$ for which $\ubd^*(J^\bd[\kappa],\{ T\cap C\})$ holds.
\end{enumerate}
\end{cor}
\begin{proof} $(i)\implies(ii)$: By Fact~\ref{hajnalulam}, we may fix a club $C\s\kappa$ such that, 
for every $\beta\in C$, there is a function $f_\beta : T\cap C\cap \beta\rightarrow \beta$ which is regressive and injective.
Now, appeal with $C$ and $T\cap C$ to the implication $(3)\implies(1)$ of Lemma~\ref{hajnaltocolouring}.

$(ii)\implies(i)$: Write $T':=T\cap C$, and let $c$ witness $\ubd^*(J^\bd[\kappa],\{ T'\})$.
In particular, for every $\tau\in T'$, the set $\{\beta<\kappa\mid \forall\eta<\tau\,(c(\eta,\beta)\neq\tau)\}$ is bounded in $\kappa$.
Thus, we may define a function $f:T'\rightarrow\kappa$ via $f(\tau):=\sup\{\beta<\kappa\mid \forall\eta<\tau\,(c(\eta,\beta)\neq\tau)\}$.
Suppose now that  $\Tr(T)\cap\reg(\kappa)$ is a stationary subset of $\kappa$. 
In particular, $\Tr(T')\cap\reg(\kappa)$ is stationary. 
Then we can pick $\mu\in \reg(\kappa)\cap\Tr(T')$ such that $f[\mu]\s\mu$. So, the set $\bar T:=T'\cap\mu$ is a stationary subset of $\mu$.
As $f[\bar T]\s\mu$, we may define a function $g:\bar T\rightarrow\mu$ via $g(\tau):=\min\{\eta<\tau\mid c(\eta,\mu)=\tau\}$.
As $g$ is regressive, there must exist $\tau\neq\tau'$ in $\bar T$ and $\eta<\kappa$ such that $g(\tau)=\eta=g(\tau')$. So $\tau=c(\eta,\mu)=\tau'$.
This is a contradiction.
\end{proof}

\begin{lemma} For $S\in\ns_\kappa$, if $\diamondsuit^*(S)$ holds then so does $\ubd^*(\ns_\kappa\restriction S,\{\kappa\setminus\epsilon\mid \epsilon<\kappa\})$.
\end{lemma}
\begin{proof} By Proposition~\ref{remark25}, using Lemma~\ref{lemma47}(1) below.
\end{proof}

We remind the reader that $\cg^*( S,T)$ was defined in Definition~\ref{cgstar}.
\begin{lemma}\label{lemma55} For $S,T\in\ns_\kappa$, if $\cg^*( S,T)$ holds, 
then so does $\ubd^*([\kappa]^\kappa,\allowbreak\ns_\kappa\restriction S,(\ns_\kappa\restriction T)^+)$.
\end{lemma}
\begin{proof} Let $\vec C=\langle C_\beta\mid \beta\in S\rangle$ witness $\cg^*( S,T)$. 
Pick an upper-regressive colouring $c:[\kappa]^2\rightarrow\kappa$ such that, 
for all $\beta\in\acc(\kappa)\cap S$ and $\eta<\beta$, $c(\eta,\beta)=\min(C_\beta\setminus\eta)$.
Towards a contradiction, suppose that $c$ does not witness $\ubd^*([\kappa]^\kappa,\ns_\kappa\restriction S,(\ns_\kappa\restriction T)^+)$.
This means that there are a cofinal subset $A\s\kappa$ and a stationary subset $B\s S$ for which the following set
$$T':=\{\tau\in T \mid \forall \eta\in (A\cap \tau) \forall \beta \in (B\setminus(\eta+1))(c(\eta, \beta) \neq \tau)\}$$
meets every element of $(\ns_\kappa\restriction T)^+$. So we can fix a club $D \s \kappa$ such that $D \cap T \s T'$.	
Consider the club $E:=D\cap\acc^+(A)$.
By the choice $\vec C$, we may now find $\beta\in B$ with $\sup(\nacc(C_\beta)\cap E\cap T)=\beta$.
Pick $\tau\in \nacc(C_\beta)\cap E\cap T$. As $\tau\in\acc^+(A)$, we may pick $\eta\in A$ with $\sup(C_\beta\cap\tau)<\eta<\tau$,
so that $c(\eta,\beta)=\min(C_\beta\setminus\eta)=\tau$. However, $\tau \in D\cap T \s T'$. This is a contradiction.
\end{proof}

The reader may consult \cite{paper22} for various sufficient conditions for when the hypothesis of the first clause of the following lemma holds.
\begin{lemma}\label{proxyapp1}
\begin{enumerate}[(1)] 
\item If $\p^-(\kappa,\kappa^+,{\sq},1)$ holds, 
then so does $\ubd^*(J^{\bd}[\kappa],\allowbreak\{{\kappa\setminus\epsilon}\mid\epsilon<\kappa\})$;
\item If  $\square(\kappa,{<}\mu)$ holds with $\mu<\kappa$,
then so does $\ubd^*(J^{\bd}[\kappa],(\ns_\kappa)^+)$.
\end{enumerate}
\end{lemma}
\begin{proof} According to \cite[\S1.1]{paper22}, $\p^-(\kappa,\kappa^+,{\sq},1)$ provides us with a sequence $\langle\mathcal C_\beta\mid\beta<\kappa\rangle$ satisfying the following:
\begin{enumerate}
\item for every $\beta<\kappa$, $\mathcal C_\beta$ is a nonempty collection of closed subsets $C$ of $\beta$ with $\sup(C)=\sup(\beta)$;
\item for all $\beta<\kappa$, $C\in\mathcal C_\beta$ and $\alpha\in\acc(C)$, $C\cap\alpha\in\mathcal C_{\alpha}$;
\item for every cofinal $\Omega\s\kappa$, there exists $\delta\in\acc(\kappa)$
such that, for every $C\in\mathcal C_\delta$,  $\sup(\nacc(C)\cap \Omega)=\delta$.
\end{enumerate}

Assuming $\square(\kappa,{<}\mu)$ with $\mu<\kappa$, by \cite[Lemma~2.5]{paper29}, 
we may fix a sequence $\langle\mathcal C_\beta\mid\beta<\kappa\rangle$ 
satisfying Clauses (i) and (ii) above, together with:
\begin{itemize}
\item[(iii${}^-$)] for every club $\Omega\s\kappa$, there exists $\delta\in\acc(\kappa)$
such that, for every $C\in\mathcal C_\delta$,  $\sup(\nacc(C)\cap \Omega)=\delta$.
\end{itemize}

Fix a $C$-sequence $\vec C=\langle C_\beta\mid \beta<\kappa\rangle$ such that $C_\beta\in\mathcal C_\beta$ for all $\beta<\kappa$. 
We shall conduct walks on ordinals along $\vec C$, following the notation of \cite[\S4.2]{paper34} (see \cite{TodWalks} for a comprehensive treatment).

Define an upper-regressive colouring $c:[\kappa]^2\rightarrow\kappa$, as follows. 
Given $\eta<\beta<\kappa$, let $\gamma:=\min(\im(\tr(\eta,\beta)))$, so that $\eta\in C_\gamma$,
and then let $c(\eta,\beta):=\min(C_\gamma\setminus(\eta+1))$ provided that the latter is a well-defined ordinal $<\beta$; otherwise, just let $c(\eta,\beta):=0$.

Now, let $B\in[\kappa]^\kappa$ be arbitrary. We need to prove that the following set
$$T:=\{ \tau<\kappa\mid \exists\eta<\kappa\exists\beta\in B[\eta<\tau<\beta\ \&\ c(\eta,\beta)=\tau]\}$$
is co-bounded (resp.~stationary).
To this end, let $\Omega$ be an arbitrary cofinal (resp.~club) subset of $\kappa$,
and we shall show that $T\cap\Omega\neq\emptyset$.
Find $\delta\in\acc(\kappa)$ such that $\sup(\nacc(C)\cap\Omega)=\delta$ for all $C\in\mathcal C_\delta$.
Pick $\beta\in B$ above $\delta$, and then let 
$$\varepsilon:=\sup(\delta \cap \{\sup(C_{\gamma}\cap \delta)\mid \gamma\in\im(\tr(\delta,\beta))\}).$$
Then $\varepsilon<\delta$ and by a standard fact (see \cite[Lemma~4.7]{paper34}), there are two cases to consider:

$\br$ If $\delta\in\nacc(C_{\min(\im(\tr(\delta,\beta))})$,
then,  for every $\eta$ with $\varepsilon<\eta<\delta$,
$\tr(\eta,\beta)=\tr(\delta,\beta){}^\smallfrown\tr(\eta,\delta)$.
In this case, pick a large enough $\tau\in\nacc(C_\delta)\cap\Omega$ for which $\eta:=\sup(C_\delta\cap\tau)$ is greater than $\varepsilon$.
Then $\min(\im(\tr(\eta,\beta)))=\delta$ and hence $c(\eta,\beta)=\min(C_\delta\setminus(\eta+1))=\tau$.
 
$\br$ Otherwise, $\delta\in\acc(C_\gamma)$, for $\gamma:=\min(\im(\tr(\delta,\beta)))$,
and, for every $\eta$ with $\varepsilon<\eta<\delta$,
$\tr(\eta,\beta)=\tr(\gamma,\beta){}^\smallfrown\tr(\eta,\gamma)$.
As $\delta\in\acc(C_\gamma)$, $C_\gamma\cap\delta$ is in $\mathcal C_\delta$,
and hence we may pick a large enough $\tau\in\nacc(C_\gamma\cap\delta)\cap\Omega$ for which $\eta:=\sup(C_\gamma\cap\tau)$ is greater than $\varepsilon$.
Then $\min(\im(\tr(\eta,\beta)))=\gamma$ and $c(\eta,\beta)=\min(C_\gamma\setminus(\eta+1))=\tau$.
\end{proof}

The next lemma demonstrates that $\ubd^*(\ldots)$ is often times stronger than $\ubd(\ldots)$.
Note that, unlike Corollary~\ref{cor228}(3), in the following the ideal $J$ is not assumed to be subnormal. 
This connects to the content of Subsection~\ref{ulamremark}.
\begin{lemma}\label{ubdplus} Suppose that $c:[\kappa]^2\rightarrow\kappa$ witnesses $\ubd^*(J,\{T\})$ for $J$ a $\kappa$-complete ideal over $\kappa$ and $T\s\kappa$ stationary.
Then the following strong form of $\ubd^+(J,\kappa)$ holds:
For every sequence $\langle B_\tau\mid \tau<\kappa\rangle$ of $J^+$-sets,
there exists $\eta<\kappa$ such that $\{\tau\in T\mid \{\beta\in B_\tau\mid c(\eta,\beta)=\tau\}\in J^+\}$ is stationary in $\kappa$.
\end{lemma}
\begin{proof} We commence with an easy observation.
\begin{claim} Let $B\in J^+$ and $\tau\in T$. Then there exists $\eta<\tau$ such that $B^{\eta,\tau}:=\{\beta\in B\mid c(\eta,\beta)=\tau\}$ is in $J^+$.
\end{claim}
\begin{why} Suppose not. Since $J$ is $\kappa$-complete, it follows that $B':=B\setminus \bigcup_{\eta<\tau}B^{\eta,\tau}$ is in $J^+$.
Since $B'$ is in $J^+$ and $\tau\in T$, there must exist $\eta<\tau$ and $\beta\in B'$ such that $c(\eta,\beta)=\tau$. 	
So $\beta\in B'\cap B^{\eta,\tau}$. This is a contradiction.
\end{why}

Now, given a sequence $\langle B_\tau\mid \tau<\kappa\rangle$ of $J^+$-sets,
for each $\tau\in T$, pick $\eta_\tau<\tau$ such that $B_\tau^{\eta_\tau,\tau}$ is in $J^+$.
By Fodor's lemma, there must exist $\eta<\kappa$ for which $\{\tau\in T\mid \eta_\tau=\eta\}$ is stationary.
Clearly, $\eta$ is as sought.	
\end{proof}

Let us now demonstrate the utility of $\ubd^*(\ldots)$.

\begin{lemma}\label{minustoplus} Suppose that $\ubd^*(J,(\ns_\kappa\restriction  S)^*)$ holds for an ideal $J$ over $\kappa$,
and a stationary $S\s\kappa$. If $\onto^-(S,\theta)$ holds, then so does $\onto(J,\theta)$.
\end{lemma}
\begin{proof} Fix an upper-regressive map  $u:[\kappa]^2\rightarrow\kappa$ witnessing $\ubd^*(J,(\ns_\kappa\restriction  S)^*)$.
Fix a bijection $\pi:\kappa\leftrightarrow\kappa\times\kappa$.
Now, given a 		colouring $c:[\kappa]^2\rightarrow\theta$ that witnesses $\onto^-(S,\theta)$,
pick a colouring $d:[\kappa]^2\rightarrow\theta$ 
such that, for all $\eta<\beta<\kappa$, if $\pi(\eta)=(\eta_0,\eta_1)$, then $d(\eta,\beta)=c(\{\eta_0,u(\{\eta_1,\beta\})\})$.

Next, given $B\in J^+$, by the choice of $u$, we may fix a club $D\s\kappa$ such that, for every $\delta\in  S\cap D$,
for some $\eta\in\delta$ and $\beta\in B\setminus(\delta+1)$, $u(\eta,\beta)=\delta$.
By shrinking $D$, we may also assume that $\pi[\delta]=\delta\times\delta$ for all $\delta\in D$.
Define a function $f: S\cap D\rightarrow\kappa$ by letting $f(\delta)$ be the least $\eta_1\in\delta$ such that, 
for some $\beta\in B\setminus(\delta+1)$, $u(\eta_1,\beta)=\delta$.
Evidently, $f$ is regressive.
So, since $c$ witnesses $\onto^-(S,\theta)$, we may find $\eta_0,\eta_1<\kappa$ such that
$$c[\{\eta_0\}\circledast\{\delta\in S\cap D\mid f(\delta)=\eta_1\}]=\theta.$$

Finally, given $\tau<\theta$, fix $\delta\in S\cap D\setminus(\eta_0+1)$ with $f(\delta)=\eta_1$ such that $c(\eta_0,\delta)=\tau$.
Pick $\beta\in B\setminus(\eta+1)$ such that $u(\eta_1,\beta)=\delta$.
As $\max\{\eta_0,\eta_1\}<\delta\in D$, there exists $\eta<\delta$ such that $\pi(\eta)=(\eta_0,\eta_1)$.
Altogether $\eta<\delta<\beta$ and $d(\eta,\beta)=c(\{\eta_0,u(\{\eta_1,\beta\})\})=c(\eta_0,\delta)=\tau$, as sought.
\end{proof}

\begin{lemma}\label{lemma57} Suppose that $I,J$ are ideals over $\kappa$,
with $I$ being moreover normal extending $J^{\bd}[\kappa]$, and that $\ubd^*(J,I^+)$ holds.
\begin{enumerate}[(1)]
\item If $\onto(I,\theta)$ holds, then so does $\onto(J,\theta)$;
\item If $\ubd(I,\theta)$ holds, then so does $\ubd(J,\theta)$.
\end{enumerate}
\end{lemma}
\begin{proof} Fix an upper-regressive map $u:[\kappa]^2\rightarrow\kappa$ witnessing $\ubd^*(J,I^+)$.
and a bijection $\pi:\kappa\leftrightarrow\kappa\times\kappa$.
Set $D:=\{ \delta<\kappa\mid \pi[\delta]=\delta\times\delta\}$.
Now, given a (resp.~upper-regressive) colouring $c:[\kappa]^2\rightarrow\theta$,
pick a (resp.~upper-regressive) colouring $d:[\kappa]^2\rightarrow\theta$ 
such that, for all $\eta<\beta<\kappa$, if $\pi(\eta)=(\eta_0,\eta_1)$, then $d(\eta,\beta)=c(\{\eta_0,u(\{\eta_1,\beta\})\})$.

Next, given $B\in J^+$, fix $T\in I^+$ such that, for every $\tau\in  T$,
for some $\eta<\tau$ and $\beta\in B\setminus(\tau+1)$, $u(\eta,\beta)=\tau$.
As $I$ is normal, we may find some $\eta_1<\kappa$ for which $T':=\{\tau\in T\cap D\mid \exists \beta\in B~[\eta_1<\tau<\beta\ \&\ u(\eta_1,\beta)=\tau]\}$ is in $I^+$.		
So, assuming that $c$ witnesses $\onto(I,\theta)$ (resp.~$\ubd(I,\theta)$),
we may find an $\eta_0<\kappa$ such that $A:=c[\{\eta_0\}\circledast T']$ is equal to $\theta$ (resp of order-type $\theta$).		
Set $\eta:=\pi^{-1}(\eta_0,\eta_1)$.
\begin{claim} $A\s d[\{\eta\}\circledast B]$.
\end{claim}
\begin{why} Let $\alpha\in A$.	Pick $\tau\in T'$ above $\eta_0$ such that $c(\eta_0,\tau)=\alpha$.
As $\max\{\eta_0,\eta_1\}<\tau$ and $\tau\in D$, $\eta<\tau$.
In addition, as $\tau\in T'$, we may pick $\beta\in B$ above $\tau$ such that $u(\eta_1,\beta)=\tau$.
Altogether, $\eta<\beta$ and $d(\eta,\beta)=c(\{\eta_0,u(\{\eta_1,\beta\})\})=c(\eta_0,\tau)=\alpha$, as sought.		
\end{why}
This completes the proof.
\end{proof}

\begin{cor} For $S,T\in\ns_\kappa$, if $\cg^*( S,T)$ and $\onto(\ns_\kappa\restriction T,\theta)$ both hold,
then $\onto(\ns_\kappa\restriction S,\theta)$ holds, as well.
\end{cor}
\begin{proof} By Lemmas \ref{lemma55} and \ref{lemma57}.
\end{proof}

\section{Pumping-up results}\label{sectionpumping}
The theme of this section is to obtain instances of the colouring principles we have defined, 
either from classical colouring principles, or from other instances of our own colouring principles. 
A recurring theme is that at the cost of lowering the total number of colours we can upgrade the quality of the colouring.
This can be seen by examining the three most technical results, Theorems \ref{thm58} and \ref{thm42}, and some cases of Theorem~\ref{ubdtoonto}. 
Apart from Theorem~\ref{ubdtoonto}, Lemma~\ref{increasecolours} provides another example where one can upgrade the quality of the colouring without reducing the number of colours, 
though one needs an auxiliary colouring to perform this upgrade.

\begin{prop}\label{singularprojection} Suppose that $\theta\le\cf(\kappa)<\kappa$.
Let $\textsf p\in\{\onto,\onto^+,\onto^{++},\allowbreak\ubd,\ubd^+,\ubd^{++}\}$.

If $\textsf p(J^{\bd}[\cf(\kappa)],\theta)$ holds, then so does $\textsf p(\{\cf(\kappa)\},J^{\bd}[\kappa],\theta)$.
\end{prop}
\begin{proof} Set $\varkappa:=\cf(\kappa)$. 
Fix a cofinal subset $x\s\kappa\setminus\varkappa$ of order-type $\varkappa$.
Define a map $\pi:\kappa\rightarrow\varkappa$ via $\pi(\beta):=\otp(x\cap\beta)$.
Note that $\pi(\beta)<\beta$ for all $\beta<\kappa$,
and that, for every $B\s\kappa$, 
$B\in (J^\bd[\kappa])^+$ iff $\pi[B]\in (J^{\bd}[\varkappa])^+$.
Now, given a colouring $d:[\varkappa]^2\rightarrow\theta$
witnessing $\mathsf p(J^{\bd}[\varkappa],\theta)$,
pick any upper-regressive colouring $c:[\kappa]^2\rightarrow\theta$ such that for all $\eta<\varkappa\le\beta<\kappa$,
$c(\eta,\beta)=d(\eta,\pi(\beta))$.	
It is easy to verify that $c$ witnesses $\textsf p(\{\cf(\kappa)\},J^{\bd}[\kappa],\theta)$.
\end{proof}

\begin{lemma}\label{lemma64} Let $J\in\mathcal J^\kappa_\omega$,
and suppose that $\theta$ is an infinite cardinal less than $\kappa$.
\begin{enumerate}[(1)]
\item If $\ubd^+(J,\theta^+)$ holds, then so does  $\ubd^+(J,\theta)$;
\item If $J$ is subnormal and $\theta^+<\cf(\kappa)$, then $\ubd(J,\theta^+)$ implies $\ubd(J,\theta)$.
\end{enumerate}
\end{lemma}
\begin{proof} Suppose $c:[\kappa]^2\rightarrow\theta^+$ is a given upper-regressive colouring. 
Fix a bijection $\pi:\kappa\leftrightarrow\kappa\times\theta^+$.
For every $\tau<\theta^+$, fix an injection $e_\tau:\tau\rightarrow\theta$.
Now, pick any upper-regressive colouring $d:[\kappa]^2\rightarrow\theta$ such that for all $\eta<\beta<\kappa$,
if $\pi(\eta)=(\eta',\tau')$, $\max\{\eta',\theta\}<\beta$ and $c(\eta',\beta)<\tau'$, then $d(\eta,\beta)=e_{\tau'}(c(\eta',\beta))$.

(1) Suppose that $c$ witnesses $\ubd^+(J,\theta^+)$,
and we shall show that $d$ witnesses $\ubd^+(J,\theta)$.
To this end, let $B\in J^+$ be arbitrary.
Fix $\eta'<\kappa$ for which the following set has order-type $\theta^+$: $$T:=\{\tau<\theta^+\mid \{\beta\in B\setminus(\eta'+1))\mid c(\eta',\beta)=\tau\}\in J^+\}.$$
Fix $\tau'\in T$ such that $\otp(T\cap\tau')=\theta$,
and let $\eta:=\pi^{-1}(\eta',\tau')$. We claim that 
$$S:=\{\sigma<\theta\mid \{\beta\in B\setminus(\eta+1))\mid d(\eta,\beta)=\sigma\}\in J^+\}$$
covers the set $e_{\tau'}[T]$, hence has order-type $\theta$.
To see this, let $\sigma\in e_{\tau'}[T]$ be arbitrary. 
Fix $\tau\in T$ such that $e_{\tau'}(\tau)=\sigma$. As $\tau\in T$ and $J$ extends $J^{\bd}[\kappa]$, the set
$$\hat B:=\{\beta\in B\setminus(\max\{\theta,\eta',\eta\}+1))\mid c(\eta',\beta)=\tau\}$$ is in $J^+$.
Clearly, for every $\beta\in\hat B$, $d(\eta,\beta)=e_{\tau'}(c(\eta',\beta))=e_{\tau'}(\tau)=\sigma$, as sought.

(2)	Suppose that $\theta^+<\cf(\kappa)$, 
so that $D:=\{ \delta<\kappa\mid \pi[\delta]=\delta\times\theta^+\}$ is cofinal in $\kappa$.
Suppose also that $J$ is subnormal and that $c$ witnesses $\ubd(J,\theta^+)$; we shall show that $d$ witnesses $\ubd(J,\theta)$.
To this end, let $B\in J^+$ be arbitrary.
By Lemma~\ref{dseparated}, and by possibly passing to a positive subset of $B$,
we may assume that $B$ is $D$-separated.
Now, by the choice of $c$, fix $\eta'<\kappa$ such that $c[\{\eta'\}\circledast B]$ has order-type $\theta^+$.
Set $\bar\beta:=\min(B\setminus(\eta'+1))$ and $B':=B\setminus(\bar\beta+1)$.
As $B$ is $D$-separated, we may fix $\delta\in D$ such that $\bar\beta<\delta<\min(B')$.
As the sets $\{\eta'\}\circledast B$ and $\{\eta'\}\times B'$  differ on at most one element, $T:=c[\{\eta'\}\times B']$ has order-type $\theta^+$, 
so we may fix ${\tau'}\in T$ such that $\otp(T\cap{\tau'})=\theta$,
and let $\eta:=\pi^{-1}(\eta',{\tau'})$. 

As $\eta'<\bar\beta<\delta$, we infer that $\eta<\delta<\min(B')$, so 
$$d[\{\eta\}\circledast B]\supseteq e_{\tau'}[c[\{\eta'\}\times B']\cap \tau']=e_{\tau'}[T],$$
with the latter being a subset of $\theta$ of order-type $\theta$.
\end{proof}	

\begin{defn}[\cite{paper34}] $\U(\kappa, \mu, \theta, \chi)$ asserts the existence of a colouring 
$c:[\kappa]^2\rightarrow \theta$ such that for every $\sigma < \chi$, every pairwise disjoint subfamily
$\mathcal{A} \s [\kappa]^{\sigma}$ of size $\kappa$,
and every $i < \theta$, there exists $\mathcal{B} \in [\mathcal{A}]^\mu$
such that $\min(c[a \times b]) > i$ for all $a, b \in \mathcal{B}$ with $\sup(a)<\min(b)$.
\end{defn}

\begin{prop}\label{prop45} Suppose $\theta<\kappa$, $c:[\kappa]^2\rightarrow\theta$ is a colouring and $S \s \kappa$.
\begin{enumerate}[(1)]
\item Suppose that $\theta\in\reg(\cf(\kappa))$ and $c$ satisfies that for every cofinal $B\s\kappa$,
$\sup(c``[B]^2)=\theta$. Then $c$ witnesses $\ubd(J^{\bd}[\kappa],\theta)$;
\item Suppose that $\kappa$ is regular uncountable, $S$ is stationary, and $c$ satisfies that for every stationary $B\s S$, $c``[B]^2=\theta$. 
Then $c$ witnesses $\onto(\ns_\kappa\restriction S,\theta)$;
\item  Suppose $\theta< \cf(\kappa)$ and that $c$ satisfies that for every cofinal $B\s\kappa$,
$c``[B]^2=\theta$. Then $c$ witnesses $\onto(J^{\bd}[\kappa],\theta)$.
\end{enumerate} 

In particular, each of the following implies that $\ubd(J^\bd[\kappa],\theta)$ holds:
\begin{itemize}
\item $\U(\kappa,2,\theta,2)$;
\item $\kappa\nrightarrow[\kappa]^2_\theta$;
\item the existence of a $\kappa$-Souslin tree.\footnote{By \cite[Lemma~1]{MR0371662}, 
a $\kappa$-Souslin tree gives rise to a colouring witnessing $\kappa\nrightarrow[\kappa]^2_\kappa$,
and it is in fact the case that the existence of a $\kappa$-Souslin tree implies $\kappa\nrightarrow[\kappa;\kappa]^2_\kappa$.}
\end{itemize}
\end{prop}
\begin{proof} We focus on Clause~(1). So, suppose that $c$ is as above yet there exists a cofinal $B\s\kappa$ such that, 
for every $\eta<\kappa$, $\varsigma_\eta:=\sup(c``[\{\eta\}\circledast B])$ is $<\theta$.
As $\theta<\cf(\kappa)$, we may now fix some $\varsigma<\theta$ for which $B':=\{\eta\in B\mid \varsigma_\eta=\varsigma\}$ is cofinal.
Appealing to the property of $c$, we have that $\sup(c``[B']^2)=\theta$, so there must exist $(\eta,\beta)\in[B']^2$ with $c(\eta,\beta)>\varsigma$,
contradicting the fact that $\varsigma=\varsigma_\eta\ge c(\eta,\beta)$.
\end{proof}

\begin{prop}\label{prop52}
Suppose that $\theta\in\reg(\cf(\kappa))$ and there is a tree $\mathcal T$ of height $\theta$ with at least $\kappa$ many branches,
and	all of whose levels have size than $<\kappa$. Then $\U(\kappa, 2, \theta, 2)$ holds. 
In particular, each of the following implies $\U(\kappa,2, \theta, 2)$ and hence $\ubd(J^\bd[\kappa], \theta)$ as well:
\begin{enumerate}
\item there are cardinals $\theta\in\reg(\cf(\kappa))$ and $\lambda< \kappa$ such that $\lambda^{<\theta}<\kappa\leq \lambda^{\theta}$;
\item $\theta= \aleph_0<\cf(\kappa)\le\kappa \leq 2^{\aleph_0}$;
\item $\theta = \aleph_0<\cf(\kappa)\le \kappa<\kappa^{\aleph_0}$.
\end{enumerate}
\end{prop}
\begin{proof} Same argument as in \cite[Lemma~2.7]{paper34},
and then we may appeal to Proposition~\ref{prop45}.	
Now, let us show why a relevant tree exists under any of the hypotheses (i)--(iii):
\begin{enumerate}
\item Assuming there are cardinals $\theta\in\reg(\cf(\kappa))$ and $\lambda< \kappa$ such that $\lambda^{<\theta}<\kappa\leq \lambda^{\theta}$,
we can take the tree $\mathcal T:=({}^{<\theta}\lambda,{\s})$.
\item Assuming $\aleph_0<\cf(\kappa)\le\kappa \leq 2^{\aleph_0}$, we appeal		to Clause~(i) with $\lambda:=2$ and $\theta:=\aleph_0$.
\item Assuming $\kappa<\kappa^{\aleph_0}$ for a cardinal $\kappa$ of uncountable cofinality, 
there must exist a cardinal $\lambda<\kappa$ such that $\lambda^{\aleph_0}>\kappa$,
so we can take the tree $\mathcal T:=({}^{<\omega}\lambda,{\s})$, with $\theta:=\aleph_0$.\qedhere
\end{enumerate}
\end{proof}

The next proposition should be well-known, and is probably due to Erd\H{o}s and Hajnal.

\begin{prop}\label{prop46}
Suppose that $\theta<\cf(\kappa)= \kappa$. Then the following are equivalent:
\begin{enumerate}[(1)]
\item $\kappa\nrightarrow[\kappa;\kappa]^2_\theta$;
\item There is a colouring $c:[\kappa]^2\rightarrow\theta$ witnessing $\onto^{++}(J^+,\allowbreak J,\theta)$ 
for every subnormal $J\in\mathcal J^\kappa_{\theta^+}$;
\item $\onto([\kappa]^\kappa,J^\bd[\kappa], \theta)$.
\end{enumerate} 
\end{prop}
\begin{proof} $(1)\implies(2)$: Suppose that $c$ is a colouring  witnessing $\kappa\nrightarrow[\kappa;\kappa]^2_\theta$.
Suppose that we are given 		a set $A\in J^+$ and a sequence $\langle B_\tau\mid \tau<\theta\rangle$ of $J^+$-sets,
for some subnormal $J\in\mathcal J^\kappa_{\theta^+}$.
Towards a contradiction, suppose that for every $\eta\in A$, there exists a $\tau<\theta$ for which 
$\{ \beta\in B_\tau\mid c(\eta,\beta)=\tau\}\in J$.
Then, by the $\theta^+$-completeness of $J$,
we may fix a $\tau<\theta$ for which 
$$\bar A:=\{\eta\in A\mid \{ \beta\in B_\tau\mid c(\eta,\beta)=\tau\}\in J\}$$ is in $J^+$.
For every $\eta\in\bar A$, let $E_\eta:=\{\beta<\kappa\mid \beta\notin B_\tau\text{ or }c(\eta,\beta)\neq\tau\}$.
Now, by subnormality of $J$, we may find $A'\s\bar A$ and $B'\s B_\tau$ both in $J^+$ such that, 
for every $(\eta,\beta)\in A'\circledast B'$, $\beta\in E_\eta$.
By the choice of $c$ and because $J$ extends $J^\bd[\kappa]$, there are $(\eta,\beta)\in A'\circledast B'$ such that $c(\eta,\beta)=\tau$.
This contradicts the fact that $\beta\in B_\tau\cap E_\eta$.

$(2)\implies(3)\implies (1)$: This is trivial.
\end{proof}

\begin{lemma}\label{ehpumpubd} Suppose that $\theta\in\reg(\kappa)$,
and that a colouring $c:[\kappa]^2 \rightarrow \theta$ witnesses $\ubd([\kappa]^\kappa,J^\bd[\kappa], \theta)$.
Then, $c$ moreover witnesses $\ubd^{++}(J^+,J,\theta)$ for every subnormal $J\in\mathcal J^\kappa_{\theta^+}$.
\end{lemma}
\begin{proof} Let $J$ be any $\theta^+$-complete subnormal ideal over $\kappa$ extending $J^{\bd}[\kappa]$.
\begin{claim} For every $B\in J^+$, the following set is in $J^*$:
$$\{\eta<\kappa\mid \otp(\{\tau<\theta\mid \{\beta\in B\setminus(\eta+1)\mid c(\eta,\beta)=\tau\}\in J^+\})=\theta\}.$$
\end{claim}
\begin{why} Towards a contradiction, suppose that $B\in J^+$ forms a counterexample.
Denote $B_\eta^\tau:=\{\beta \in B\setminus(\eta+1) \mid c(\eta, \beta) = \tau\} $
and $T_\eta:=\{\tau<\theta\mid B_\eta^\tau\in J^+\}$.
By the choice of $B$, $A:=\{\eta<\kappa\mid \otp(T_\eta)<\theta\}$ is in $J^+$.
As $J$ is $\theta^+$-complete, for every $\eta\in A$, $E_\eta:=\kappa\setminus\bigcup_{\tau\in \theta\setminus T_\eta}B_\eta^\tau$ is in $J^*$.
As $J$ is subnormal, we may find $A'\s A$ and $B'\s B$ in $J^+$ such that,
for every $(\eta,\beta)\in A'\circledast B'$, $\beta\in E_\eta$.

By the choice of $c$, there is an $\eta\in A'$ such that $c[\{\eta\}\circledast  B']$ has ordertype $\theta$. 
In particular, we may pick $\tau\in c[\{\eta\}\circledast  B']\setminus T_\eta$.
Pick $\beta \in B'$ above $\eta$ such that $c(\eta, \beta)=\tau$. As $(\eta,\beta)\in A'\circledast B'$, we have that $\beta \in E_\eta$,
and as $\tau\in\theta\setminus T_\eta$, $E_\eta\cap B_\eta^\tau=\emptyset$.
This is a contradiction.
\end{why}

Now, given a sequence $\langle B_\tau\mid \tau<\theta\rangle$ of $J^+$-sets,
we know that, for every $\tau<\theta$,
$A_\tau:=\{\eta<\kappa\mid \otp(\{\delta<\theta\mid \{\beta\in B_\tau\setminus(\eta+1)\mid c(\eta,\beta)=\delta\}\in J^+\})=\theta\}$ is in $J^*$.
As $J$ is $\theta^+$-complete, it follows that given any $A\in J^+$, we may fix $\eta\in A\cap\bigcap_{\tau<\theta}A_\tau$.
By the choice of $\eta$, it is now easy to find an injection $h:\theta\rightarrow\theta$ such that,
for every $\tau<\theta$,  $\{ \beta\in B_\tau\setminus(\eta+1)\mid c(\eta,\beta)=h(\tau)\}\in J^+$.
\end{proof}

We come now to the first of three theorems, Theorem~\ref{thm58}, Theorem~\ref{thm42}, and Theorem~\ref{ubdtoonto}, 
which are dedicated to pumping from $\ubd(\cdots, \varkappa)$ to an instance of $\onto$ with $\theta$-many colours where $\theta \leq \varkappa \leq \kappa$. 
Each of them proceed by slicing up the available cases in a similar way. The next theorem considers the simplest case.
\begin{thm}\label{thm58}  Suppose that $\theta<\varkappa<\kappa$ are infinite cardinals,
with $\theta$ and $\kappa$ regular, and $J\in\mathcal J^\kappa_\omega$ is subnormal.
If $\ubd(J,\varkappa)$ holds, then so does $\onto(J,\theta)$.
\end{thm}
\begin{proof} Suppose that $c:[\kappa]^2\rightarrow\varkappa$ is a colouring witnessing $\ubd(J,\varkappa)$.
There are three cases to consider. 
In each of the cases, we shall also define a regular cardinal $\bar\varkappa$ with $\theta<\bar\varkappa\le\varkappa$.

$\br$ If $\theta^+=\varkappa$, then $\varkappa\nrightarrow[\varkappa;\varkappa]^2_\theta$ holds (see \cite{paper14}),
so, by Proposition~\ref{prop46}, we may fix a colouring $d:[\varkappa]^2\rightarrow\theta$ witnessing 
$\onto(J^{\bd}[\varkappa],\theta)$. Set also $\bar\varkappa:=\varkappa$.

$\br$ If $\theta^+<\varkappa$ and $\varkappa$ is regular,
then set $\bar\varkappa:=\varkappa$.
By Fact~\ref{clubguessingfact},
we may		fix a $C$-sequence $\langle C_\delta\mid \delta\in E^{\bar\varkappa}_\theta\rangle$ 
with $\otp(C_\delta)=\theta$ for all $\delta\in E^{\bar\varkappa}_\theta$,
such that, for any club $D\s{\bar\varkappa}$, there exists $\delta\in E^{\bar\varkappa}_\theta$ with $C_\delta\s D$.
Then, pick any colouring $d:[{\bar\varkappa}]^2\rightarrow\theta$ such that for all $\beta<\delta<{\bar\varkappa}$ with $\cf(\delta)=\theta$,
$d(\beta,\delta)=\sup(\otp(C_\delta\cap\beta))$.

$\br$ If $\theta^+<\varkappa$ and $\varkappa$ is singular, then set $\bar\varkappa:=\theta^{++}$.
Fix a $C$-sequence $\langle C_\delta\mid \delta\in E^{\bar\varkappa}_\theta\rangle$,
and a colouring $d:[{\bar\varkappa}]^2\rightarrow\theta$ as in the previous case.
In addition, fix a $C$-sequence $\langle C_\delta\mid \delta\in E^{\varkappa}_{\bar\varkappa}\rangle$ 
with $\otp(C_\delta)=\bar\varkappa$ for all $\delta\in E^{\varkappa}_{\bar\varkappa}$.

Fix three maps $\pi_0:\kappa\rightarrow\bar\varkappa$, $\pi_1:\kappa\rightarrow\varkappa$ 
and $\pi_2:\kappa\rightarrow\kappa$ such that, for every $(i,j,k)\in\bar\varkappa\times\varkappa\times\kappa$,
there exists $\eta<\kappa$ such that $\pi_0(\eta)=i$, $\pi_1(\eta)=j$ and $\pi_2(\eta)=k$.
Fix a club $E\s\kappa$ such that, for every $\epsilon\in E$ and $(i,j,k)\in\bar\varkappa\times\varkappa\times\epsilon$,
there exists $\eta<\epsilon$ such that $\pi_0(\eta)=i$, $\pi_1(\eta)=j$ and $\pi_2(\eta)=k$.
	
Finally, pick a colouring $e:[\kappa]^2\rightarrow\theta$ such that, for all $\eta<\beta<\kappa$, 
if $\pi_0(\eta)=i$, $\pi_1(\eta)=j$ and $\pi_2(\eta)=k$, then
$$e(\eta,\beta):=\begin{cases}
d(\{i,\otp(C_j\cap c(k,\beta))\}),&\text{if }\bar\varkappa< \varkappa \text{ and }j\in E^\varkappa_{\bar\varkappa};\\
d(\{i,c(k,\beta)\}),&\text{otherwise}.\end{cases}$$

We claim that this colouring works. 
To see why, let $B\in J^+$. We can assume without loss of any generality that $B \s \kappa\setminus \varkappa$.
By Lemma~\ref{dseparated}, we may also assume that for every $(\alpha,\beta)\in[B]^2$, 
there is an $\epsilon\in E$ with $\alpha<\epsilon<\beta$.
By the choice of $c$, pick $\eta^*<\kappa$
such that $c[\{\eta^*\}\circledast B]$ has order-type $\varkappa$.
Let $\alpha:=\min(B\setminus(\eta^*+1))$ and $B':=B\setminus(\alpha+1)$.
Clearly, $X:=c[\{\eta^*\}\times B']$ has order-type $\varkappa$.

$\br$ If $\theta^+=\varkappa$, then 
as $d$ witnesses $\onto(J^{\bd}[\varkappa],\theta)$,
we may pick $\delta<\varkappa$ such that $d[\{\delta\}\circledast X]=\theta$.

$\br$ If $\theta^+<\varkappa$ and $\varkappa=\bar\varkappa$,
then $D:=\acc^+(X)$ is a club in $\bar\varkappa$,
so we may pick $\delta\in E^{\bar\varkappa}_\theta$ such that $C_\delta\s D$.
We claim that $d[X\circledast \{\delta\}]=\theta$. Indeed, for every $\tau<\theta$,
pick $\gamma\in C_\delta$ such that $\otp(C_\delta\cap\gamma)=\tau+1$, and then,
since $C_\delta\s\acc^+(X)$, find $\xi\in X$ such that $\sup(C_\delta\cap\gamma)<\xi<\gamma$.
Consequently, $d(\xi,\delta)=\sup(\otp(C_\delta\cap\xi))=\sup(\tau+1)=\tau$.

$\br$ If $\theta^+<\varkappa$ and $\varkappa>\bar\varkappa$,
then let $j\in E^\varkappa_{\bar\varkappa}$ denote the least ordinal to satisfy $\otp(X\cap j)=\bar\varkappa$.
So, $\bar X:=\{ \otp(C_j\cap\xi)\mid \xi\in X\}$ is a cofinal subset of $\bar\varkappa$,
and then, as in the previous case, we may pick 
we may pick $\delta\in E^{\bar\varkappa}_\theta$ such  $d[\bar X\circledast \{\delta\}]=\theta$.

Fix $\epsilon\in E$ with $\alpha<\epsilon<\min(B')$.
As $\eta^*<\alpha$, we may find $\eta<\epsilon$ such that $\pi_0(\eta)=\delta$ and $\pi_2(\eta)=\eta^*$.
If $\varkappa>\bar\varkappa$, then we may also require that $\pi_1(\eta)=j$.

\begin{claim} $e[\{\eta\}\circledast B]=\theta$.
\end{claim}
\begin{why} Let $\tau<\theta$.

$\br$ If $\theta^+=\varkappa$, then pick $\xi\in X\setminus(\delta+1)$ such that $d(\delta,\xi)=\tau$.
Then pick $\beta\in B'$ such that $c(\eta^*,\beta)=\xi$.
As $\bar\varkappa=\varkappa$, it follows that $e(\eta,\beta)=d(\delta,c(\eta^*,\beta))=d(\delta,\xi)=\tau$.	

$\br$ If $\theta^+<\varkappa$ and $\varkappa=\bar\varkappa$,
then pick $\xi\in X\cap \delta$ such that $d(\xi,\delta)=\tau$.
Then pick $\beta\in B'$ such that $c(\eta^*,\beta)=\xi$.
Evidently, $e(\eta,\beta)=d(c(\eta^*,\beta),\delta)=d(\xi,\delta)=\tau$.	

$\br$ Otherwise, pick $\bar\xi\in\bar X$ such that $d(\bar\xi,\delta)=\tau$,
then pick $\xi\in X$ such that $\otp(C_j\cap\xi)=\bar\xi$,
and then pick $\beta\in B'$ such that $c(\eta^*,\beta)=\xi$.
It follows that $e(\eta,\beta)=d(\otp(C_j\cap c(\eta^*,\beta)),\delta)=d(\otp(C_j\cap\xi),\delta)=d(\bar\xi,\delta)=\tau$.
\end{why}
This completes the proof.
\end{proof}
\begin{remark}\label{thm58b} A proof similar to the preceding shows that if $\kappa$ is regular,
$J\in\mathcal J^\kappa_\omega$, and $\ubd^+(J,\kappa)$ holds, then $\onto^+(J,\theta)$ holds for every $\theta\in\reg(\kappa)$.
\end{remark}

Compared to the preceding theorem, in the next theorem we obtain the stronger colouring principle $\onto^{++}(\ldots)$,
but we shall also be assuming that $\varkappa$ is regular.
\begin{thm}\label{thm42} Suppose that $\theta<\varkappa\le\kappa$ are infinite regular cardinals,
$I$ is an ideal over $\kappa$ that is not weakly $\kappa$-saturated,
and $S\s\kappa$ is stationary.
\begin{enumerate}[(1)]
\item For every $J\in\mathcal J^\kappa_{\theta^+}$, if $\ubd^+(J^+,J,\varkappa)$ holds, 
then so does $\onto^{++}(I^*,\allowbreak J,\theta)$;
\item If $\ubd(\ns_\kappa\restriction S,\kappa)$ holds, then for every normal ideal $J$ over $S$ extending $J^{\bd}[S]$,
$\onto^{++}(J^*,J,\theta)$ holds.
\end{enumerate}
\end{thm}
\begin{proof} The proof will be similar to, but a little more elaborate than, that of Theorem~\ref{thm58}

\underline{Step 1.} 
In Case~(1), pick a colouring $c$ witnessing $\ubd^+(J^+,J,\varkappa)$.

In Case~(2), fix a colouring $c$ as in Lemma~\ref{lemma44}(2).
Let $J$ be any normal ideal over $S$ extending $J^{\bd}[S]$. Then $c$ witnesses $\ubd^+(J^+,J,\varkappa)$ for $\varkappa:=\kappa$.

\smallskip\underline{Step 2.} If $\varkappa=\theta^+$,
then $\varkappa\nrightarrow[\varkappa;\varkappa]^2_\theta$ holds,
so, by Proposition~\ref{prop46}, we may fix a colouring $d:[\varkappa]^2\rightarrow\theta$ witnessing 
$\onto([\varkappa]^\varkappa,J^{\bd}[\varkappa],\theta)$. If $\varkappa\neq\theta^+$,
then we define a colouring $d:[\varkappa]^2\rightarrow\theta$, as follows.
Fix a $C$-sequence $\langle C_\delta\mid \delta\in E^\varkappa_\theta\rangle$ 
with $\otp(C_\delta)=\theta$ for all $\delta\in E^\varkappa_\theta$,
such that, for any club $D\s\varkappa$, there exists $\delta\in E^\varkappa_\theta$ with $C_\delta\s D$.
Then, pick any colouring $d:[\varkappa]^2\rightarrow\theta$ such that for all $\beta<\delta<\varkappa$ with $\cf(\delta)=\theta$,
$d(\beta,\delta)=\sup(\otp(C_\delta\cap\beta))$.

\smallskip\underline{Step 3.} In Case~(1),
as $I$ is not weakly $\kappa$-saturated,
we may fix two maps $\pi_0:\kappa\rightarrow\varkappa$ and $\pi_1:\kappa\rightarrow\kappa$ such that, for every $(i,j)\in\varkappa\times\kappa$,
the set $$P_{i,j}:=\{ \eta<\kappa\mid \pi_0(\eta)=i\ \&\ \pi_1(\eta)=j\}$$ is in $I^+$.

In Case~(2), as $c$ witnesses $\ubd^+(J^+,\allowbreak J,\kappa)$, 
Proposition~\ref{weaksaturation} implies that $J$ is not weakly $\kappa$-saturated.
So, in this case, 
we fix two maps $\pi_0:\kappa\rightarrow\varkappa$ and $\pi_1:\kappa\rightarrow\kappa$ such that 
the corresponding sets $P_{i,j}$ are in $J^+$ for every $(i,j)\in\varkappa\times\kappa$.

\smallskip\underline{Step 4.} 
Pick a colouring $e:[\kappa]^2\rightarrow\theta$ such that, for all $\eta<\beta<\kappa$, 
if $\pi_0(\eta)=i$ and $\pi_1(\eta)=j$, then $e(\eta,\beta)=d(\{i,c(j,\beta)\})$. We claim this colouring works.

\smallskip\underline{Step 5:} 
In Case~(1), assume that we are given $A\in I^*$ 
and a sequence $\langle B_\tau\mid \tau<\theta\rangle$ of sets in $J^+$.
As $c$ witnesses $\ubd^+(J^+,J,\varkappa)$ and $J$ is $\theta^+$-complete,
we may fix large enough $\eta^*<\kappa$ such that, for every $\tau<\theta$,
$$X_\tau:=\{ \xi<\varkappa\mid \{ \beta\in B_\tau\mid c(\eta^*,\beta)=\xi\}\in J^+\}$$
is cofinal in $\varkappa$. 
	
In Case~(2), assume that we are given $A\in J^*$ and a sequence $\langle B_\tau\mid \tau<\theta\rangle$ of sets in $J^+$.
By the choice of $c$, fix a large enough $\eta^*<\kappa$ such that, for every $\tau<\theta$,
the corresponding set $X_\tau$ is cofinal in $\kappa$. 

\smallskip\underline{Step 6.} If $\theta^+=\varkappa$, then 
as $d$ witnesses $\onto([\varkappa]^\varkappa,J^{\bd}[\varkappa],\theta)$,
we may find a large enough $\delta<\varkappa$ such that $d[\{\delta\}\circledast X_\tau]=\theta$ for all $\tau<\theta$.

If $\theta^+<\varkappa$, then $D:=\bigcap_{\tau<\theta}\acc^+(X_\tau)$ is a club in $\varkappa$,
and by the choice of $\vec C$, we may pick $\delta\in E^\varkappa_\theta$ such that $C_\delta\s D$.
Consequently, $d[X_\tau\circledast\{\delta\}]=\theta$ for all $\tau<\theta$.

So, in both cases, for every $\tau<\theta$, we may fix $\xi_\tau\in X_\tau$ such that $d(\{\delta,\xi_\tau\})=\tau$.

\smallskip\underline{Step 7.} Fix $\eta\in A\cap P_{\delta,\eta^*}$, so that $\pi_0(\eta)=\delta$ and $\pi_1(\eta)=\eta^*$.
For every $\tau<\theta$, as $\xi_\tau\in X_\tau$,
$B_\tau^*:=\{ \beta\in B_\tau\setminus(\eta+1)\mid c(\eta^*,\beta)=\xi_\tau\}$ is in $J^+$.
Clearly, for all $\tau<\theta$ and $\beta\in B_\tau^*$,
$e(\eta,\beta)=d(\{\delta,c(\{\eta^*,\beta\})\})=\tau$, as sought.
\end{proof}

\begin{cor}\label{cor610} Suppose that $\kappa$ is a regular uncountable cardinal,
and let $J$ denote $\ns_\kappa\restriction S$ for some stationary $S\s\kappa$.

\begin{enumerate}[(1)] 
\item If $\ubd(J,\kappa)$ holds, then $\onto^{++}(J,\theta)$ holds for every $\theta\in\reg(\kappa)$;
\item If $\reg(\kappa)$ is nonstationary or if $\diamondsuit^*(\reg(\kappa))$ holds, then $\onto^{++}(\ns_\kappa,\theta)$ holds for every $\theta\in\reg(\kappa)$;
\item If $\onto^{++}(J,\theta)$ fails for some $\theta\in\reg(\kappa)$, then $\square(\kappa,{<}\mu)$ fails for all $\mu<\kappa$, $\refl(\kappa,\kappa,\reg(\kappa))$ holds and $\kappa$ is greatly Mahlo;
\end{enumerate}
\end{cor}
\begin{proof}\begin{enumerate}[(1)]
\item By Theorem~\ref{thm42}(2).
\item By Clause~(1), using $S:=\kappa$ together with Proposition~\ref{diamondstar} and Lemma~\ref{lemma44}.
\item If $\onto^{++}(J,\theta)$ fails for some $\theta\in\reg(\kappa)$,
then, by Clause~(1), in particular $\ubd(J^{\bd}[\kappa],\kappa)$ fails. Now appeal to Corollaries \ref{square_is_amenable} and \ref{strongamenmahlo}.\qedhere
\end{enumerate}
\end{proof}

Recall that $\mathcal C(\kappa,\theta)$ denotes the least cardinality of a family of sets $\mathcal X\s[\kappa]^\theta$ with the property that for every closed and unbounded subset $C$ of $\kappa$,
there exists $X\in\mathcal X$ with $X\s C$.
Whether $\mathcal C(\theta^+,\theta)=\theta^+$ holds for every singular cardinal $\theta$ is an open problem,
but, by \cite[Lemma~3.1]{paper38}, $\mathcal C(\theta^+,\theta)\le\cf([\theta^+]^{\cf(\theta)},{\supseteq})\le\theta^{\cf(\theta)}$.

\begin{prop} If $\kappa=\theta^+$, then $\mathcal C(\kappa,\theta)=\kappa$ implies $\onto^+(J^\bd[\kappa],\theta)$.
\end{prop}
\begin{proof}[Proof sketch] 
Due to Corollary~\ref{prop41} below, this result is only of interest in the case that $\theta$ is singular.
We commence with a couple of reductions.
By Proposition~\ref{prop219}(4), it suffices to prove that $\onto(J^{\bd}[\kappa],\theta)$ holds.
By Lemma~\ref{upbyone}, the latter task reduces to proving that $\onto(\ns_\kappa,\theta)$ holds.
The proof of the latter is now very similar to that of Theorem~\ref{thm42}.
Fix a sequence $\langle C_\delta\mid \delta<\kappa\rangle$ of subsets of $\kappa$, each of order-type $\theta$,
such that, for every club $D$ in $\kappa$, for some $\delta<\kappa$, $C_\delta\s D$.
Define a colouring $d:[\kappa]^2\rightarrow\theta$ via  $d(\beta,\delta):=\sup(\otp(C_\delta\cap\beta))$.
As $\kappa$ is a successor cardinal, fix a colouring $c:[\kappa]^2\rightarrow\kappa$ as in Lemma~\ref{lemma44}(2).
Fix a bijection $\pi:\kappa\leftrightarrow\kappa\times\kappa$,
and finally pick a colouring $e:[\kappa]^2\rightarrow\theta$ such that, for all $\eta<\beta<\kappa$, 
if $\pi(\eta)=(i,j)$, then $e(\eta,\beta)=d(\{i,c(\{j,\beta\})\})$. 
By now, it should be clear that this works.
\end{proof}

Note that in the upcoming theorem, we allow $\theta$ to be finite. Unlike the similar Theorem~\ref{thm58} and Theorem~\ref{thm42}, below the focus is on obtaining an instance of $\onto(\{\nu\}, \ldots)$ for a fixed in advance $\nu \leq \kappa$.
\begin{thm}\label{ubdtoonto}
Suppose $\theta\le\varkappa\le\nu\le\kappa$ are cardinals,
and $\ubd(\{\nu\},J,\varkappa)$ holds for a given $J\in\mathcal J^\kappa_\omega$.
If $\nu=\kappa$, suppose also that $\kappa$ is regular, and that $J$ is subnormal.

Any of the following  hypotheses imply that $\onto(\{\nu\},J,\theta)$ holds:
\begin{enumerate}[(1)]
\item $\onto([\varkappa]^{<\varkappa},\theta)$ holds;
\item $\mathcal C(\varkappa,\theta)<\nu=\kappa$;
\item $\mathcal C(\varkappa,\theta)\le\nu<\kappa$.
\end{enumerate}
\end{thm}
\begin{proof} In Case~(1), if $\varkappa=\kappa$, then $\nu=\kappa$, and then $\onto(\{\nu\},J,\theta)$ holds.
So, in Case~(1), we may moreover assume that $\varkappa<\kappa$.
Fix a map $d:[\varkappa]^2\rightarrow\theta$ witnessing $\onto([\varkappa]^{<\varkappa},\theta)$.
Set $\chi:=\varkappa$, and notice that either $\chi<\nu=\kappa$ or $\chi\le\nu<\kappa$.

In Cases (2) and (3), let $\chi:=\mathcal C(\varkappa,\theta)$,
and fix a sequence $\langle X_i\mid i<\chi\rangle$ of subsets of $\varkappa$,
each of order-type $\theta$, such that, for every club $C$ in $\varkappa$,
for some $i<\chi$, $X_i\s C$.
Then, define a map $d:\chi\times\varkappa\rightarrow\theta$ via $d(i,\gamma)=\sup(\otp(X_i\cap\gamma))$.
Notice that either $\chi<\nu=\kappa$ or $\chi\le\nu<\kappa$.

\begin{claim}\label{claim641} Let $\Gamma\in[\varkappa]^\varkappa$. Then there exists $i<\chi$ such that $d[\{i\}\times\Gamma]=\theta$.
\end{claim}		
\begin{why} In Case~(1), this is clear, so we focus on Cases (2) and (3).

If $\acc^+(\Gamma)$ is cofinal in $\varkappa$, then
find $i<\chi$ such that $X_i\s \acc^+(\Gamma)$; otherwise, there exists a cofinal subset $\Gamma'\s \Gamma$ of order-type $\omega$ which is trivially closed,
so we may find $i<\chi$ such that $X_i\s \Gamma$.
Let $\tau<\theta$ be arbitrary. Pick $\xi\in X_i$ such that $\otp(X_i\cap\xi)=\tau$.
Let $\gamma:=\min(\Gamma\setminus(\xi+1))$. 
As $X_i\s \Gamma\cap\acc^+(\Gamma)$, $\otp(X_i\cap\gamma)=\otp(X_i\cap(\xi+1))=\tau+1$,
and hence $d(i,\gamma)=\sup(\otp(X_i\cap\gamma))=\tau$.
\end{why}
		
Fix a colouring $c:[\kappa]^2\rightarrow\varkappa$ witnessing $\ubd(\{\nu\},J,\varkappa)$.		
Fix a bijection $\pi:\nu\leftrightarrow\nu\times\chi$.
Then, fix any colouring $e:[\kappa]^2\rightarrow\theta$ satisfying the following.
For every $\eta<\beta<\kappa$, if $\eta<\nu$, $\pi(\eta)=(\eta',i)$ 
and $(i,c(\eta',\beta))\in\dom(d)$, then $e(\eta,\beta)=d(i,c(\eta',\beta))$.
		
To see that $e$ witnesses $\onto(\{\nu\},J,\theta)$, let $B\in J^+$ be arbitrary. Note:

$\br$ If $\chi\le\nu<\kappa$, then we may assume that $\min(B)\ge\nu$.

$\br$ If $\chi<\nu=\kappa$, then $D:=\{\delta<\kappa\mid \pi[\delta]=\delta\times\chi\}$ is a club in $\kappa$,
and, by Lemma~\ref{dseparated}, we may assume that 
for every pair $\bar\beta<\beta$ of points from $B$, there is $\delta\in D$ with $\bar\beta<\delta<\beta$.

Next, by the choice of $c$, find $\eta'<\nu$ such that $\otp(c[\{\eta'\}\circledast B])=\varkappa$.
Set $\bar\beta:=\min(B\setminus(\eta'+1))$ and $B':=B\setminus(\bar\beta+1)$. 
Evidently, $\Gamma:=c[\{\eta'\}\circledast B']$ is in $[\varkappa]^\varkappa$.
So, by Claim~\ref{claim641}, we may find $i<\chi$ such that $d[\{i\}\times\Gamma]=\theta$.
Finally, find $\eta<\nu$ such that $\pi(\eta)=(\eta',i)$. 
\begin{claim} $\min(B')>\eta$.
\end{claim}
\begin{why} Let $\beta\in B'$ be arbitrary. 

$\br$ If $\chi\le\nu<\kappa$, then $\eta<\nu\le\beta$.

$\br$ If $\chi<\nu=\kappa$, then there exists some $\delta\in D$ such that $\eta'<\bar\beta<\delta<\beta$,
and hence $\eta<\delta<\beta$.
\end{why}

To verify that $e[\{\eta\}\circledast B]=\theta$,
let $\tau<\theta$ be arbitrary. 
As $i$ was given by Claim~\ref{claim641}, we may fix $\gamma\in\Gamma$ such that $(i,\gamma)\in\dom(d)$ and $d(i,\gamma)=\tau$.
As $\gamma\in\Gamma$, we may fix $\beta\in B'$ such that $c(\eta',\beta)=\gamma$.
Then $e(\eta,\beta)=d(i,c(\eta',\beta))=d(i,\gamma)=\tau$, as sought.
\end{proof}

\begin{cor}\label{abovecont}
If $\kappa>2^\theta$, then $\ubd(J^{\bd}[\kappa],\theta)$ iff $\onto(J^{\bd}[\kappa],\theta)$.
\end{cor}
\begin{proof} Appeal to Theorem~\ref{ubdtoonto}(2) with $\varkappa:=\theta$ and $\nu:=\kappa$.
\end{proof}

The next lemma shows at the level of successors cardinals, the choice of the ideal hardly makes any difference.
It also gives a sufficient condition for pumping-up from $\theta$ colours into $\theta^+$ colours.

\begin{lemma}\label{upbyone} Suppose that $\kappa=\mu^+$ is a successor cardinal,
$\mu\le\nu\le\kappa$,
and $J\in\mathcal J^\kappa_\kappa$ is a nontrivial ideal.
\begin{enumerate}[(1)]
\item If $\onto(\{\nu\},J,\theta)$ holds, then so does $\onto(\{\nu\},J^{\bd}[\kappa],\theta)$;
\item If $\ubd(\{\nu\},J,\theta)$ holds, then so does $\ubd(\{\nu\},J^{\bd}[\kappa],\theta)$;
\item If $\onto(\{\nu\},J,\theta)$ holds with $\theta=\mu=\nu$, then so does $\onto(\{\nu\},J^{\bd}[\kappa],\kappa)$.
\end{enumerate}
\end{lemma}
\begin{proof} Fix a bijection $\pi:\nu\leftrightarrow\mu\times\nu$.
For every $\beta<\kappa$, fix a surjection $e_\beta:\mu\rightarrow\beta+1$.
For every subset $B\s\kappa$,
and for all $i<\mu$ and $\gamma<\kappa$, let $B^i_\gamma:=\{ \beta\in B\setminus\mu\mid e_\beta(i)=\gamma\}$,
and then let $\Gamma^i(B):=\{ \gamma\in\kappa\setminus\mu\mid B^i_\gamma\in [\kappa]^\kappa\}$.
\begin{claim} Let $B\in[\kappa]^\kappa$. Then there exists $i<\mu$ such that $\Gamma^i(B)\in J^+$.
\end{claim}
\begin{why} As $|B|=\kappa$, for every $\gamma<\kappa$, there exists some $i<\mu$ such that $|B^i_\gamma|=\kappa$.
Then, as $J$ is $\kappa$-complete, there exists some $i<\mu$ such that $\Gamma^i(B)\in J^+$.
\end{why}

(1) Suppose that $c:[\kappa]^2\rightarrow\theta$ is a colouring witnessing $\onto(\{\nu\},J,\theta)$.
Pick any colouring $d:[\kappa]^2\rightarrow\theta$ such that, for all $\eta<\beta<\kappa$,
if $\eta\in\nu$ and $\pi(\eta)=(i,j)$, then $d(\eta,\beta)=c(\{j,e_\beta(i)\})$.
To see this works, let $B\in(J^{\bd}[\kappa])^+$.
Pick $i<\mu$ such that $\Gamma^i(B)\in J^+$.
Find $j<\nu$ such that $c[\{j\}\circledast \Gamma^i(B)]=\theta$.
Find $\eta<\nu$ such that $\pi(\eta)=(i,j)$. To see that
$D[\{\eta\}\circledast B]=\theta$, let $\tau<\theta$ be arbitrary.
Pick $\gamma\in \Gamma^i(B)$ above $j$ such that $c(j,\gamma)=\tau$
and pick a large enough $\beta\in B^i_\gamma$ above $\eta$.
Then $d(\eta,\beta)=c(j,e_\beta(i))=c(j,\gamma)=\tau$.

(2) This is almost the exact same proof as that of Clause~(1).

(3) Suppose that $c:[\kappa]^2\rightarrow\theta$ is a colouring witnessing $\onto(\{\nu\},J,\theta)$,
with $\theta=\mu=\nu$. 
Pick any colouring $d:[\kappa]^2\rightarrow\kappa$ such that, for all $\eta\le\nu\le\beta<\kappa$,
if $\pi(\eta)=(i,j)$, then $d(\eta,\beta)=e_{e_\beta(i)}(c(j,\beta))$.
To see this works, let $B\in(J^{\bd}[\kappa])^+$.
Pick $i<\mu$ such that $\Gamma^i(B)\in J^+$.
For each $\gamma\in\Gamma^i(B)$, pick $j_\gamma<\nu$ such that $c[\{j_\gamma\}\circledast B_\gamma^i]=\theta$.
As $\nu<\kappa$, find $j<\nu$ such that $\Gamma:=\{\gamma\in\Gamma^i(B)\mid j_\gamma=j\}$ is cofinal in $\kappa$.
Find $\eta<\nu$ such that $\pi(\eta)=(i,j)$. 
To see that $d[\{\eta\}\circledast B]=\kappa$, let $\tau<\kappa$ be some prescribed colour.
Fix $\gamma\in\Gamma$ above $\tau$.
Find $\epsilon<\mu$ such that $e_\gamma(\epsilon)=\tau$.
As $c[\{j\}\circledast B_\gamma^i]=\theta$, we may fix $\beta\in B_\gamma^i$ such that $c(j,\beta)=\delta$.
Then $$d(\eta,\beta)=e_{e_\beta(i)}(c(j,\beta))=e_\gamma(\delta)=\tau,$$
as sought.
\end{proof}

In the next lemma, we move from a strong form of $\onto(J,\theta)$ to $\onto(J,\theta^+)$, assuming the necessary hypothesis of $\ubd(J,\theta^+)$.

\begin{lemma}\label{increasecolours}  Suppose that $\theta<\theta^+<\cf(\kappa)=\kappa$ are infinite cardinals,
and $J\in\mathcal J^\kappa_\kappa$ is subnormal.

If $\onto(J^+,J,\theta)$ 	and $\ubd(J,\theta^+)$ both hold,	then so does $\onto(J,\theta^+)$.
\end{lemma}
\begin{proof} 
Suppose that $\onto(J^+,J,\theta)$ and $\ubd(J,\theta^+)$ both hold.
By Clauses (2) and (3) of Proposition~\ref{prop219}, we may fix colourings $c:[\kappa]^2\rightarrow\theta$
and $d:[\kappa]^2\rightarrow\theta^+$ such that, 
if we denote 
\begin{itemize}
\item $B^a_i(\epsilon):=\{\beta\in B\setminus(i+1)\mid a(i,\beta)=\epsilon\}$ for $a \in \{c,d\}$, and
\item $A_j(B):=\{\alpha<\theta^+\mid B^d_j(\alpha)\in J^+\}$, and
\item $D_k(B):=\{\delta<\theta\mid B^c_k(\delta)\in J^+\}$, 
\end{itemize}
then, for every $B\in J^+$, there is $j<\kappa$ for which $|A_j(B)|=\theta^+$,
and the set $\{k<\kappa\mid D_k(B)=\theta\}$ is in $J^*$.

Next, for every $\alpha<\theta^+$, fix a surjection $e_\alpha:\theta\rightarrow\alpha+1$.
Fix a bijection $\pi:\kappa\leftrightarrow\theta\times\kappa\times\kappa$.
Define $f:[\kappa]^2\rightarrow\theta^+$ as follows.
Given $\eta<\beta<\kappa$, let $(i,j,k):=\pi(\eta)$ and then set $f(\eta,\beta):=e_{e_{d(j,\beta)}(i)}(c(\{k,\beta\}))$.

To see this works, let $B\in J^+$ be arbitrary.
Fix $j<\kappa$ such that $|A_j(B)|=\theta^+$.
\begin{claim} Let $\gamma<\theta^+$. There exists $i<\theta$ for which
$B_{i,\gamma}:=\{ \beta\in B \mid e_{d(j,\beta)}(i)=\gamma\}$ is in $J^+$.
\end{claim}
\begin{why} Fix $\alpha\in A_j(B)$ above $\gamma$.
Find $i<\theta$ such that $e_\alpha(i)=\gamma$.
For every $\beta\in B_j^d(\alpha)$, $e_{d(j,\beta)}(i)=e_\alpha(i)=\gamma$.
So $B^d_j(\alpha)\s B_{i,\gamma}$ and the former is in $J^+$, since $\alpha\in A_j(B)$.
\end{why}

Fix $i<\theta$ for which $\Gamma^i(B):=\{ \gamma<\theta^+\mid B_{i,\gamma}\in J^+\}$ is cofinal in $\theta^+$.
As $J^*$ is $\theta^{++}$-complete we may find some $k<\kappa$ such that $D_k(B_{i,\gamma})=\theta$ for all $\gamma\in\Gamma^i(B)$.
Find $\eta<\kappa$ such that $\pi(\eta)=(i,j,k)$.
Now, given any $\tau<\theta^+$, pick $\gamma\in\Gamma^i(B)$
above $\tau$. Find $\delta<\theta$ such that $e_\gamma(\delta)=\tau$.
Pick $\beta\in B_{i,\gamma}$ above $\eta$ such that $c(k,\beta)=\delta$. Then
$$f(\eta,\beta)=e_{e_{d(j,\beta)}(i)}(c(k,\beta))=e_{\gamma}(\delta)=\tau,$$
as sought.
\end{proof}
\begin{remark} The restriction ``$\theta^+<\cf(\kappa)$'' cannot be waived:
By Theorem~\ref{prop41} below and by Lemma~\ref{lemma44} (resp.), $\onto((\ns_{\aleph_1})^+,\ns_{\aleph_1},\aleph_0)$ 
and $\ubd(\ns_{\aleph_1},\aleph_1)$ both hold, whereas, by Fact~\ref{larson} below,
$\onto(\ns_{\aleph_1},\aleph_1)$ may consistently fail.
\end{remark}

\section{$\zfc$ results}\label{sectionzfc}
In this section our focus is on obtaining $\zfc$-provable instances of our principles. 
Our results on infinite successor cardinals in Corollary~\ref{prop41} and Theorem~\ref{thm54} 
together with the results of Section~\ref{sectionpumping} allow us to establish monotonicity for the simplest instances of the $\ubd(\ldots)$ principles in Corollary~\ref{cor65}. 
We then turn our attention to the case when $\kappa$ is, in turn, below the continuum, the smallest infinite cardinal, and singular. 
We finish in Subsection~\ref{subsectionfodor} with an application of our principles to a weakening of a classical problem about refining, 
pointwise, a sequence of stationary subsets of a regular cardinal to another sequence of stationary subsets of the cardinal which are furthermore pairwise disjoint.

\begin{cor}\label{cor71} For every stationary subset $S$ of a regular uncountable cardinal $\kappa$, there exists a stationary $S'\s S$ such that 
$\onto^{++}(\ns_\kappa\restriction S',\theta)$ holds for every $\theta\in\reg(\kappa)$.
\end{cor}
\begin{proof} Given $ S\s \kappa$ as above,
by Corollary~\ref{amensubset}, there is a stationary $ S' \s  S$ which carries an amenable $C$-sequence (in fact even a strongly amenable $C$-sequence). 
So, by Lemma~\ref{lemma44}, $\ubd(\ns_\kappa\restriction S',\kappa)$ holds, 
and then by Theorem~\ref{thm42}(2), so does $\onto^{++}(\ns_\kappa\restriction S',\theta)$ for every $\theta\in\reg(\kappa)$.
\end{proof}
\begin{remark} The preceding is optimal in the sense that, by Fact~\ref{larson} and 
Lemma~\ref{upbyone}(1), $\bpfa$ implies that $\onto(\ns_{\aleph_1}\restriction S,\aleph_1)$ fails for any stationary $S\s\aleph_1$.
\end{remark}

\begin{cor}\label{prop41}  Suppose that $\kappa=\theta^+$ for an infinite regular cardinal $\theta$. 
Then there exists a colouring $c:[\kappa]^2\rightarrow\theta$ 
witnessing $\onto^{++}(J^+,J,\theta)$ for every subnormal $J\in\mathcal J^\kappa_\kappa$.
\end{cor}
\begin{proof} By Proposition~\ref{prop46}, using \cite{paper14}.
\end{proof}

\begin{thm} \label{thm54} Suppose that $\kappa=\theta^+$ for a singular cardinal $\theta$. 
Then there exists a colouring $c:[\kappa]^2\rightarrow\theta$ 
witnessing $\ubd^{++}(J^+,J,\theta)$ for every subnormal $J\in\mathcal J^\kappa_\kappa$.	
\end{thm}
\begin{proof} By Lemma~\ref{ehpumpubd}, it suffices to prove that $\ubd(J^+, J,\theta)$ holds for $J:=J^{\bd}[\kappa]$.
By a theorem of Shelah (see \cite[Theorem~3.53]{MR2768694}), we may fix a sequence of regular uncountable cardinals $\langle \theta_i\mid i<\cf(\theta)\rangle$ converging to $\theta$
such that $\tcf(\prod_{i<\cf(\theta)}\theta_i,{<^*})=\kappa$.
Fix a scale $\langle f_\beta\mid \beta<\kappa\rangle$ witnessing the preceding.
By another theorem of Shelah (see \cite[Theorem~5.16]{MR2768694}), fix a colouring $d:[\kappa]^2\rightarrow\cf(\theta)$ witnessing $\kappa\nrightarrow[\kappa;\kappa]^2_{\cf(\theta)}$.
Let $c:[\kappa]^2\rightarrow\theta$ be any upper-regressive colouring such that, for all $\theta\le\eta<\beta<\kappa$, 
$$c(\eta,\beta)=f_\beta(d(\eta,\beta)).$$

Towards a contradiction, suppose that we are given $A,B\in [\kappa]^\kappa$ demonstrating that $c$ does not witness $\ubd([\kappa]^\kappa,J^{\bd}[\kappa],\theta)$.
So, for every $\eta \in A$ there is an $\epsilon_\eta <\theta$ such that 
$\otp(c[\{\eta\} \circledast B]) < \epsilon_\eta$. 
By stabilising, we can find an $\epsilon^* < \theta$ and $A' \s A\setminus\theta$ of size $\kappa$ such that for every $\eta \in A'$ we have $\epsilon_\eta = \epsilon^*$.

\begin{claim} \label{ubdscale}There exists $i<\cf(\theta)$ with $\theta_i>\epsilon^*$ such that 
$$\sup\{ \xi<\theta_i\mid \{ \beta\in B\mid f_\beta(i)=\xi\}\in J^+\}=\theta_i.$$
\end{claim}
\begin{why} Suppose not. Fix the least $j<\cf(\theta)$ such that $\theta_{j}>\epsilon^*$.
Then there exists a function $g\in\prod_{i<\cf(\theta)}\theta_i$ such that, for every $i\in[j,\cf(\theta))$,
$$g(i)=\sup\{ \xi<\theta_i\mid \{ \beta\in B\mid f_\beta(i)=\xi\}\in J^+\}.$$
It follows that, for every $i\in[j,\cf(\theta))$, $B^i:=\{ \beta\in B\mid g(i)<f_\beta(i)\}$ is in $J$.
Find $\alpha<\kappa$ such that $g<^* f_\alpha$. 
As $\bigcup_{i\in[j,\cf(\theta))}B^i$ is in $J$, we may let $\beta:=\min(B\setminus(\bigcup_{i\in[j,\cf(\theta))}B^i\cup\alpha))$.
Then $\alpha\le\beta$, so that $g<^* f_\beta$. In particular, there exists $i\in[j,\cf(\theta))$ with $g(i)<f_\beta(i)$, contradicting the fact that $\beta\notin B^i$.
\end{why}
Fix $i$ as in the claim. 
Since $\epsilon^*<\theta_i$, and recalling the definition of $A'$,
we may find $\varsigma<\theta_i$ such that the following set has size $\kappa$:
$$A'':=\{\eta\in A'\mid \sup(c[\{\eta\}\circledast B] \cap \theta_i)=\varsigma\}.$$ 

By the choice of $i$, we may find $\xi\in(\varsigma,\theta_i)$ such that the following set is in $J^+$:
$$B':=\{ \beta\in B\mid f_\beta(i)=\xi\}.$$

Finally, by the choice of $d$, we may find $(\eta,\beta)\in A''\circledast B'$ such that $d(\eta,\beta)=i$.
Altogether, $c(\eta,\beta)=f_\beta(d(\eta,\beta))=f_\beta(i)=\xi\in(\varsigma,\theta_i)$, contradicting the fact that $\eta \in A''$.
\end{proof}
\begin{remark} The preceding is optimal in the sense that, by Proposition~\ref{prop912} below, $\ubd^{++}(J^{\bd}[\kappa],\kappa)$ fails.
\end{remark}

\begin{cor}[monotonicity]\label{cor65} For a stationary subset $S$ of a regular uncountable cardinal $\kappa$,
and an infinite cardinal $\theta<\kappa$:
\begin{enumerate}[(1)]
\item If $\ubd( J^\bd[\kappa],\allowbreak\varkappa)$ holds for some $\varkappa$ with $\theta<\varkappa<\kappa$, 
or if $\kappa$ is a successor cardinal, then $\ubd(J^\bd[\kappa],\theta)$ holds.
If, in addition, $\theta$ is regular, then moreover $\onto(J^\bd[\kappa],\theta)$ holds.
\item If $\ubd( \ns_\kappa\restriction S,\varkappa)$ holds for some $\varkappa$ with $\theta<\varkappa\le\kappa$, then so does $\ubd(\ns_\kappa\restriction S,\theta)$.
If, in addition, $\theta$ is regular, then moreover $\onto(\ns_\kappa\restriction S,\theta)$ holds.
\end{enumerate}
\end{cor}
\begin{proof} (1) By Corollary~\ref{prop41}  and Theorem~\ref{thm54}, we may assume that $\theta^+<\kappa$.
If $\kappa$ is a successor, then by Fact~\ref{ulamoriginal} and Lemma~\ref{ubdplus}, $\ubd^+(J^{\bd}[\kappa],\allowbreak\kappa)$ holds.
So, by Remark~\ref{thm58b},  $\onto^+(J^{\bd}[\kappa],\theta^+)$ holds.
In particular, $\onto(J^\bd[\kappa],\theta)$ does.
Next, suppose that $\ubd( J^\bd[\kappa],\allowbreak\varkappa)$ holds for some $\varkappa$ with $\theta<\varkappa<\kappa$,.
There are two cases to consider:
\begin{itemize}
\item[$\br$] If $\theta^+=\varkappa$, then since $\theta^+<\kappa$, Lemma~\ref{lemma64}(2) implies that $\ubd(J^\bd[\kappa],\theta)$ holds.
If, in addition $\theta$ is regular, then $\onto(J^\bd[\kappa],\theta)$ holds by Theorem~\ref{thm58}.
\item[$\br$] Otherwise, $\theta^+<\varkappa<\kappa$,
so, by Theorem~\ref{thm58}, $\onto(J^\bd[\kappa],\theta^+)$ holds.
In particular, $\onto(J^\bd[\kappa],\theta)$ holds.
\end{itemize}

(2) By Corollary~\ref{prop41}  and Theorem~\ref{thm54}, we may assume that $\theta^+<\kappa$. Now, there are three cases to consider:
\begin{itemize}
\item[$\br$] If $\varkappa=\kappa$, then by Corollary~\ref{cor610}(1), $\onto^{++}(\ns_\kappa\restriction S,\theta^+)$ holds,
hence, so does $\onto(\ns_\kappa\restriction S,\theta)$.

\item[$\br$] If $\theta^+=\varkappa<\kappa$, then the result follows from Lemma~\ref{lemma64}(2) and Theorem~\ref{thm58}.

\item[$\br$] If $\theta^+<\varkappa<\kappa$, then 
by Theorem~\ref{thm58}, $\onto(\ns_\kappa\restriction S,\theta^+)$ holds,
hence, so does $\onto(\ns_\kappa\restriction S,\theta)$.\qedhere
\end{itemize}
\end{proof}

\begin{prop}\label{ontofinite} Suppose $\aleph_0<\kappa\leq 2^{\aleph_0}$.  
Then $\onto(\{\aleph_0\}, J^\bd[\kappa], n)$ holds for every positive integer $n$.
\end{prop}
\begin{proof}  Fix an injective sequence $\langle r_\beta \mid \beta< \kappa \rangle$ of elements of ${}^\omega 2$. 
Let $n$ be a positive integer, and fix a bijection $\pi: \omega \leftrightarrow{}^n(\omega\times2)$. 
Then let $c:[\kappa]^2 \rightarrow n$ be any colouring that satisfies that for all $\eta<\omega \leq \beta < \kappa$,
if $\pi(\eta)=\langle (m_0,i_0), \ldots ,(m_{n-1},i_{n-1})\rangle$ and
there is $j< n$ such that $r_\beta(m_j)\neq i_j$, then $c(\eta, \beta)$ is equal to the least such $j$.

To see this works, let $B$ be some cofinal subset of $\kappa$. 
As $\kappa>\omega$, we may assume that $\min(B)\ge\omega$. 
For any $x\in{}^{<\omega}2$, denote:
$$B_{x}:=\{\beta\in B\mid x\sq r_\beta\},$$
and then let
$$\tree(B):=\{ x \in{}^{<\omega}2\mid \forall i< 2 (B_{x{}^\smallfrown\langle i\rangle}\neq\emptyset)\}.$$

For all $\alpha<\beta<\kappa$, let  $\Delta(\alpha,\beta)$ denote the first $k<\omega$ such that $r_\alpha(k)\neq r_\beta(k)$.
Note that if there exists a large enough $l<\omega$ such that $\Delta``[B]^2\s l$,
then $\beta\mapsto(r_\beta\restriction l)$ would form an injection from the infinite set $B$ to the finite set ${}^l2$,
which is impossible.		
It thus follows that $(\tree(B),{\sq})$ is a finitely-splitting infinite tree, and then, by K\"onig's lemma,
it admits an infinite chain. In particular, we may pick $y \in \tree(B)$ for which the set $\{x \in\tree(B)\mid x\sq y\}$ has size $n+1$.
Let $\langle m_0, \ldots ,m_{n-1}\rangle$ denote the increasing enumeration of the following set
$$\{ m<\dom(y)\mid (y\restriction m)\in \tree(B)\}.$$ 
For each $j<n$, set $i_j:=y(m_j)$.
Then set $\eta:=\pi^{-1}\langle (m_0,i_0), \ldots,(m_{n-1},i_{n-1})\rangle$.
To see that $c[\{\eta\}\circledast B] = n$, let $j<n$  be arbitrary.
As $x:=y\restriction m_j$ is in $\tree(B)$, so is $x':=x{}^\smallfrown\langle 1-i_j\rangle$.
Pick $\beta\in B_{x'}$. Then $c(\eta,\beta)=j$.
\end{proof}
		
\begin{prop}\label{prop50} $\ubd([\aleph_0]^1,J^\bd[\aleph_0], \aleph_0)$ holds.
\end{prop}
\begin{proof} Let $c:[\omega]^2 \rightarrow \omega$ be the colouring obtained by declaring that for all $n<m<\omega$, $c(n,m)$ is the floor of $m\over 2$. It is clear that this works.  
\end{proof}
\begin{remark}\label{noupgraderemark} In contrast with Lemma~\ref{ehpumpubd}, 
the colouring $c$ of the preceding proof is an example of a witness to $\ubd(J,\theta)$ which does not witness $\ubd^+(J,\theta)$.
\end{remark}

\begin{thm}\label{prop51} Suppose that $\kappa$ is a singular cardinal. Then:
\begin{enumerate}[(1)]
\item $\ubd([\kappa]^1,J^\bd[\kappa], \cf(\kappa))$ holds;
\item $\ubd(\{\cf(\kappa)\},J^\bd[\kappa], \theta)$ holds for every  $\theta<\cf(\kappa)$;
\item $\onto(\{\cf(\kappa)\},J^\bd[\kappa], \theta)$ holds for every regular $\theta<\cf(\kappa)$;
\item $\onto(\{\nu\},J^\bd[\kappa], \theta)$ holds for every $\theta\le\cf(\kappa)$ such that $\nu:=\mathcal C(\cf(\kappa),\theta)$ is $<\kappa$.
\end{enumerate}
\end{thm}
\begin{proof} Set $\varkappa:=\cf(\kappa)$. 
Fix a cofinal subset $x\s\kappa\setminus\varkappa$ of order-type $\varkappa$.
Define a map $\pi:\kappa\rightarrow\varkappa$ via $\pi(\beta):=\otp(x\cap\beta)$.
Note that $\pi(\beta)<\beta$ for all $\beta<\kappa$,
and that, for every $B\in (J^\bd[\kappa])^+$, $\pi[B]$ is unbounded in $\varkappa$.

(1) Define an upper-regressive map $c:[\kappa]^2\rightarrow\varkappa$ via $c(\eta,\beta):=\pi(\beta)$. 
Evidently, $c$ witnesses that $\ubd([\kappa]^1,J^\bd[\kappa], \varkappa)$ holds.

(2) Let $\theta<\varkappa$ be some cardinal. By Clause~(4), we may assume that $\theta$ is infinite.

If $\varkappa$ is a successor cardinal,
then, by Corollary~\ref{cor65}(1), $\ubd(J^{\bd}[\varkappa],\theta)$ holds.
So, by Proposition~\ref{singularprojection}, $\ubd(\{\varkappa\},\allowbreak J^\bd[\kappa], \theta)$ holds, as well.

If $\varkappa$ is a not a successor cardinal, then $\varkappa$ is inaccessible,
and so by Fact~\ref{clubguessingfact},
we may fix a club-guessing $C$-sequence $\langle C_\eta\mid \eta\in E^\varkappa_\theta\rangle$ 
with $\otp(C_\eta)=\theta$ for all $\eta\in E^\varkappa_\theta$.
Pick any upper-regressive colouring $c:[\kappa]^2\rightarrow\theta$ such that for all $\eta<\varkappa\le\beta<\kappa$ with $\cf(\eta)=\theta$,
$$c(\eta,\beta)=\sup(\otp(C_{\eta}\cap\pi(\beta)).$$
An argument as in the proof of Claim~\ref{claim641} makes it clear that $c$ witnesses that $\onto(\{\varkappa\},J^\bd[\kappa], \theta)$ holds.
In particular, $\ubd(\{\varkappa\},\allowbreak  J^\bd[\kappa], \theta)$ holds.

(3)  Let $\theta<\varkappa$ be some regular cardinal.
By the proof of Clause~(2), we may assume that $\theta$ is infinite and $\varkappa$ is not inaccessible.
So, $\varkappa$ is a successor cardinal.
Then by Fact~\ref{ulamoriginal} and Lemma~\ref{ubdplus}, $\ubd^+(J^{\bd}[\kappa],\allowbreak\kappa)$ holds.
So, by Remark~\ref{thm58b},  $\onto(J^{\bd}[\varkappa],\theta)$ holds. We finish by observing that this, combined with Proposition~\ref{singularprojection}, implies that $\onto(\{\varkappa\},J^\bd[\kappa], \theta)$ holds.

(4) Should be clear at this stage.
\end{proof}

\begin{cor} Suppose that $\kappa$ is a singular cardinal. Then:
\begin{enumerate}[(1)]
\item $\ubd(J^\bd[\kappa], \theta)$ holds for every $\theta\le\cf(\kappa)$;
\item $\onto(J^{\bd}[\kappa],\theta)$ holds whenever $\theta^+<\cf(\kappa)$;
\item If $2^{\cf(\kappa)}<\kappa$, then $\onto(J^\bd[\kappa], \theta)$ holds for every $\theta\le\cf(\kappa)$.\qed
\end{enumerate}
\end{cor}

\subsection{Fodor's question}\label{subsectionfodor}
In the 1970's (see \cite{MR0369081}), Fodor asked whether, given a sequence
$\vec S=\langle S_i\mid i<\kappa\rangle$ of stationary subsets of some regular uncountable cardinal $\kappa$, 
there exists a pairwise disjoint sequence $\langle S_i'\mid i<\kappa\rangle$ such that, for each $i<\kappa$,
$S_i'$ is a stationary subset of $S_i$. 
Note that the special case in which $\vec S$ is a constant sequence coincides with Solovay's decomposition theorem.

The general answer to Fodor's question is negative.
For instance, one problematic scenario is when $\ns_{\aleph_1}$ is $\aleph_1$-dense and
$\vec S$ enumerates a dense subset of $\ns_{\aleph_1}$. 
An affirmative answer to a weakening of Fodor's question is implicit in \cite[Lemma~1.15]{paper29},
as was announced in \cite[Lemma~1.3]{paper38}. It reads as follows:
\begin{lemma}[implicit in \cite{paper29}] Suppose that
$\langle S_i\mid i<\theta\rangle$ is a sequence of stationary subsets of a regular uncountable cardinal $\kappa$, with $\theta\le\kappa$.
Then there exists a sequence $\langle S_i'\mid i\in I\rangle$ of pairwise disjoint stationary sets such that:
\begin{enumerate}
\item $S_i'\s S_i$ for every $i\in I$;
\item $I$ is a cofinal subset of $\theta$.
\end{enumerate}
\end{lemma}
\begin{proof}[Proof sketch] First, in light of Clause~(ii), we may assume that $\theta$ is an infinite regular cardinal.
Second, by \cite[Remark~1.5]{paper29} (see also Corollary~\ref{amensubset}), we may assume that,
for every $i<\theta$, $S_i$ carries an amenable $C$-sequence.
For each $i<\theta$, let $\Lambda^i:=\{ \iota<\theta\mid S_\iota\cap S_i\text{ is stationary}\}$.
Now, there are two cases to consider:

$\br$ If there exists $I\in[\theta]^\theta$ such that, for every pair $\iota<i$ of elements of $I$,
$\iota\notin \Lambda^i$, then fix such an $I$ and, for every $i\in I$, let $S_i':=S_i\setminus\bigcup_{\iota\in I\cap i}S_\iota$.
Clearly, $\langle S_i'\mid i\in I\rangle$ is as sought.

$\br$ Otherwise, we must be able to find $i<\theta$ such that $|\Lambda^i|=\theta$.
In this case, set $\Gamma:=S_i\cap\acc(\kappa)$, and, for every $\iota\in\Lambda^i$, denote $\Omega^\iota:=S_\iota\cap\Gamma$.
Now appeal to \cite[Lemma~1.15]{paper29} with  $\langle \Omega^\iota\mid \iota\in \Lambda^i\rangle$.
\end{proof}

The results of this section provide the following improvement in the case that $\kappa=\theta^+$:
\begin{cor} For every infinite cardinal $\theta$ and every sequence $\langle S_i\mid i<\theta\rangle$ of stationary subsets of $\theta^+$,
there exists a pairwise disjoint sequence $\langle S_i'\mid i<\theta\rangle$ such that,
for every $i<\theta$, $S_i'$ is a stationary subset of $S_i$.
\end{cor}
\begin{proof} By Corollary~\ref{prop41} and Theorem~\ref{thm54}, in particular, $\ubd^{++}(\ns_{\theta^+},\theta)$ holds.
The conclusion now follows immediately.
\end{proof}
\begin{remark} The preceding already found an application in \cite[Footnote~3]{paper44}.
\end{remark}

In the general case, we have an answer which is affirmative only in a \emph{dense} sense:
\begin{cor} For every stationary subset $S$ of a regular uncountable cardinal $\kappa$, there exists a stationary $S'\s S$ for which the following hold.
For every sequence $\langle S_i\mid i<\theta\rangle$ of stationary subsets of $S'$ with $\theta<\kappa$,
there exists a pairwise disjoint sequence $\langle S_i'\mid i<\theta\rangle$ such that,
for every $i<\theta$, $S_i'$ is a stationary subset of $S_i$.

If $\square(\kappa,{<}\mu)$ holds for some $\mu<\kappa$, then one can moreover take $S':=S$.
\end{cor}
\begin{proof} If $\kappa$ is a successor cardinal, then by the preceding corollary, one can take $S':=S$.
If $\kappa$ is limit cardinal,
then, for every $\theta<\kappa$ there exists $\theta'\in\reg(\kappa)$ above $\theta$,
hence the subset $S'\s S$ is given by Corollary~\ref{cor71}.
If $\kappa$ is a limit cardinal and $\square(\kappa,{<\mu})$ holds for some $\mu<\kappa$, then the result follows from Corollary~\ref{cor610}(3).
\end{proof}

\section{\texorpdfstring{$\onto$}{onto} with maximal colours}\label{sectionontomax}
In Sections~\ref{sectionstronglyamenable} and \ref{sectionamenable} we have studied the behaviour of the $\ubd(\ldots)$ principle for two natural ideals when the number of colours is maximal.
Here we study the $\onto(\ldots)$ principle for the maximal number of colours.
Colourings of this form, more specifically the \emph{onto mapping principle of Sierpi\'{n}ski} (see Fact~\ref{sierpinski}), 
were the starting point of this research project as we have already mentioned in the introduction.

In all of our results in this section we require guessing principles ranging in strength from $\non(\mathcal M) = \aleph_1$ to $\p^\bullet(\kappa,\kappa^+,{\sq},1)$ (see Theorem~\ref{pbullet}).
\begin{fact}\label{sierpinski} The following are equivalent:
\begin{enumerate}[(1)]
\item $\non(\mathcal M)=\aleph_1$, that is, there exists a nonmeagre set of reals of size $\aleph_1$;
\item $\onto(\{\aleph_0\}, J^\bd[\aleph_1], \aleph_1)$;
\item $\onto([\aleph_0]^{\aleph_0},J^{\bd}[\aleph_1],\aleph_1)$;
\item $\onto([\aleph_1]^{\aleph_0},J^{\bd}[\aleph_1],\aleph_1)$.
\end{enumerate}
\end{fact}
\begin{remark} The above equivalence is due to various authors. Sierpi\'{n}ski \cite{sierpinski1934hypothese} showed that $(2)$ follows from $2^{\aleph_0}=\aleph_1$ 
and this principle is called \emph{Sierpi\'{n}ski's onto mapping principle}. 
Todor{\v{c}}evi{\'c} \cite[pp.~290--291]{TodActa} proved the equivalence $(4)\iff (3)$.
The same argument  shows that for every infinite regular cardinal $\lambda$,
$\onto([\lambda]^{\lambda},J^{\bd}[\lambda^+],\lambda^+)$ implies the a priori stronger $\onto([\lambda^+]^{\lambda},J^{\bd}[\lambda^+],\lambda^+)$,
and Lemma~\ref{lemma89}(2) below establishes the same implication for $\lambda$ a singular strong limit.
Miller \cite{miller} showed that $(3)\iff(2)\implies(1)$. 
Finally,	Guzm\'{a}n \cite{MR3694336} showed that $(2)\impliedby(1)$ and hence that Sierpi\'{n}ski's onto mapping principle is equivalent to another classical statement in set theory. 
See \cite[\S3]{paper50} for a detailed treatment of Sierpi\'{n}ski's onto mapping principle including more equivalent versions.
\end{remark}

\begin{lemma}[implicit in {\cite[Lemma~14.1]{EHM}}]\label{sticklemma} Suppose that $\kappa=\theta^+$ is a successor cardinal.
\begin{enumerate}[(1)]
\item If $\kappa^{<\kappa}=\kappa$, then $\onto([\kappa]^\theta,J^{\bd}[\kappa],\kappa)$ holds;
\item If $\stick(\kappa)$ holds, then so does $\onto([\kappa]^\kappa,J^{\bd}[\kappa],\kappa)$.
\end{enumerate}
\end{lemma}
\begin{proof} If $\kappa^{<\kappa}=\kappa$, then let $\mathcal A:=[\kappa]^\theta$,
and if $\stick(\kappa)$ holds, then let $\mathcal A:=[\kappa]^\kappa$.
In both cases, we may fix a sequence of functions $\langle g_\alpha\mid \alpha<\kappa\rangle$ such that:
\begin{itemize}
\item for every $\alpha<\kappa$, $\dom(g_\alpha)\in[\alpha]^\theta$;
\item for every function $g:A\rightarrow\kappa$ with $A\in\mathcal A$, for some $\alpha<\kappa$, $g_\alpha\s g$.
\end{itemize}
	
Now, for every ordinal $\beta$ with $\theta \leq \beta<\kappa$, fix a surjection $e_\beta:\theta\rightarrow\beta$, 
and then define an injective sequence $\langle \eta^j_\beta\mid j<\theta\rangle$ of ordinals in $\beta$, by recursion on $j<\theta$:
$$\eta^j_\beta:=\min(\dom(g_{e_\beta(j)})\setminus\{ \eta_\beta^{i}\mid i<j\}).$$
Pick any colouring $c:[\kappa]^2\rightarrow\kappa$ satisfying that for all $\theta \leq \beta<\kappa$ and $j<\theta$, 
$$c(\eta^j_\beta,\beta)=g_{e_\beta(j)}(\eta_\beta^j).$$

To see that $c$ witnesses $\onto(\mathcal A,J^{\bd}[\kappa],\kappa)$, let $A\in\mathcal A$ and $B\in[\kappa]^\kappa$ be arbitrary.
Towards a contradiction, suppose that for every $\eta\in A$, $c[\{\eta\}\circledast B]\neq\kappa$.
Define $g:A\rightarrow\kappa$ via $g(\eta):=\min(\kappa\setminus c[\{\eta\}\circledast B])$.
Pick $\alpha<\kappa$ such that $g_\alpha\s g$, and then let $\beta\in B$ be above $\alpha$ and $\theta$. Let $j<\theta$ be such that $e_\beta(j)=\alpha$.
Write $\eta:=\eta_\beta^j$. Then $\eta\in\dom(g_\alpha)\s\dom(g)=A$ and $$c(\eta,\beta)=g_{e_\beta(j)}(\eta^j_\beta)=g_\alpha(\eta)=g(\eta),$$
contradicting the fact that $g(\eta)\notin c[\{\eta\}\circledast B]$.
\end{proof}

\begin{lemma}\label{lemma47}  Suppose $S$ is a stationary subset of a regular uncountable cardinal $\kappa$.
\begin{enumerate}[(1)]
\item If $\diamondsuit^*(S)$ holds, then so does $\onto(\ns_\kappa\restriction S,\kappa)$;
\item If $\diamondsuit(S)$ holds, then so does $\onto^-(S,\kappa)$.
\end{enumerate}
\end{lemma}
\begin{proof} (1) Assuming $\diamondsuit^*(S)$, we may fix a matrix $\langle g_\beta^\eta: \beta \rightarrow \beta\mid \beta\in S,\eta<\beta\rangle$
with the property that, for every function $g:\kappa\rightarrow\kappa$, there are club many $\beta\in S$ for which,
for some $\eta<\beta$, $g\restriction\beta=g^\eta_\beta$.
Pick $c:[\kappa]^2\rightarrow\kappa$ that satisfies  $c(\eta,\beta)=g^\eta_\beta(\eta)$ for all $\beta\in S$ and $\eta<\beta$.

We claim that $c$ witnesses $\onto(\ns_\kappa\restriction S,\kappa)$.
So, towards a contradiction, suppose that $B\s S$ is a stationary set such that, for every $\eta<\kappa$, $c[\{\eta\}\circledast B]\neq\kappa$.
Define $g:\kappa\rightarrow\kappa$ via $g(\eta):=\min(\kappa\setminus c[\{\eta\}\circledast B])$. Now, pick $\beta\in B$ and $\eta<\beta$ such that $g\restriction\beta=g^\eta_\beta$.
Then $c(\eta,\beta)=g^\eta_\beta(\eta)=g(\eta)$, contradicting the fact that $g(\eta)\notin c[\{\eta\}\circledast B]$.

(2) Assuming $\diamondsuit(S)$, we may fix a sequence of functions $\langle g_\beta: \beta \rightarrow\beta\mid \beta\in S\rangle$
with the property that, for every function $g:\kappa\rightarrow\kappa$, there are stationarily many $\beta \in S$ with $g_\beta=g\restriction\beta$.
Pick $c:[\kappa]^2\rightarrow\kappa$ that satisfies  $c(\eta,\beta)=g_\beta(\eta)$ for all $\beta\in S$ and $\eta\in\dom(g_\beta)$. 
We claim that $c$ witnesses a strong form of $\onto^-(S,\kappa)$ in which $\eta_0=\eta_1$.

Towards a contradiction, suppose that we are given a club $D\s\kappa$ and a regressive map $f:S\cap D\rightarrow \kappa$ 
such that, for any $\eta< \kappa$, 
$$c[\{\eta\}\circledast\{\beta\in S\cap D\mid f(\beta)=\eta\}]\neq\kappa.$$
Define a function $g:\kappa\rightarrow\kappa$ via 
$$g(\eta):=\min(\kappa\setminus c[\{\eta\}\circledast\{\beta\in S\cap D\mid f(\beta)=\eta\}]).$$
Now, pick an infinite $\beta\in S\cap D$ such that $g_\beta=g\restriction\beta$. Set $\eta:=f(\beta)$.
Then $c(\eta,\beta)=g_\beta(\eta)=g(\eta)$, contradicting the fact that 
\[g(\eta)\notin c[\{\eta\}\circledast\{\beta\in S\cap D\mid f(\beta)=\eta\}].\qedhere\]
\end{proof}

\begin{lemma} \label{prop47}
Suppose that $\kappa=\kappa^{<\kappa}$ is a limit cardinal.
For any colouring $c:[\kappa]^2\rightarrow\kappa$,
there exists a corresponding colouring $d:[\kappa]^2\rightarrow\kappa$ satisfying the following.
\begin{enumerate}[(1)]
\item For any ideal $J$ over $\kappa$, 
if $c$ witnesses	$\onto(J,\kappa)$, then $d$ witnesses $\onto([\kappa]^\kappa,\allowbreak J,\kappa)$;
\item For any $\kappa$-complete ideal $J$ over $\kappa$, 
if $c$ witnesses	$\onto^+(J,\kappa)$, then $d$ witnesses $\onto^+([\kappa]^\kappa,J,\kappa)$.
\end{enumerate}
\end{lemma}
\begin{proof} Fix a sequence $\langle \kappa_\eta\mid \eta<\kappa\rangle$ 
of cardinals such that, for every $\eta<\kappa$, $(\sum_{\zeta<\eta}\kappa_\zeta)<\kappa_\eta<\kappa$.
For every $\eta<\kappa$, let $\Phi_\eta:=\bigcup\{{}^{A}\kappa\mid A\in [\kappa]^{\kappa_\eta}\}$,
and then fix an injective enumeration $\langle \phi_\eta^\tau\mid\tau<\kappa\rangle$ of $\Phi_\eta$.
Fix a surjection $\sigma:\kappa\rightarrow\kappa$ such that the preimage of any singleton is cofinal in $\kappa$.

Let $c:[\kappa]^2\rightarrow\kappa$ be any colouring.
For all $\eta<\beta<\kappa$, denote $\psi_\eta^\beta:=\phi_\eta^{\sigma(c(\eta,\beta))}$.
Then, pick an upper-regressive map $d:[\kappa]^2\rightarrow\kappa$ satisfying that,
for all $\alpha<\beta<\kappa$, if there exists an ordinal $\eta<\kappa$ such that $\alpha\in\dom(\psi^\beta_\eta)$ and
$\psi_\eta^\beta(\alpha)<\beta$, then $d(\alpha,\beta)=\psi_\eta^\beta(\alpha)$ for the least such $\eta$.

(1) Suppose that $J$ is an ideal over $\kappa$, and that $c$ witnesses $\onto(J,\kappa)$.
To see that $d$ witnesses $\onto([\kappa]^\kappa,J,\kappa)$,
fix arbitrary $A\in[\kappa]^\kappa$ and $B\in J^+$.
Towards a contradiction, suppose that there exists a function $f:A\rightarrow\kappa$
such that $f(\alpha)\notin d[\{\alpha\}\circledast B]$ for all $\alpha\in A$.
As $B\in J^+$, let us fix $\eta<\kappa$ with $c[\{\eta\}\circledast B]=\kappa$.
Pick $A'\in[A]^{\kappa_\eta}$ and then find $\tau<\kappa$ such that $\phi_\eta^\tau=f\restriction A'$.
Let $\epsilon:=\max\{\eta,\sup(A'),\sup(f[A'])\}+1$.
By the choice of $\eta$, $c[\{\eta\}\times (B\setminus\epsilon)]$ is co-bounded in $\kappa$,
so we may fix $\beta\in B\setminus\epsilon$ with $\sigma(c(\eta,\beta))=\tau$.

As $|\bigcup_{\zeta<\eta}\dom(\psi_\zeta^\beta)|\le(\sum_{\zeta<\eta}\kappa_\zeta)<\kappa_\eta=|A'|$,
it follows that we may pick $\alpha\in A'\setminus\bigcup_{\zeta<\eta}\dom(\psi_\zeta^\beta)$,
so that $\eta$ is the least ordinal to satisfy $\alpha\in\dom(\psi^\beta_\eta)$.
We have that $\psi_\eta^\beta(\alpha)=\phi_\eta^{\sigma(c(\eta,\beta))}(\alpha)=\phi_\eta^\tau(\alpha)=f(\alpha)<\epsilon\le\beta$,
and hence $d(\alpha,\beta)=f(\alpha)$, contradicting the choice of $f$.

(2) Suppose that $J$ is a $\kappa$-complete ideal over $\kappa$, and that $c$ witnesses $\onto^+(J,\kappa)$.
To see that $d$ witnesses $\onto^+([\kappa]^\kappa,J,\kappa)$,
fix arbitrary $A\in[\kappa]^\kappa$ and $B\in J^+$.
Towards a contradiction, suppose that there exists a function $f:A\rightarrow\kappa$
such that, for every $\alpha\in A$, $\{\beta\in B\setminus(\alpha+1)\mid d(\alpha,\beta)=f(\alpha)\}\in J$.
By the hypothesis on $c$, we may fix $\eta<\kappa$ such that $\{\beta\in B\setminus(\eta+1)\mid c(\eta,\beta)=\tau\}\in J^+$ for all $\tau<\kappa$.

Pick $A'\in[A]^{\kappa_\eta}$ and then find $\tau<\kappa$ such that $\phi_\eta^\tau=f\restriction A'$.
Let $\epsilon:=\max\{\eta,\sup(A'),\sup(f[A'])\}+1$.
By the choice of $\eta$, $B':=\{\beta\in B\setminus\epsilon\mid \sigma(c(\eta,\beta))=\tau\}$ is in $J^+$.
As in the previous case, for each $\beta\in B'$, 
we may pick $\alpha_\beta\in A'\setminus\bigcup_{\zeta<\eta}\dom(\psi_\zeta^\beta)$,
so that $\eta$ is the least ordinal to satisfy $\alpha_\beta\in\dom(\psi^\beta_\eta)$.
As $|A'|<\kappa$ and $J$ is $\kappa$-complete, there exists some $\alpha\in A'$ for which 
$B'':=\{\beta\in B'\mid \alpha_\beta=\alpha\}$ is in $J^+$.
Now, for every $\beta\in B''$, we have that $\psi_\eta^\beta(\alpha)=\phi_\eta^{\sigma(c(\eta,\beta))}(\alpha)=\phi_\eta^\tau(\alpha)=f(\alpha)<\epsilon\le\beta$,
and hence $d(\alpha,\beta)=f(\alpha)$, contradicting the choice of $f$.
\end{proof}

\begin{cor}\label{cor811} Suppose that $\kappa$ is a regular uncountable cardinal.
\begin{itemize}
\item If there exists a stationary $S\s\kappa$ that does not reflect at regulars
and $\diamondsuit(S)$ holds, then $\onto([\kappa]^\kappa,J^\bd[\kappa],\kappa)$ holds.
In particular:
\item If $\kappa$ is a non-Mahlo cardinal
and $\diamondsuit(\kappa)$ holds, then so does $\onto([\kappa]^\kappa,J^\bd[\kappa],\kappa)$.
\end{itemize}
\end{cor}
\begin{proof} Suppose that $S\s\kappa$ is stationary and $\diamondsuit(S)$ holds. 
In particular, if $\kappa$ is a successor cardinal,
then by Lemma~\ref{sticklemma}, $\onto([\kappa]^\kappa,J^{\bd}[\kappa],\kappa)$ holds, and we are done.
From now on, suppose that $\kappa$ is a limit cardinal,
so, by Lemma~\ref{prop47}, it suffices to prove that $\onto(J^{\bd}[\kappa],\kappa)$ holds.

By Lemma~\ref{lemma47}(2), $\onto^-(S,\kappa)$ holds.
Suppose that $S$ does not reflect at regulars.
Then, by Corollary~\ref{hajnaltocolouring2}, $\ubd^*(J^\bd[\kappa],\{ S\cap C\})$ holds for some club $C\s\kappa$.
Finally, by Lemma~\ref{minustoplus}, $\onto(J^{\bd}[\kappa],\kappa)$ holds, and we are done.
\end{proof}
\begin{remark} The preceding finding plays a key role in the proof of one clause of \cite[Theorem~B]{paper45}.
\end{remark}

For a list of sufficient conditions for the hypothesis of the following lemma to hold, 
the reader may consult \cite[Theorem 6.1]{paper23}.
\begin{lemma}\label{pbullet} Suppose that $\kappa$ is a regular uncountable cardinal.
If $\p^\bullet(\kappa,\kappa^+,{\sq},1)$ holds, 
then so does $\onto([\kappa]^\kappa,J^{\bd}[\kappa],\kappa)$.
\end{lemma}
\begin{proof} According to \cite[\S5]{paper23}, $\p^\bullet(\kappa,\kappa^+,{\sq},1)$ 
provides us with a sequence $\langle\mathcal C_\beta\mid\beta<\kappa\rangle$ satisfying the following:
\begin{enumerate}
\item for every $\beta<\kappa$, $\mathcal C_\beta$ is a nonempty collection of functions 
$C:\mathring{C}\rightarrow  H_\kappa$ such that $\mathring{C}$ is a closed subset of $\beta$ with $\sup(\mathring C)=\sup(\beta)$;
\item for all $\beta<\kappa$, $C\in\mathcal C_\beta$ and $\alpha\in\acc(\mathring C)$, $C\restriction \alpha\in\mathcal C_{\alpha}$;
\item for all $\Omega\s H_\kappa$ and $p\in H_{\kappa^+}$, there exists $\delta\in\acc(\kappa)$
such that, for all $C\in\mathcal C_\delta$ and $\epsilon<\delta$,
there exists an elementary submodel $\mathcal M\prec H_{\kappa^+}$ with $\{\epsilon,p\}\in\mathcal M$ 
such that $\tau:=\mathcal M\cap\kappa$ is in $\nacc(\mathring C)$ and $C(\tau)=\mathcal M\cap\Omega$.
\end{enumerate}
For our purposes, it suffices to consider the following special case of (iii):
\begin{enumerate}
\item[(iii')] for every function $g: \kappa\rightarrow \kappa$, there exists $\delta\in\acc(\kappa)$
such that, for all $C\in\mathcal C_\delta$, 
$$\sup\{\tau\in \nacc(\mathring C)\mid g[\tau]\s\tau\ \&\ C(\tau)=g\restriction \tau\}=\delta.$$
\end{enumerate}

In particular, $\kappa^{<\kappa}=\kappa$.
So, by Lemma~\ref{sticklemma}, we may assume that $\kappa$ is a limit cardinal,
and then, by Lemma~\ref{prop47}, it suffices to verify that $\onto(J^{\bd}[\kappa],\kappa)$ holds.

Fix a sequence $\vec C=\langle C_\beta\mid \beta<\kappa\rangle$ such that, $C_\beta\in\mathcal C_\beta$ for all $\beta<\kappa$.
Similar to the proof of Lemma~\ref{proxyapp1}, we shall conduct walks on ordinals along 
$\vec{\mathring C}:=\langle \mathring C_\beta\mid \beta<\kappa\rangle$.

Define an upper-regressive colouring $c:[\kappa]^2\rightarrow\kappa$ as follows. Given $\eta<\beta<\kappa$,
let $\gamma:=\min(\im(\tr(\eta,\beta)))$
and then let  $c(\eta,\beta):=C_\gamma(\min(\mathring C_\gamma\setminus(\eta+1)))(\eta)$
provided that the latter is a well-defined ordinal less than $\beta$; otherwise, let $c(\eta,\beta):=0$.

Now, let $B\in[\kappa]^\kappa$ be arbitrary.
Towards a contradiction, suppose that there exists a function $g:\kappa\rightarrow\kappa$ such that,
for every $\eta<\kappa$, $g(\eta)\notin c[\{\eta\}\circledast B]$.
Fix $\delta\in\acc(\kappa)$ as in Clause~(iii').
Pick $\beta\in B$ above $\delta$,
and then let $$\varepsilon:=\sup(\delta \cap \{\sup(\mathring C_{\gamma}\cap \delta)\mid \gamma\in\im(\tr(\delta,\beta))\}).$$
As in the proof of Lemma~\ref{proxyapp1}, $\varepsilon<\delta$ and there are two cases to consider:

$\br$ If $\delta\in\nacc(\mathring C_{\min(\im(\tr(\delta,\beta))})$,
then,  for every $\eta$ with $\varepsilon<\eta<\delta$,
$\tr(\eta,\beta)=\tr(\delta,\beta){}^\smallfrown\tr(\eta,\delta)$.
In this case, pick a large enough $\tau\in\nacc(\mathring C_\delta)$ with $g[\tau]\s\tau$ and $C_\delta(\tau)=g\restriction \tau$ for which $\eta:=\sup(\mathring C_\delta\cap\tau)$ is $>\varepsilon$.
Then $\min(\im(\tr(\eta,\beta)))=\delta$, $g(\eta)<\tau<\beta$, and hence 
$c(\eta,\beta)=C_\delta(\min(\mathring C_\delta\setminus(\eta+1)))(\eta)=C_\delta(\tau)(\eta)=g(\eta)$.

$\br$ Otherwise, $\delta\in\acc(\mathring C_\gamma)$, for $\gamma:=\min(\im(\tr(\delta,\beta)))$,
and, for every $\eta$ with $\varepsilon<\eta<\delta$,
$\tr(\eta,\beta)=\tr(\gamma,\beta){}^\smallfrown\tr(\eta,\gamma)$.
As $\delta\in\acc(\mathring C_\gamma)$, $C_\gamma\restriction \delta$ is in $\mathcal C_\delta$,
and hence we may pick 	$\tau\in\nacc(\mathring C_\gamma\cap\delta)$ with $g[\tau]\s\tau$ and $C_\gamma(\tau)=g\restriction\tau$ for which $\eta:=\sup(\mathring C_\gamma\cap\tau)$ is greater than $\varepsilon$.
Then $\min(\im(\tr(\eta,\beta)))=\gamma$, $g(\eta)<\tau<\beta$, and hence 
$c(\eta,\beta)=C_\gamma(\min(\mathring C_\gamma\setminus(\eta+1)))(\eta)=C_\gamma(\tau)(\eta)=g(\eta)$.
\end{proof}

By the next lemma, for every limit cardinal $\kappa$, if $\kappa$ is a strong limit or if $\aleph_\kappa>\kappa$, 
then $\onto(\{\kappa\},J^{\bd}[\kappa^+],\kappa)$ implies $\kappa^+\nrightarrow[\kappa;\kappa^+]^2_{\kappa^+}$.

\begin{lemma}\label{lemma89} Suppose that $\kappa$ is an infinite cardinal for which there exists a sequence of cardinals 
$\langle \kappa_\eta\mid \eta<\kappa\rangle$ and a family $\mathcal A\s[\kappa]^{<\kappa}$ of size $\kappa$ such that,
for every $\eta<\kappa$:
\begin{itemize}
\item $(\sum_{\zeta<\eta}\kappa_\zeta)<\kappa_\eta<\kappa$;
\item for every $A\in[\kappa]^\kappa$, there exists $A'\in\mathcal A$ with $|A'\cap A|=|A'|=\kappa_\eta$.
\end{itemize}
Then:
\begin{enumerate}[(1)]
\item If $\kappa$ is regular and $\onto(J^{\bd}[\kappa],\kappa)$ holds, then so does $\kappa\nrightarrow[\kappa;\kappa]^2_\kappa$;
\item If $\onto(\{\kappa\},J^{\bd}[\kappa^+],\kappa)$ holds, then so does $\kappa^+\nrightarrow[\kappa;\kappa^+]^2_{\kappa^+}$.
\end{enumerate}
\end{lemma}
\begin{proof} The proof is similar to that of Lemma~\ref{prop47}. For conciseness, we focus on proving Clause~(2).
First, note that the hypotheses imply that we may fix
$\mathcal A\s[\kappa^+]^{<\kappa}$ of size $\kappa^+$ such that,
for every $A\in[\kappa^+]^\kappa$ and every $\eta<\kappa$, there exists $A'\in\mathcal A$ with $|A'\cap A|=|A'|=\kappa_\eta$.
For every $\eta<\kappa^+$, let $\Phi_\eta:=\{A'\times\{\gamma\}\mid A'\in{\mathcal A}\cap[\kappa^+]^{\kappa_\eta},\gamma<\kappa^+\}$ and note that these are all constant functions with domains in $\mathcal A\cap[\kappa^+]^{\kappa_\eta}$.
Fix an injective enumeration $\langle \phi_\eta^\tau\mid\tau<\kappa^+\rangle$ of $\Phi_\eta$.
Fix a surjection $\sigma:\kappa^+\rightarrow\kappa^+$ such that the preimage of any singleton is cofinal in $\kappa^+$.
Now, assuming that $\onto(\{\kappa\},J^{\bd}[\kappa^+],\kappa)$ holds,
Lemma~\ref{upbyone}(3) yields a colouring $c:\kappa\times\kappa^+\rightarrow\kappa^+$ witnessing $\onto(\{\kappa\},J^{\bd}[\kappa^+],\kappa^+)$.
For every pair $(\eta,\beta)\in\kappa\times\kappa^+$, set $\psi_\eta^\beta:=\phi_\eta^{\sigma(c(\eta,\beta))}$.
Finally, define a colouring $d:[\kappa^+]^2\rightarrow\kappa^+$ by letting, for all $\alpha<\beta<\kappa^+$:
$$d(\alpha,\beta):=\begin{cases}
0,&\text{if }\alpha\notin\bigcup_{\zeta<\kappa}\dom(\psi_\zeta^\beta);\\
\psi_\eta^\beta(\alpha),&\text{if }\eta=\min\{\zeta<\kappa\mid \alpha\in\dom(\psi_\zeta^\beta)\}.
\end{cases}$$

To see that $d$ witnesses $\kappa^+\nrightarrow[\kappa;\kappa^+]^2_{\kappa^+}$,
fix arbitrary $A\in[\kappa^+]^\kappa$ and $B\in[\kappa^+]^{\kappa^+}$ and a prescribed colour $\gamma<\kappa^+$;
we shall find $(\alpha,\beta)\in A\circledast B$ for which $d(\alpha,\beta)=\gamma$.

As $B\notin J^{\bd}[\kappa^+]$, let us fix $\eta<\kappa$ with $c[\{\eta\}\circledast B]=\kappa^+$.
Pick $A'\in\mathcal A$ with $|A'\cap A|=|A'|=\kappa_\eta$.
Find $\tau<\kappa^+$ such that $\phi_\eta^\tau=A'\times\{\gamma\}$.
Let $\epsilon:=\sup(A')+\kappa$.
By the choice of $\eta$, $c[\{\eta\}\times (B\setminus\epsilon)]$ is co-bounded in $\kappa^+$,
so we may fix $\beta\in B\setminus\epsilon$ with $\sigma(c(\eta,\beta))=\tau$.

As $|\bigcup_{\zeta<\eta}\dom(\psi_\zeta^\beta)|\le(\sum_{\zeta<\eta}\kappa_\zeta)<\kappa_\eta=|A'\cap A|$,
it follows that we may pick $\alpha\in (A'\cap A)\setminus(\bigcup_{\zeta<\eta}\dom(\psi_\zeta^\beta))$,
so that $d(\alpha,\beta)=\psi_\eta^\beta(\alpha)$.
Altogether $\alpha\in\dom(\psi_\zeta^\beta)\cap A' \cap A\cap\beta$ and $$d(\alpha,\beta)=\psi_\eta^\beta(\alpha)=\phi_\eta^{\sigma(c(\eta,\beta))}(\alpha)=\phi_\eta^\tau(\alpha)=\gamma,$$ as sought.
\end{proof}

\section{Failures and consistent failures}\label{sectionfailures}
Our focus in this section is the study of when the colouring principles do not hold. 
Some of these results require the assumption of large cardinals, and this is justified by our results from Sections~\ref{sectionstronglyamenable} and \ref{sectionamenable}. 
We finish by showing that the $\ubd^{++}$ principle with the maximal number of colours always fails.

\smallskip

In contrast to Proposition~\ref{prop50}, we have:
\begin{cor}  $\onto(J^{\bd}[\aleph_0],\aleph_0)$ fails.
\end{cor}
\begin{proof} By a standard diagonalization argument. Alternatively, by Lemma~\ref{prop47} and Ramsey's theorem.
\end{proof}

\begin{fact}[{\cite[Corollary~6.8]{MR2298476}}]\label{larson} If $\zfc$ is consistent, 
then it is consistent that $\onto(\ns_{\aleph_1},\aleph_1)$ fails. Also, $\bpfa$ implies that $\onto(\ns_{\aleph_1},\aleph_1)$ fails. 
\end{fact}
\begin{remark}
Evidently, $\onto(J^{\bd}[\aleph_1],\aleph_1)\implies \onto(\ns_{\aleph_1},\aleph_1)\implies \onto^-(\aleph_1,\aleph_1)$.
By Corollary~\ref{hajnaltocolouring2} and Lemma~\ref{minustoplus}, the last two are equivalent, 
and, by Lemma~\ref{upbyone}(1), the first two are equivalent, as well.
In contrast, in Section~\ref{sectionineffable}, we shall show that if $\kappa$ is an ineffable cardinal,
then $\onto(\ns_\kappa,2)$ fails yet $\onto^-(\kappa,\kappa)$ holds.
\end{remark}

\begin{prop}\label{prop94}$\ubd(J^{\bd}[\kappa],\theta)$ fails whenever $\cf(\kappa)<\theta\le\kappa$.

In particular, $\onto(J^{\bd}[\kappa],\kappa)$ fails for every singular cardinal $\kappa$.
\end{prop}
\begin{proof} Some positive sets with respect to the ideal $J^{\bd}[\kappa]$ have size $\cf(\kappa)$.
\end{proof}

The next theorem shows that neither $\onto^-(\kappa,\kappa)$ nor $\sup\{\theta<\kappa\mid \onto(\ns_\kappa,\theta)\text{ holds}\}=\kappa$ imply $\onto(\ns_\kappa,\kappa)$.

\begin{thm}\label{kunenineff} If $\kappa$ is ineffable, then in some cofinality-preserving forcing extension:
\begin{enumerate}[(1)]
\item $\kappa$ is strongly inaccessible;
\item $\onto^{++}([\kappa]^\kappa,J^{\bd}[\kappa],\theta)$ holds for all $\theta<\kappa$;
\item $\onto^-(\kappa,\kappa)$ holds;
\item $\ubd(\ns_\kappa,\kappa)$ fails.
\end{enumerate}
\end{thm}
\begin{proof} The aforementioned Kunen's model $V^{\mathbb R}$ from \cite[\S3]{MR495118} where we start from an ineffable cardinal $\kappa$ has the following properties:
\begin{enumerate}
\item $\kappa$ is strongly inaccessible;
\item there is a $\kappa$-Souslin tree $\mathbb T$;
\item $\diamondsuit(\kappa)$ holds;
\item $V^{\mathbb R,\mathbb T}\models ``\kappa\text{ is ineffable}"$.
\end{enumerate}
As there is a $\kappa$-Souslin tree, $\kappa\nrightarrow[\kappa;\kappa]^2_\kappa$ holds, so by Proposition~\ref{prop46},
$\onto^{++}([\kappa]^\kappa,\allowbreak J^{\bd}[\kappa],\theta)$ holds for all $\theta<\kappa$.
By Lemma~\ref{lemma47}(2), $\onto^-(\kappa,\kappa)$ holds in $V^{\mathbb R}$. 
Finally, since $\mathbb T$ has the $\kappa$-cc and $V^{\mathbb R,\mathbb T}\models ``\kappa\notin \amen_\kappa"$, 
we conclude that $V^{\mathbb R}\models ``\kappa \notin \amen_\kappa"$ by Proposition~\ref{amenpullneg} and hence $\ubd(\ns_\kappa,\kappa)$ fails in $V^{\mathbb R}$ by Lemma~\ref{lemma44}. 
\end{proof}
\begin{thm}\label{kunenwc} If $\kappa$ is weakly compact, then in some cofinality-preserving forcing extension,
(1)---(3)  of Theorem~\ref{kunenineff} hold, and so does
\begin{itemize}
\item[(4)] $\ubd(J^\bd[\kappa],\kappa)$ fails.
\end{itemize}
\end{thm}
\begin{proof} The proof is almost exactly the same as that of Theorem~\ref{kunenineff}. 
In Kunen's model $V^{\mathbb R}$ from \cite[\S3]{MR495118} where we start from an weakly compact cardinal $\kappa$ has properties (i)---(iii), together with 
\begin{itemize}
\item[(iv)] $V^{\mathbb R,\mathbb T}\models ``\kappa\text{ is weakly compact}"$.
\end{itemize}
So as $V^{\mathbb R,\mathbb T}\models ``\kappa\notin \sa_\kappa"$, we conclude that $V^{\mathbb R}\models ``\kappa \notin \sa_\kappa"$ 
by Proposition~\ref{strongamenpullneg} and hence $\ubd(J^\bd[\kappa],\kappa)$ fails in $V ^{\mathbb R}$ by Lemma~\ref{lemma43}. 
\end{proof}

\begin{thm} Assuming the consistency of a weakly compact cardinal, the following is consistent:
\begin{itemize}
\item $\kappa$ is strongly inaccessible;
\item $\onto(\ns_\kappa, \kappa)$ holds;
\item $\ubd(J^\bd[\kappa], \omega)$ fails.
\end{itemize}
\end{thm}
\begin{proof} Work in $\mathsf{L}$ and suppose that $\kappa$ is a weakly compact cardinal which is not ineffable (e.g., the first weakly compact cardinal). 
As $\kappa$ is not ineffable,
a result of Jensen \cite{jensen1969some} (see \cite[Theorem~5.39]{MR3243739}) states that $\diamondsuit^*(\kappa)$ holds. 
So by Lemma~\ref{lemma47}(1) we have that $\onto(\ns_\kappa, \kappa)$ holds.
In addition, since $\kappa$ is weakly compact, by Proposition~\ref{reducetotwo}, $\ubd(J^\bd[\kappa],\omega)$ fails. 	
\end{proof}

\begin{lemma}\label{preservation}
Suppose that $\kappa$ is uncountable, $\theta\le\kappa$ is infinite and regular, $\mathbb P$ is a $\theta$-cc poset, and $S\s\kappa$ is stationary.
\begin{enumerate}[(1)]
\item If $\ubd(J^{\bd}[\kappa],\theta)$ fails, then it also fails in $V^{\mathbb P}$.	
\item If $\ubd(\ns_\kappa\restriction S,\theta)$ fails, then it also fails in $V^{\mathbb P}$.
\end{enumerate}
\end{lemma}
\begin{proof} Since $\mathbb P$ has the $\theta$-cc,
for every colouring $c:[\kappa]^2\rightarrow\theta$ in $V^{\mathbb P}$,
there exists a colouring $d:[\kappa]^2\rightarrow\theta$ such that $c(\alpha,\beta)\le d(\alpha,\beta)$ for every $(\alpha,\beta)\in[\kappa]^2$.

(1) If $\ubd(J^{\bd}[\kappa],\theta)$ fails, then we may fix $B\in[\kappa]^\kappa$ such that $d``[B]^2$ is bounded in $\theta$,
and then, in $V^{\mathbb P}$, $B$ is a $J^{\bd}[\kappa]$-positive such that $c``[B]^2$ is bounded in $\theta$,
so $B$ demonstrates that $c$ is not a witness to $\ubd(J^{\bd}[\kappa],\theta)$.

(2) If $\ubd(\ns_\kappa\restriction S,\theta)$ fails, then we may fix a stationary $B\s S$ such that $d``[B]^2$ is bounded in $\theta$.
As $\mathbb P$ has the $\kappa$-cc, $B$ remains stationary in $V^{\mathbb P}$,
so it demonstrates that $c$ is not a witness to $\ubd(\ns_\kappa\restriction S,\theta)$.
\end{proof}

\begin{thm}\label{thm97} If $\kappa$ is ineffable, then for every cardinal $\theta=\theta^{<\theta}<\kappa$, 
there is a cofinality-preserving forcing extension in which:
\begin{itemize}
\item $\kappa=2^\theta=2^{\theta^+}$;
\item $\onto(\{\theta\},J^{\bd}[\kappa],\theta)$ holds;
\item $\ubd(\ns_\kappa,\theta^+)$ fails.
\end{itemize}
\end{thm}
\begin{proof} Let $\mathbb P$ be Cohen's notion of forcing for adding $\kappa$ many subsets of $\theta$.
So, the conditions in $\mathbb P$ are functions $p:a\rightarrow\theta$ with $a\s \theta\times \kappa$ and $|a|<\theta$.
This is a $\theta$-closed notion of forcing, and since $\theta^{<\theta}=\theta$, it has the $\theta^{+}$-cc,
so that $\mathbb P$ preserves the cardinal structure.
By Fact~\ref{reducetoone}, as $\kappa$ is ineffable, $\ubd(\ns_\kappa,\theta^+)$ fails,
and then, by Lemma~\ref{preservation}, $\ubd(\ns_\kappa,\theta^+)$ fails in the extension, as well.

Let $G$ be $\mathbb P$-generic over $V$, and work in $V[G]$. Let $c:=\bigcup G$, so that $c$ is a function from $\theta\times\kappa$ to $\theta$.
Towards a contradiction, suppose that $\onto(\{\theta\},J^{\bd}[\kappa],\theta)$ fails, as witnessed by some $B\in[\kappa]^\kappa$.
In particular, there exists a function $g:\theta\rightarrow\theta$ such that $g(\eta)\notin c[\{\eta\}\circledast B]$ for all $\eta<\theta$.
However, by a standard density argument (see the proof of \cite[Theorem~27]{strongcoloringpaper}),
in $V[G]$, for all $A\in[\theta]^{\theta}$ and $B\in[\kappa]^{\kappa}$,
for every function $g:A\rightarrow\theta$, there exist $\eta\in A$ and $\beta\in B\setminus\theta$ such that $g(\eta)=c(\eta,\beta)$.
This is a contradiction.
\end{proof}

\begin{thm} \label{almostineffablecohen} If $\kappa$ is weakly compact, then for every cardinal $\theta=\theta^{<\theta}<\kappa$, 
there is a cofinality-preserving forcing extension in which:
\begin{itemize}
\item $\kappa=2^\theta=2^{\theta^+}$;
\item $\onto(\{\theta\},J^{\bd}[\kappa],\theta)$ holds;
\item $\ubd(J^{\bd}[\kappa],\theta^+)$ fails.
\end{itemize}
\end{thm}
\begin{proof} As in the proof of Theorem~\ref{thm97}, we let $\mathbb P$ be Cohen's notion of forcing for adding $\kappa$ many subsets of $\theta$,
so that $\kappa=2^\theta$ and $\onto(\{\theta\},J^{\bd}[\kappa],\theta)$ holds in the extension.
Let us see now why $\ubd(J^\bd[\kappa],\theta^+)$ fails. 

Since we were using a $\theta^+$-cc forcing, given a colouring $c:[\kappa]^2\rightarrow\theta^+$ in the extension,
we may find $d:[\kappa]^2\rightarrow\theta^+$ in the ground-model such that, for every $\alpha<\beta<\kappa$, $c(\alpha,\beta)\leq  d(\alpha,\beta)$. 
Appealing to Proposition~\ref{reducetotwo} we find a cofinal set $B \s \kappa$ such that for every $\eta< \kappa$, $|d[\{\eta\}\circledast B]| \leq 2$. So, for every $\eta<\kappa$, 
$$\sup(c[\{\eta\}\circledast B])\le\sup(d[\{\eta\}\circledast B]) < \theta^+.$$
Here the last inequality follows as we take the supremum over a two element set.
\end{proof}

It follows from Proposition~\ref{prop46} that, in $\zfc$, $\onto^{++}(J^{\bd}[\kappa],\theta)$ holds for many pairs $\theta<\kappa$ of infinite regular cardinals.
The next proposition rules out the case $\theta=\kappa$.
\begin{prop}\label{prop912} Suppose that $J$ is a nontrivial ideal over $\kappa$. Then:
\begin{enumerate}[(1)]
\item $\ubd^{++}(J,\kappa)$ fails;
\item If $[\kappa]^{<\kappa}\s J$, then there is no colouring $c:[\kappa]^2\rightarrow\kappa$ with the property that,
for every sequence $\langle B_\tau\mid \tau<\kappa\rangle$ of elements of $J^+$,
there is an $\eta<\kappa$ and an injection $h:\kappa \rightarrow \kappa$ such that, 
$$\{ \tau<\kappa\mid \{ \beta\in B_\tau\mid c(\eta,\beta)=h(\tau)\}\in J^+\}\in J^*.$$
\end{enumerate}
\end{prop}
\begin{proof} 
Due to constraints of space, we settle for proving Clause~(2).

Towards a contradiction, 
suppose that any $J^+$-set is of size $\kappa$,
and yet some colouring $c:[\kappa]^2\rightarrow\kappa$ does satisfy the conclusion.
In particular, by invoking it with the constant $\kappa$-sequence in which all the $B_\tau$ sets are equal to $\kappa$, we conclude that
there exists an $\eta^*<\kappa$ for which $|\{ \tau<\kappa \mid \{\beta<\kappa\mid c(\eta^*,\beta)=\tau\}\in J^+\}|=\kappa$.
It follows that we may fix a surjection $f:\kappa\rightarrow\kappa$ such that the preimage of any singleton is in $J^+$.

Next, for every $\eta<\kappa$,
if there exists $\tau<\kappa$ for which $\{\beta<\kappa\mid c(\eta,\beta)=\tau\}$ is in $J^+$,
then let $\xi_\eta$ be the least such $\tau$, and then let $B^\eta:=\{\beta<\kappa\mid c(\eta,\beta)=\xi_\eta\}$.
Otherwise, let $\xi_\eta:=0$ and $B^\eta:=\kappa$.
Finally, for every $\tau<\kappa$, let $B_\tau:=B^{f(\tau)}$.

Towards a contradiction, suppose that there exists $\eta<\kappa$ and an injection $h:\kappa \rightarrow \kappa$ for which 
the following set is in $J^*$:
$$T:=\{ \tau<\kappa\mid \{ \beta\in B_\tau\mid c(\eta,\beta)=h(\tau)\}\in J^+\}.$$		
In particular, $B^\eta\in J^+\setminus\{\kappa\}$.
Now, recalling the choice of $f$, we may find $\tau_0\neq\tau_1$ in $T$ such that $f(\tau_0)=\eta=f(\tau_1)$. For each $i<2$:
$$\{\beta\in B^\eta\mid c(\eta,\beta)=h(\tau_i)\}=\{\beta\in B_{\tau_i}\mid c(\eta,\beta)=h(\tau_i)\}\in J^+.$$
It follows that $h(\tau_0)=\xi_\eta=h(\tau_1)$,
contradicting the fact that $h$ is injective.
\end{proof}
\begin{remark} By weakening $J^*$ to $J^+$ at the end of Clause~(2) we arrive at a principle which is consistent
even with familiar  ideals like $J=\ns_\kappa$. See Lemma~\ref{ubdplus} and Corollary~\ref{hajnaltocolouring2}.
\end{remark}

\section{Weakly compact cardinals}\label{sectionweaklycompact}
In this section we characterise weakly compact cardinals using our principles. 
This is done in three ways: in Corollary~\ref{cor93}, in Corollary~\ref{wcubdtoonto} (for $\kappa \geq 2^{\aleph_0}$), and in Corollary~\ref{weaklycompactthreepieces}. 
An additional characterisation in $\L$ is given in Corollary~\ref{wcvl}.

\begin{lemma}\label{treelemma}Suppose $\kappa=\kappa^{\aleph_0}$.
For every colouring $c:[\kappa]^2\rightarrow2$,
there exists a corresponding colouring $d:[\kappa]^2\rightarrow\omega$ satisfying the following.
For every subnormal $J\in\mathcal J^\kappa_\omega$,
if $c``[B]^2=2$ for every $B\in J^+$, then $d$ witnesses $\onto^+(J,\omega)$.
\end{lemma}
\begin{proof} Since $\kappa^{\aleph_0}=\kappa$, we may fix  an enumeration $\langle x_\eta \mid \eta< \kappa\rangle$ of ${}^\omega \kappa$. 
Given a colouring $c:[\kappa]^2\rightarrow2$,
derive a colouring $d: [\kappa]^2 \rightarrow \omega$ by letting $d(\eta, \beta)$ be the least $n$ such that $c(x_\eta(n), \beta) =1$ if such an $n$ exists, and if not, $d(\eta, \beta) := 0$. 

Next, suppose that we are given a subnormal $J\in\mathcal J^\kappa_\omega$ such that $c``[B]^2=2$ for every $B\in J^+$. 
We shall show that the $d$ derived from it in the above fashion witnesses $\onto^+(J,\omega)$.
To this end, fix an arbitrary $B\in J^+$.
For $n<\omega$, $x:n\rightarrow\kappa$ and $y:n\rightarrow 2$ denote:
$$B_{x,y}:=\{\beta\in B\mid \forall i<n( c(x(i),\beta)=y(i))\},$$
and also
$$\tree(B):=\{ (x, y) \in{}^n\kappa \times {}^n2\mid n<\omega \ \&\ B_{x,y}\in J^+\}.$$
Then it is clear that $\tree(B)$ is a subset of ${}^{<\omega}\kappa\times {}^{<\omega}2$ consisting of pairs of tuples of the same length and closed under initial segments. 
That is, it is a subtree of $\bigcup_{n<\omega}({}^{n}\kappa\times {}^{n}2)$ where the order is given by the end-extension relation. 
Since $B\in J^+$, it is clear that $(\emptyset,\emptyset)\in \tree(B)$, 
so in particular that $\tree(B)$ is nonempty.

\begin{claim}\label{wcontoclaim}
Let $n<\omega$ and let $x\in{}^n2$ and $y \in{}^n2$ be such that $(x,y) \in \tree(B)$.
Then there exists $\eta<\kappa$ such that, for all $i<2$, $(x{}^\smallfrown\langle\eta\rangle, y{}^\smallfrown\langle i\rangle)\in \tree(B)$.
\end{claim}
\begin{why} In fact our proof will show that there is such an $\eta$ in $B_{x, y}$.
So suppose our claim does not hold. This means that for every $\eta<\kappa$ we can pick an $i_\eta <2$ along with an $E_\eta\in J^*$
such that $$(B_{x{}^\smallfrown\langle\eta\rangle, y{}^\smallfrown\langle i_\eta\rangle})\cap E_\eta=\emptyset.$$

As $J$ is subnormal, we may now find $B' \s B_{x,y}$ in $J^+$ such that,
for every $(\eta,\beta)\in [B']^2$, 	$\beta\in E_\eta$.
As $J$ is an ideal, we may also fix an $i^*< 2$ for which $B'':=\{\eta\in B'\mid i_{\eta} = i^*\}$ is in $J^+$.
By the hypothesis on $J$, we may find a pair $(\eta,\beta)\in[B'']^2$ such that $c(\eta,\beta) =i^*$. 
As $(\eta,\beta)\in [B'']^2$, 	$\beta\in E_\eta$,
so that $\beta\in B_{x,y}\setminus B_{x{}^\smallfrown\langle\eta\rangle, y{}^\smallfrown\langle i_\eta\rangle}$.
So $c(\eta,\beta)\neq i_\eta$, contradicting the fact that $i_\eta=i^*$.
\end{why}
For $n< \omega$ let $y_n :n \rightarrow \{0\}$ denote the constant function with value $0$. 
Now using the claim we can recursively find an  $x\in{}^\omega\kappa$ such that, for every $n<\omega$, 
$(x\restriction (n+1),y_n{}^\smallfrown\langle1\rangle)\in \tree(B)$. 
Note that since $\tree(B)$ is closed under initial segments this implies that for every $n<\omega$ we also have that $(x\restriction n,y_n)\in \tree(B)$. 
We shall call such an $x\in{}^\omega\kappa$ a \emph{canonical branch for $B$}. Let us see the details of its construction.

Set $(x_0,y_0):=(\emptyset,\emptyset)$ which we know is in $\tree(B)$. 
Suppose that $n<\omega$ and we have already found an $x_n \in{}^n\kappa $ such that $(x_n, y_n) \in\tree(B)$, so in particular such that $B_{x_n, y_n}\in J^+$.
By the preceding claim, we now find $\eta<\kappa$ such that $B_{x_n{}^\smallfrown\langle\eta\rangle, y_n{}^\smallfrown\langle0\rangle}$
and $B_{x_n{}^\smallfrown\langle\eta\rangle, y_n{}^\smallfrown\langle1\rangle}$ are both in $J^+$.
Then, let $x_{n+1}:=x_n{}^\smallfrown\langle\eta\rangle$. In this way we can find a canonical branch by discovering its initial segments. 

Now let $x\in{}^\omega\kappa$ be a canonical branch for $B$. Pick $\eta < \kappa$ such that $x = x_\eta$. 
To see that $d[\{\eta\}\circledast B] = \omega$, let $n< \omega$ be arbitrary.
Since $x_\eta$ is a canonical branch for $B$, the set $B_{{x_\eta \restriction(n+1), y_n{}^\smallfrown\langle1\rangle}}$ is in $J^+$.
For every $\beta$ in this set, $c(x_\eta(n) , \beta) =1$ and for every $i< n$ it is the case that $d(x_\eta(i) , \beta) =0$. 
So $\{\beta\in B\setminus(\eta+1)\mid d(\eta,\beta) = n\}$
covers the positive set $B_{{x_\eta \restriction (n+1), y_n{}^\smallfrown\langle1\rangle}}\setminus(\eta+1)$.
\end{proof}

\begin{thm}\label{wconto} Suppose $\kappa=\kappa^{\aleph_0}$ is a regular cardinal which is not weakly compact. 
Then $\onto^+(J^\bd[\kappa],\omega)$ holds.
\end{thm}
\begin{proof} Since $\kappa$ is not weakly compact, we may fix a colouring $c:[\kappa]^2\rightarrow 2$ such that $c``[B]^2=2$ for every $B\in[\kappa]^\kappa$.
Then, by Lemma~\ref{treelemma}, there exists a corresponding colouring $d$ witnessing $\onto^+(J^{\bd}[\kappa],\omega)$.
\end{proof}

\begin{cor}\label{cor93} The following are equivalent:
\begin{enumerate}[(1)]
\item $\kappa$ is not weakly compact;
\item $\ubd(J^\bd[\kappa],\omega)$ holds.
\end{enumerate}
\end{cor}
\begin{proof} The implication $(2)\implies(1)$ follows from Proposition~\ref{reducetotwo},
so we assume that $\kappa$ is not weakly compact, and prove that $\ubd(J^\bd[\kappa],\omega)$ holds:
\begin{itemize}	
\item[$\br$] If $\kappa=\omega$, then this follows from Proposition~\ref{prop50}.

\item[$\br$] If $\kappa$ is a singular cardinal, then this follows from Theorem~\ref{prop51}, Clauses (1) and (2).

\item[$\br$] If $\kappa$ is regular and $\kappa^{\aleph_0}=\kappa$, then by Theorem~\ref{wconto} we even have that $\onto^+(J^{\bd}[\kappa],\omega)$ holds.

\item[$\br$] Otherwise, this follows from Proposition~\ref{prop52}(iii).\qedhere
\end{itemize}
\end{proof} 

\begin{cor}\label{wcubdtoonto}For every cardinal $\kappa\ge  2^{\aleph_0}$, the following are equivalent:
\begin{enumerate}[(1)]
\item $\kappa$ is not weakly compact;
\item $\onto(J^\bd[\kappa],\omega)$ holds.
\end{enumerate}
\end{cor}
\begin{proof}
Galvin and Shelah \cite{MR0329900} proved that $\cf(2^{\aleph_0})\nrightarrow[\cf(2^{\aleph_0})]^2_{\aleph_0}$ holds. So, appealing to Propositions \ref{singularprojection} and \ref{prop45},
this settles the case $\kappa = 2^{\aleph_0}$.
As for $\kappa>2^{\aleph_0}$, this follows from Corollaries \ref{cor93} and \ref{abovecont}.
\end{proof}

\begin{remark}\label{remark105} 
\begin{enumerate}
\item The hypothesis of the preceding cannot be waived:
by \cite{paper53}, $\onto(J^\bd[\aleph_\omega], \omega)$ fails if $\mathfrak t>\aleph_\omega$.
\item In contrast to the fact that $\onto(J^{\bd}[2^{\aleph_0}],\aleph_0)$ holds, 
$\onto(J^{\bd}[2^{\aleph_1}],\aleph_1)$ and even $\ubd(J^{\bd}[2^{\aleph_1}],\aleph_1)$ may consistently fail.
To see this, appeal to Theorem~\ref{almostineffablecohen} with $\theta:=\aleph_0$. 
\end{enumerate}
\end{remark}

\begin{cor}\label{ontojbdomegafailure}
Suppose that $\kappa$ is a non-weakly-compact uncountable cardinal such that $\onto(J^\bd[\kappa], \omega)$ fails.
Then $\kappa<2^{\aleph_0}$ and one of the following must be true:
\begin{enumerate}[(1)]
\item $\kappa$ is a singular cardinal of countable cofinality;
\item $\kappa$ is a greatly Mahlo cardinal which is weakly compact in $\L$.
\end{enumerate}
\end{cor}
\begin{proof} By Corollary~\ref{wcubdtoonto}, if $\onto(J^\bd[\kappa], \omega)$ fails, then $\kappa<2^{\aleph_0}$. 
\begin{enumerate}[(1)]
\item By Theorem~\ref{prop51}(3), $\onto(J^\bd[\kappa], \omega)$ holds for every singular cardinal $\kappa$ of uncountable cofinality.
\item If $\kappa$ is a regular uncountable cardinal, then by Proposition~\ref{prop45}(3),
$\kappa\nrightarrow[\kappa]^2_\omega$ fails.
Then by \cite[Theorem~8.1.11]{TodWalks}, $\kappa$ does not carry a nontrivial $C$-sequence.
So by Remark~\ref{trivialamenable}, $\kappa\notin\sa_\kappa$,
which, by Lemma~\ref{lemma43}, means that $\ubd(J^{\bd}[\kappa],\kappa)$ fails.
Then, by Corollary~\ref{strongamenmahlo}, $\kappa$ is greatly Mahlo,
and, by Corollary~\ref{square_is_amenable}, $\kappa$ is weakly compact in $L$.\qedhere
\end{enumerate}
\end{proof}

\begin{cor}\label{wcvl} Assuming $\V=\L$, for every regular uncountable cardinal 
$\kappa$, the following are equivalent:
\begin{enumerate}[(1)]
\item $\kappa$ is not weakly compact;
\item $\ubd^*(J^\bd[\kappa], \{T\})$ holds for some stationary $T\s\kappa$;
\item $\ubd([\kappa]^\kappa,J^\bd[\kappa], \kappa)$ holds;
\item $\onto([\kappa]^\kappa,J^\bd[\kappa], \kappa)$ holds;
\item $\kappa\in\sa_\kappa$.
\end{enumerate}
\end{cor} 
\begin{proof} As proved by Jensen  \cite{jensen}, assuming $\V=\L$, 
a regular uncountable cardinal $\kappa$ admits a nonreflecting stationary set iff it is not weakly compact. 
So $(1)\iff(2)$ by Corollary~\ref{hajnaltocolouring2}. By another theorem of Jensen \cite{jensen}, assuming $\V=\L$, 
$\diamondsuit$ holds over any stationary subset of any regular uncountable cardinal.
Putting this together with the previous fact, we get $(1)\implies(4)$ from Corollary~\ref{cor811}.
It is clear that $(4)\implies(3)$, and $(3)\iff(5)$ is given by Lemma~\ref{lemma43}. 
Finally, the implication $(5)\implies(1)$ is given by Proposition~\ref{thm215}.
\end{proof}

By Remark~\ref{almosteffabletwopieces} below,
$\onto(J^{\bd}[\kappa],2)$ holds iff $\kappa$ is not almost ineffable.
Moving from $2$ to $3$, we arrive at the following.

\begin{cor} \label{weaklycompactthreepieces}The following are equivalent for every uncountable cardinal $\kappa$:
\begin{enumerate}[(1)]
\item $\onto(J^\bd[\kappa], 3)$ holds;
\item $\kappa$ is not weakly compact.
\end{enumerate}
\end{cor}
\begin{proof} $(1) \implies (2)$ This follows from Proposition~\ref{reducetotwo}.

$(2) \implies (1)$ In case $\kappa \geq 2^{\aleph_0}$ this follows from Corollary~\ref{wcubdtoonto}. 
In case $\aleph_0<\kappa \le2^{\aleph_0}$ this follows from Proposition~\ref{ontofinite}.
\end{proof}
When put together with Proposition~\ref{reducetotwo}, the preceding shows that for every infinite cardinal $\kappa$:
$$\onto(J^\bd[\kappa], 3) \iff \kappa\nrightarrow[\kappa]^2_2\iff \bigwedge_{n=1}^\infty\onto(J^\bd[\kappa], n).$$

\section{Ineffable cardinals}\label{sectionineffable}
In this section, $\kappa$ denotes a regular uncountable cardinal. 
We characterise here stationary sets $S \s \kappa$ which are not ineffable using our principles. 
This is done in three ways: in Lemma~\ref{effabletwopieces}, in Corollary~\ref{cor115}, and in Corollary~\ref{cor117} (for $\kappa \geq 2^{\aleph_0}$). 
We finish by proving some related results related to consistency results, both positive and negative, and with a reminder of Conjecture~\ref{conj412}.

\begin{lemma} \label{effabletwopieces} For every stationary $ S\s\kappa$, the following are equivalent:
\begin{enumerate}[(1)] 
\item $\onto(\ns_\kappa\restriction S,2)$ holds;
\item $ S$ is non-ineffable;
\item there exists a sequence of functions $\langle g_\beta:\beta\rightarrow2\mid \beta\in S\rangle$
with the property that, for every $g:\kappa\rightarrow2$, 
the following set is nonstationary:
$$\{ \beta\in S\mid \forall \eta<\beta(g(\eta)\neq g_\beta(\eta))\}.$$
\end{enumerate}
\end{lemma}
\begin{proof} $(1)\implies(2)$ Given a colouring $c:[\kappa]^2\rightarrow2$ and an ineffable set $ S$, 
we can consider $c(\cdot,\beta) : \beta \rightarrow 2$ as the indicator function of a subset of $\beta$. Appealing then to the ineffability of $ S$, we see that
there must exist a function $f:\kappa \rightarrow 2$ such that the set $B:=\{\beta \in S\mid c(\cdot,\beta)= f\restriction \beta\}$ is stationary.
So, for every $\eta<\kappa$, $c[\{\eta\}\times (B\setminus(\eta+1))]=\{f(\eta)\}$. So $c$ cannot be a witness to $\onto(\ns_\kappa\restriction S, 2)$.

$(2)\implies(3)$ Assuming $ S$ is non-ineffable, we may find a sequence $\langle A_\beta\mid \beta\in S\rangle$ such that,
for every $\beta\in S$, $A_\beta\s \beta$ and such that, for every $A\s\kappa$,
the set $\{\beta\in S\mid A_\beta=A\cap\beta\}$ is nonstationary.
For each $\beta\in S$, let $g_\beta:\beta\rightarrow2$ denote the characteristic function of $A_\beta$.
Now, given $g:\kappa\rightarrow2$, let $A:=\{\eta<\kappa\mid g(\eta)=0\}$.
Fix a club $C\s\kappa$ such that $C\cap S\s \{\beta\in S\mid A_\beta\neq A\cap\beta\}$.
Let $\beta\in C\cap S$. Now, pick $\eta\in(A\cap\beta)\symdiff A_\beta$.

$\br$ If $\eta\in A\cap\beta$, then $\eta\notin A_\beta$, so that $g(\eta)=0=g_\beta(\eta)$.

$\br$ If $\eta\in A_\beta$, then $\eta\notin A\cap\beta$, so that $g(\eta)=1=g_\beta(\eta)$.

$(3)\implies(1)$ Pick a colouring $c:[\kappa]^2\rightarrow 2$ that satisfies  $c(\eta,\beta)=g_\beta(\eta)$ for all $\eta<\beta<\kappa$ with $\beta\in S$.
Towards a contradiction, suppose that $B\s S$ is a stationary set such that, for every $\eta<\kappa$, $c[\{\eta\}\circledast B]\neq 2$.
Define $g:\kappa\rightarrow 2$ via $g(\eta):=\min(2\setminus c[\{\eta\}\circledast B])$. Now, pick $\beta\in B$ and $\eta<\beta$ such that $g(\eta)=g_\beta(\eta)$.
Then $c(\eta,\beta)=g_\beta(\eta)=g(\eta)$, contradicting the fact that $g(\eta)\notin c[\{\eta\}\circledast B]$.
\end{proof}

\begin{remark}\label{almosteffabletwopieces} A similar argument shows that $\onto(J^{\bd}[\kappa],2)$ holds iff $\kappa$ is not almost ineffable.
\end{remark}

In contrast, $\onto^-(S,2)$ is not refuted by $ S$ being ineffable.
\begin{defn}[\cite{jensen1969some}] A subset $S \s \kappa$ is called a \emph{subtle} subset of $\kappa$
if for every sequence $\langle A_\beta\mid \beta \in  S\rangle$ and $D \s \kappa$ a club, there are $\alpha< \beta$ in $D\cap S$ such that $A_\alpha\cap\alpha = A_\beta \cap \alpha$.
\end{defn}
It is easily seen that an inffable set is subtle. In fact, any almost ineffable set is subtle. The following fact by Kunen together with Lemma~\ref{lemma47}(2) shows that $\onto^-(S,\kappa)$ holds for any ineffable set $ S$. 
\begin{fact}[\cite{jensen1969some}] For every subtle $S\s \kappa$, $\diamondsuit(S)$ holds. 
\end{fact}

\begin{thm} \label{ineffonto} Suppose $\kappa=\kappa^{\aleph_0}$ and $ S\s\kappa$ is a stationary set that is not ineffable.
Then $\onto^+(\ns_\kappa\restriction S,\omega)$ holds. 
\end{thm}
\begin{proof} By the choice of $S$, we may fix a colouring $c:[\kappa]^2\rightarrow 2$ such that $c``[B]^2=2$ for every stationary $B\s S$.
Then, by Lemma~\ref{treelemma}, there exists a corresponding colouring $d$ witnessing $\onto^+(\ns_\kappa\restriction S,\omega)$.
\end{proof}

\begin{cor}\label{cor115} For every regular uncountable cardinal $\kappa$, and every stationary $ S\s\kappa$, the following are equivalent:
\begin{enumerate}[(1)] 
\item $ S$ is not ineffable;
\item $\ubd(\ns_\kappa\restriction S,\omega)$ holds.
\end{enumerate}
\end{cor}
\begin{proof} $(1)\implies(2)$ Suppose that $ S\s\kappa$ is a non-ineffable set.

$\br$ If $\kappa^{\aleph_0}=\kappa$, then by Theorem~\ref{ineffonto} we even have that $\onto^+(\ns_\kappa\restriction S,\omega)$ holds.

$\br$ If $\kappa^{\aleph_0}>\kappa$,
then $\kappa$ is not weakly compact, and so by Corollary~\ref{cor93} we even have that $\ubd(J^\bd[\kappa],\omega)$ holds.

$(2)\implies(1)$ Suppose that $ S$ is ineffable. Let $c:[\kappa]^2\rightarrow\omega$ be an arbitrary colouring. 
Fix a bijection $\pi:\kappa\times\omega\leftrightarrow\kappa$,
and consider the club $D:=\{\delta<\kappa\mid \pi[\delta\times\omega]=\delta\}$.
By the ineffability of $S$ there is a set $A\s\kappa$ with the property that $B:=\{\beta \in S\cap D \mid \pi[c(\cdot, \beta)]\cap\beta=A\cap\beta\}$ is stationary.
Consequently, $\{ c(\cdot,\beta)\mid \beta\in B\}$ is a chain converging to some function $f:\kappa \rightarrow \omega$,
and hence $c[\{\eta\}\times (B\setminus(\eta+1))]=\{f(\eta)\}$ for every $\eta<\kappa$. So $c$ cannot be a witness to $\ubd(\ns_\kappa\restriction S, \omega)$. 
\end{proof}

\begin{cor}\label{cor117} For every regular cardinal $\kappa \ge 2^{\aleph_0}$, and every $ S\s\kappa$, the following are equivalent:
\begin{enumerate}[(1)]
\item $ S$ is not ineffable;
\item $\onto(\ns_\kappa\restriction S,\omega)$ holds.
\end{enumerate}
\end{cor}
\begin{proof} 
$\br$ For $\kappa>2^{\aleph_0}$, this follows from Theorem~\ref{ubdtoonto}(2).

$\br$ For $\kappa = 2^{\aleph_0}$, by Corollary~\ref{wcubdtoonto} we even have that $\onto(J^\bd[\kappa], \omega)$ holds.
\end{proof}
The preceding is optimal.

\begin{prop} \label{ineffablecohen} Assuming the consistency of an ineffable cardinal,
it is consistent that $\kappa= 2^{\aleph_0}=2^{\aleph_1}$ is a non-ineffable cardinal, 
and hence $\onto(J^\bd[\kappa],\aleph_0)$ holds, yet $\ubd(\ns_\kappa,\aleph_1)$ fails.
\end{prop}
\begin{proof} By invoking Theorem~\ref{thm97} with $\theta:=\aleph_0$.
\end{proof}

\begin{cor}\label{ontonsomegafailure} Suppose that $\kappa$ is a non-ineffable regular uncountable cardinal such that $\onto(\ns_\kappa, \omega)$ fails. 
Then $\kappa < 2^{\aleph_0}$ and $\kappa$ is weakly compact in $\L$.
\end{cor}
\begin{proof} By Corollary~\ref{cor117} we know that $\kappa < 2^{\aleph_0}$. 
By Theorem~\ref{thm42}(2), it must be the case that $\ubd(\ns_\kappa,\kappa)$ fails.
In particular, $\ubd(J^{\bd}[\kappa],\allowbreak\kappa)$ fails, 
which by Corollary~\ref{square_is_amenable} implies that $\kappa$ is weakly compact in $\L$.
\end{proof}
The above would be improved in case Conjecture~\ref{conj412} has a positive answer.
	
\section*{Acknowledgments}

We are grateful to the anonymous referee for a thorough reading of this paper and for providing a long list of significant corrections and suggestions.

The first author is supported by the Israel Science Foundation (grant agreement 2066/18).
The second author is partially supported by the European Research Council (grant agreement ERC-2018-StG 802756) and by the Israel Science Foundation (grant agreement 2066/18). 

Some of the results of this paper were announced by the first author at the \emph{Israel Mathematical Union Annual Meeting} special session in set theory and logic in July 2021. 
He would like to thank the organisers for this opportunity.

\newpage

\section*{Appendix: Diagram of Implications}\label{sectiondiagrams}
In the diagram below, $\kappa$ is a regular uncountable cardinal and all unidirectional arrows are irreversible and we have indicated next to each arrow where to find the relevant proofs except for most of the arrows from left to right which follow by the definitions. The diagram shows how the well-known analogy that weakly compact cardinals are to $J^\bd[\kappa]$ as ineffable cardinals are to $\ns_\kappa$ manifests itself in our principles as well as the importance of Conjecture~\ref{conj412} 
which is one of two points of asymmetry. The other is the fact that $\onto(J^\bd[\kappa],2)$ characterises not the weakly compact cardinals, but rather the almost ineffable cardinals by Remark~\ref{almosteffabletwopieces}.
	\begin{center}
	\begin{figure}[H]
		\begin{tikzcd}[every arrow/.append style={-latex}, sep=huge]
			\kappa \text{ is not weakly compact in }\L
			\arrow{d}{\text{Cor. }\ref{square_is_amenable}, \text{ Rem. }\ref{forcingsquare}} 
			\arrow{r}
			&
			\kappa \text{ is not ineffable in }\L
			\arrow[dotted]{d}{\text{Conj. }\ref{conj412},\text{ Rem. }\ref{forcingsquare}}
			\\
			\kappa \in \sa_\kappa
			\arrow{d}{\text{Lem. }\ref{lemma43}}
			\arrow{r}{\text{Lem. }\ref{amenableideal}} 
			&
			\kappa \in \amen_\kappa
			\arrow{d}{\text{Lem. }\ref{lemma44}} 
			\\
			\ubd(J^\bd[\kappa], \kappa) 
			\arrow{d}{\text{Cor. }\ref{cor65}, \text{ Thm. }\ref{kunenwc}}
			\arrow{u}\arrow{r}  
			&
			\ubd(\ns_\kappa, \kappa)
			\arrow{d}{\text{Cor. }\ref{cor65}, \text{ Thm. }\ref{kunenineff}}
			\arrow{u}
			\\
			\forall \theta< \kappa~ \ubd(J^\bd[\kappa], \theta)
			\arrow{d}{\text{Obvious, Prop. }\ref{almostineffablecohen}}
			\arrow{r} 
			&
			\forall \theta< \kappa~\ubd(\ns_\kappa, \theta)
			\arrow{d}{\text{Obvious, Prop. }\ref{ineffablecohen}}
			\\
			\ubd(J^\bd[\kappa], \aleph_0)
			\arrow{d}{\text{Cor. }\ref{cor93}}
			\arrow{r}  
			&
			\ubd(\ns_\kappa, \aleph_0)
			\arrow{d}{\text{Cor. }\ref{cor115}}
			\\
			\kappa\text{ is not weakly compact }
			\arrow{u}\arrow{d}{\text{Cor. }\ref{weaklycompactthreepieces}}
			\arrow{r}  
			&
			\kappa\text{ is not ineffable}
			\arrow{u}\arrow{d}{\text{Lem. }\ref{effabletwopieces}}
			\\
			\onto(J^\bd[\kappa],3)\arrow{r}\arrow{u}
			&
			\onto(\ns_\kappa, 2)\arrow{u}
		\end{tikzcd}
		\end{figure}
	\end{center}
\end{document}